\documentclass{scrartcl}

\usepackage[utf8]{inputenc}
\usepackage[T1]{fontenc}

\usepackage{graphicx}
\usepackage{amsmath,amssymb,amsthm}
\usepackage{esint} 
\usepackage{enumitem,color}
\usepackage{comment}
\usepackage{hyperref}
\usepackage{diagbox}
\usepackage{mathtools}
\mathtoolsset{showonlyrefs}

\usepackage[style=alphabetic,maxalphanames=4,maxnames=4]{biblatex}
\AtEveryBibitem{\clearfield{issn}}
\newtheorem{thm}{Theorem}[section]
\newtheorem{prop}[thm]{Proposition}
\newtheorem{lem}[thm]{Lemma}
\newtheorem{cor}[thm]{Corollary}
\theoremstyle{definition}

\theoremstyle{remark}
\newtheorem{rem}[thm]{Remark}

\newcommand{\fig}[2]{\includegraphics[width=#1\textwidth]{#2}}

\newcommand{\Z}{\mathbb{Z}}
\newcommand{\R}{\mathbb{R}}

\newcommand{\CCm}{\mathcal{C}_{\mathrm{min}}}

\newcommand{\bip}[1]{\left.#1\right\vert_{\mathrm{bip}}}
\newcommand{\bt}{\mathbf{t}}
\newcommand{\bZ}{\mathbf{Z}}
\newcommand{\bU}{\mathbf{U}}

\newcommand{\bx}{\mathbf{x}}
\newcommand{\by}{\mathbf{y}}

\title{Enumeration of maps with tight boundaries and the Zhukovsky transformation}%
\author{J\'er\'emie Bouttier%
  \thanks{Sorbonne Université and Université Paris Cité, CNRS, IMJ-PRG, F-75005 Paris, France}
  \thanks{Université Paris-Saclay, CNRS, CEA, Institut de physique
    théorique, 91191, Gif-sur-Yvette, France} \and %
  Emmanuel Guitter\footnotemark[2] \and %
  Grégory Miermont\thanks{ENS de Lyon, UMPA, CNRS UMR 5669, 46 allée
    d’Italie, 69364 Lyon Cedex 07, France, and Institut Universitaire de France}}%
\date{\today}

\begin{document}
\maketitle

\begin{abstract}
  We consider maps with tight boundaries, i.e.\ maps whose boundaries
  have minimal length in their homotopy class, and discuss the
  properties of their generating functions
  $T^{(g)}_{\ell_1,\ldots,\ell_n}$ for fixed genus $g$ and prescribed
  boundary lengths $\ell_1,\ldots,\ell_n$, with a control on the
  degrees of inner faces. We find that these series appear as
  coefficients in the expansion of $\omega^{(g)}_n(z_1,\ldots,z_n)$, a
  fundamental quantity in the Eynard-Orantin theory of topological
  recursion, thereby providing a combinatorial interpretation of the
  Zhukovsky transformation used in this context. This interpretation
  results from the so-called trumpet decomposition of maps with
  arbitrary boundaries. In the planar bipartite case, we obtain a
  fully explicit formula for $T^{(0)}_{2\ell_1,\ldots,2\ell_n}$ from
  the Collet-Fusy formula. We also find recursion relations satisfied
  by $T^{(g)}_{\ell_1,\ldots,\ell_n}$, which consist in adding an
  extra tight boundary, keeping the genus $g$ fixed. Building on a
  result of Norbury and Scott, we show that
  $T^{(g)}_{\ell_1,\ldots,\ell_n}$ is equal to a parity-dependent
  quasi-polynomial in $\ell_1^2,\ldots,\ell_n^2$ times a simple power
  of the basic generating function $R$. In passing, we provide a
  bijective derivation in the case $(g,n)=(0,3)$, generalizing a
  recent construction of ours to the non bipartite case.
\end{abstract}

\setcounter{tocdepth}{2}

\newpage
\tableofcontents
\newpage

\section{Introduction}\label{sec:introduction-1}

\subsection{Context and motivations}
\label{sec:context}

We pursue our investigation, started in~\cite{triskell,polytightmaps},
of the enumerative properties of maps with tight boundaries. In this
episode, we add a new twist to the story, by making an unexpected
connection with the theory of topological
recursion~\cite{Eynard2007,Eynard2016}. In particular, we find a
combinatorial interpretation of the Zhukovsky transformation, which
was so far used as an analytical tool in this theory.

Let us first recall some context. The enumeration of maps (graphs
embedded in surfaces) is a venerable topic in combinatorics, initiated
by Tutte in his famous series of ``Census'' papers,
see~\cite{Tutte1963} and references therein. Accounts of further
developments may be found for instance
in~\cite{Goulden2004,Schaeffer15}. A fertile connection with random
matrices was made in~\cite{BIPZ}, and it is essentially its
ramifications, motivated by the study of 2D quantum gravity---see
e.g.~\cite{Witten1991,DiFrancesco1995,Ambjoern1997} for
overviews---that led to topological recursion.

In the paper~\cite{triskell}, we found a surprisingly simple formula
for the generating function of planar bipartite maps with three
\emph{tight} boundaries (as we will explain in more detail below, a
boundary is said tight if it has minimal length in its homotopy
class). The tightness property seems to play a role, as without it the
corresponding generating function becomes slightly more
involved~\cite{CoFu12}. One may therefore wonder whether a similar
phenomenon occurs for maps of other topologies (higher genus, more
boundaries).

This question was already explored in the limit case of \emph{tight
  maps}. Colloquially speaking, a tight map is a map in which every
face is seen as a boundary, and is forced to be tight. The counting of
tight maps was actually investigated first by
Norbury~\cite{Norbury2010}, as tight maps are nothing but fatgraphs
describing lattice points in the moduli space of curves. A remarkable
\emph{quasi-polynomiality phenomenon} occurs in this problem: if we
denote by $N_{g,n}(b_1,\ldots,b_n)$ the number of tight maps of genus
$g$ with $n$ boundaries of lengths $b_1,\ldots,b_n$, then Norbury
showed that it is a quasi-polynomial of degree $3g-3+n$ in
$b_1^2,\ldots,b_n^2$ depending on the parities, and we gave
in~\cite{polytightmaps} a bijective construction of these
quasi-polynomials in the planar case $g=0$.

In a further paper~\cite{NoSc13}, Norbury and Scott observed that the
quasi-polynomiality phenomenon can be related to a specific feature in
topological recursion, namely that the ``$x$'' meromorphic function
attached to the ``spectral curve'' takes the particular form
\begin{equation}
  x(z) := \alpha+\gamma \left( z + z^{-1} \right) .
 \label{eq:Zhu}
\end{equation}
This form
is nothing but the so-called Zhukovsky
transformation~\cite{Zhukovskii1910}. The case of tight maps
corresponds to taking the second ``$y$'' function equal to the
Zhukovsky variable $z$, but Norbury and Scott showed that the
quasi-polynomiality subsists when $y$ is modified.

In this paper, we relate the observation of Norbury and Scott 
to the enumeration of maps with tight boundaries. At the combinatorial
level, the Zhukovsky transformation $z\mapsto x(z)$ corresponds to a 
substitution: namely we find that it translates the natural idea of transforming a
map with tight boundaries into a map with arbitrary boundaries by gluing 
a ``trumpet'' onto each boundary (see Figure~\ref{fig:tubamirum} below for an illustration). 
By a trumpet, we mean a planar map with two boundaries,
one of them being tight while the other is arbitrary. We use the term by analogy
with \cite{Saad2019}, where a similar notion is introduced in the context of 
hyperbolic geometry. The coefficients $\alpha$ and $\gamma$ appearing in \eqref{eq:Zhu}
are related to the ``basic'' generating functions $R$ and $S$ defined in the next section
by $\alpha=S$ and $\gamma=R^{1/2}$. 

By this approach, we deduce the interesting property that the generating function
$T^{(g)}_{\ell_1,\ldots,\ell_n}$ of maps of genus $g$ and $n$ tight boundaries
of lengths $\ell_1,\ldots,\ell_n$ has the general form 
  \begin{equation}
    T^{(g)}_{\ell_1,\ldots,\ell_n} = \tau^{(g)}_{\ell_1,\ldots,\ell_n} \gamma^{\ell_1+\cdots+\ell_n}
  \end{equation}
with $\tau^{(g)}_{\ell_1,\ldots,\ell_n}$ a quasi-polynomial of degree $3g-3+n$ in $\ell_1^2,\ldots, \ell_n^2$ 
depending on the parities, and $\gamma=R^{1/2}$ as before. In particular, in the
planar bipartite case, we will derive a fully explicit expression for $T^{(0)}_{2\ell_1,\ldots,2\ell_n}$
from the Collet-Fusy formula \cite{CoFu12}.

Let us conclude this section by mentioning the recent paper \cite{Budd2023} which introduces the notion of tight boundaries in the context
of hyperbolic geometry, and makes the connection with the so-called JT gravity \cite{Saad2019}. A very similar
polynomiality phenomenon occurs therein, which makes us believe that we are looking at the same 
reality from different angles.

\subsection{Basic definitions}
\label{sec:defs}
Let us now introduce precisely the main definitions and conventions used in this paper.

\paragraph{Combinatorial notions: maps with boundaries, and their
automorphisms.}

For $g$ a nonnegative integer, a \emph{map} of genus $g$ is a cellular embedding 
of a finite connected multigraph into the closed orientable surface of genus $g$, considered up
to orientation-preserving homeomorphism. By cellular embedding, we mean that the graph (which consists 
of \emph{vertices} and \emph{edges}) is drawn on the surface
without edge crossings, and that the connected components of the complement of 
the graph, called \emph{faces}, are homeomorphic to open disks. A map of
genus zero is said \emph{planar}.

A \emph{map with boundaries} is a map in which we mark some faces and vertices
and call them \emph{boundaries}. We shall use the terms 
\emph{boundary-face} and \emph{boundary-vertex} when we want to specify the type 
of a boundary. A face which is not a boundary is called an \emph{inner} face.
The \emph{length} of a boundary is by definition equal to its degree (number of incident edge sides)
for a boundary-face, and to $0$ for a boundary-vertex. A boundary-face is said
\emph{rooted} if one of its incident corners is marked. A boundary-vertex cannot be rooted. 

An \emph{automorphism} of a map with boundaries is a graph isomorphism that also preserves the map structure, in the following sense.  
\begin{itemize}
\item
The images of the edges pointing away from a given vertex, listed cyclically in clockwise order, are the edges pointing away from the image vertex, also in cyclic clockwise order. In particular, the graph isomorphism also induces a bijection of the set of faces onto itself, respecting the incidence relations with vertices and edges. 
\item 
The marked elements, i.e.\ the (possibly rooted) boundary-faces and boundary-vertices, are preserved. 
\end{itemize}
The automorphisms of a given map with boundaries form a group, which is often reduced to the identity element. For instance, this holds for a rooted map of any genus,  or for a planar map possessing at least three distinguished boundaries. Contrary to the former situation, the latter is not true in higher genera: Figure \ref{fig:torushex} provides a counterexample in genus $1$.

\begin{figure}
  \centering
  \includegraphics[scale=.4]{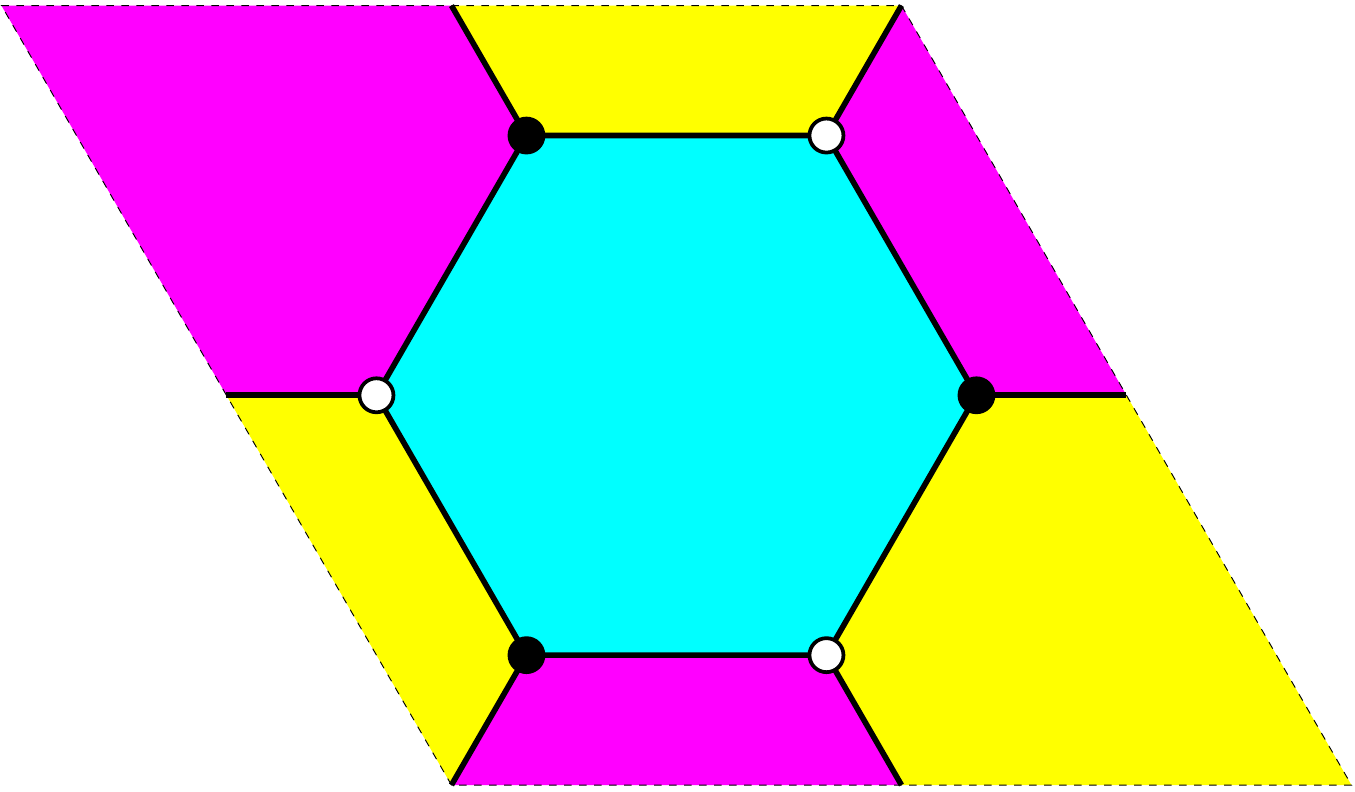}
  \caption{A toric map, with three distinguished boundary-faces of degree 6, admitting an automorphism group of order $3$. There are only $2$ rooted maps resulting from marking one of the $6$ corners incident to the cyan face, and we must account for this fact by weighing this map by the inverse automorphism group order factor $1/3$. }
  \label{fig:torushex}
\end{figure}

\paragraph{Topological notions: boundaries and homotopy.}

We think of boundaries as representing holes in the surface. Precisely, we create a puncture inside
each boundary-face, while we draw a small circle around each boundary-vertex: the interior of the circle
is removed but the edges incident to the vertex remain connected by the circle. See \cite[Figure 2]{triskell} for
an illustration.

Given a map, possibly with boundaries, a \emph{path} is a sequence of
consecutive edges. It defines a curve on the underlying surface
provided that, if the path visits a boundary-vertex, we choose a
direction for ``circumventing'' the corresponding puncture in the
surface. A path will always be assumed to contain this data in the
following. A path is said \emph{closed} if it starts and ends at the
same vertex. A path is said \emph{simple} if it does not visit the
same vertex twice (except at its endpoints for a simple
closed path). The \emph{contour} of a face is the closed path
formed by its incident edges. Two closed paths on the map are said
(freely) \emph{homotopic} to one another if their corresponding curves
can be continuously deformed into one another on the punctured
surface.  A closed path is said \emph{homotopic} to a boundary if, on
the underlying surface of the map, the curve associated with the path
can be contracted onto the puncture associated with the boundary.  A
boundary-face is said \emph{tight} (resp.\ \emph{strictly tight}) if
its degree is not larger (resp.\ is strictly smaller) than the number of
edges of any other closed path homotopic to it.
A boundary-vertex is considered as (strictly) tight by convention.

\paragraph{Generating functions.}

Let $t,t_1,t_2,t_3,\ldots$ be a collection of formal variables.
We attach a weight $t$ to each vertex which is not a boundary, and a weight $t_i$ to each 
inner face of degree $i$, $i=1,2,3,\ldots$. The global weight of a map is the product
of the weights of all its (non-boundary-) vertices and faces, multiplied by the inverse of the order of the automorphism group of the map. This last factor yields better combinatorial properties for the corresponding generating functions, see the caption of Figure \ref{fig:torushex} for a quick explanation. 

We now introduce two important quantities denoted by $R$ and $S$, which are formal
power series in $t,t_1,t_2,t_3,\ldots$.
The first one $R$ is the generating function of planar maps with two distinguished boundaries 
of length $1$. The second one $S$ is that of planar maps with two boundaries,
one of length $1$ and the other of length $0$. It is known\footnote{See for instance \cite{census} for a proof in the dual
language. Beware that $t$ is a weight per edge in this reference, one has to set $t_i=t\, g_i$ to match the notations.} that, as formal power series,
$R$ and $S$ are determined by the following equations:
\begin{equation}
\begin{split}
R &= t+\sum_{i\geq 1} t_i\ [z^{-1}]\left(z+S+\frac{R}{z}\right)^{i-1} ,\\
S &= \sum_{i\geq 1} t_i\ [z^{0}]\left(z+S+\frac{R}{z}\right)^{i-1} .\\
\end{split}
\label{eq:defRS}
\end{equation}
A simplification occurs in the \emph{essentially bipartite} case
$t_1=t_3=t_5=\cdots=0$, which corresponds to imposing that every inner
face has even degree.  Indeed we then have $S=0$ and $R$ satisfies the
single equation
\begin{equation}
  \label{eq:Rbipeq}
  R = t+ \sum_{j\geq 1}t_{2j} \binom{2j-1}{j} R^j .
\end{equation}
In this case, $R$ can also be understood as the generating function of
bipartite planar maps with two unrooted boundaries, one of length $2$ and the
other of
length $0$.

\subsection{Outline}

The remainder of this paper is organized as follows. Section~\ref{sec:fundec} is devoted
to the trumpet decomposition which relates maps with arbitrary boundaries to maps with tight 
boundaries. We first describe in Section~\ref{sec:trumpdec} the bijection for maps with boundary-faces only,
before analyzing its enumerative consequences in Section~\ref{sec:trumpenum}. The case of maps
having also boundary-vertices is discussed in Section~\ref{sec:adding-vert-bound}. In Section~\ref{sec:bipcase},
we derive from the Collet-Fusy formula an explicit expression for the generating function of planar bipartite maps with tight boundaries. Section~\ref{sec:recrel} is devoted to recursion relations satisfied by the 
generating functions of maps with tight boundaries. These recursion relations consist in adding an extra
tight boundary, preserving the genus. We first discuss in Section~\ref{sec:extrabvtx} the addition of 
a boundary-vertex, and then that of a boundary-face in Section~\ref{sec:extrabface}.
In Section~\ref{sec:consistency}, we check that the explicit expression for planar bipartite maps 
found in Section~\ref{sec:bipcase} satisfies these recursion relations. Section~\ref{sec:quasipol}
is devoted to the quasi-polynomiality phenomenon: we first consider maps with boundary-faces only
in Section~\ref{sec:toprec}, revisiting the result of Norbury and Scott. We then treat the case
of maps having also boundary-vertices in Section~\ref{sec:quasipolvia}: the proof is by induction
on the number of boundaries, using the recursion relations of Section~\ref{sec:recrel} for the induction
step and results from topological recursion for the initialization. For the sake of clarity, we
begin with the case of maps with even face degrees. The extension to the general case follows the same strategy
but requires extra technical steps which are offloaded to the appendices. Concluding remarks
are gathered in Section~\ref{sec:conclusion}. Additional material is contained in the 
appendices: Appendix~\ref{app:tightpants} generalizes the formula of \cite{triskell} for the generating
function of planar maps with three tight boundaries (i.e.~\ ``tight pair of pants'') to the non-bipartite case.
Appendices~\ref{sec:proof-prop-refpr-1} and \ref{sec:proof-prop-refpr} contain the proof of 
the induction step and the initialization step, respectively, for the quasi-polynomiality phenomenon.
Appendix~\ref{sec:higherdiff} gives a combinatorial expression for the derivatives of inverse functions
of several variables, used in Appendix~\ref{sec:proof-prop-refpr}.

\section{The trumpet decomposition}\label{sec:fundec}

In this section, we discuss a bijective decomposition of a
map with boundaries into a map with tight boundaries and a collection
of annular maps with one strictly tight boundary (a \emph{trumpet}). This generalizes
\cite[Proposition 6.5]{triskell}, which dealt with planar maps 
with three boundaries, to arbitrary genera and number of
boundaries. 

\subsection{The decomposition theorem}\label{sec:trumpdec}

Let us introduce some notation. A \emph{trumpet} is a map $M$ of genus
$0$ with two boundary-faces $f,F$, such that the boundary-face $F$ is
rooted, and such that the boundary-face $f$, called the
\emph{mouthpiece}, is unrooted and strictly tight. Recall from
Section~\ref{sec:defs} that this means that the contour of $f$ is the
\emph{unique} closed path of minimal length among all closed paths of
$M$ that are (freely) homotopic to $f$.  By considering a leftmost
geodesic from the root corner of the face $F$ aimed towards the face
$f$, we may canonically distinguish one corner incident to $f$. We
state the following simple but crucial fact.

\begin{lem}
  \label{sec:decomp-theor-1}
  The contour of the boundary of the mouthpiece of a trumpet is simple.
\end{lem}

This lemma is a consequence of the fact that the trumpet $M$ is drawn on a 
topological cylinder, and is characteristic 
of this topology: as can be seen for example in Figure 
\ref{fig:tubamirum}, tight faces need not have simple boundaries in 
maps of arbitrary topology (precisely, of negative Euler 
characteristic $\chi=2-2g-n$, where $g$ is the genus and $n$ is the number 
of boundaries).

\begin{proof}
We may assume
that the trumpet $M$ is embedded in the plane, with the puncture of the mouthpiece
$f$ set at point $0$, and the puncture of the rooted face $F$ sent to
infinity. We let $\mathbb{D}$ be the unit open disk  in $\R^2$ and
$\mathbb{S}^1=\partial \mathbb{D}$ be the unit circle. 

The contour of $f$ is a closed path, which we orient arbitrarily and
denote by $c$. Plainly, it admits a simple closed subpath $c'$ not
reduced to a single vertex, and we aim at showing that $c'=c$.  By the
Jordan-Schoenflies theorem, $c'$ cuts the plane into two domains $D$
(bounded) and $\R^2\setminus \overline{D}$ (unbounded), and moreover,
$c'$ viewed as a continuous injective mapping
$\mathbb{S}^1\to \mathrm{Im}(c')$ can be extended into a homeomorphism
$h:\R^2\to \R^2$ sending $\mathbb{D}$ to $D$. If $0$ does not belong
to $D$, it follows $c'$ is homotopic to the trivial path in the
punctured plane $\R^2\setminus \{0\}$, so that $c$ is homotopic to a
strictly shorter path, contradicting the tightness assumption of
$f$. Therefore, it must be that $0\in D$, so that $c'$ is homotopic to
$\mathbb{S}^1$ and therefore generates the fundamental group
$\pi_1(\R^2\setminus \{0\})\simeq \Z$. Since $c$, being the contour of
the face $f$ that contains $0$, is also a generator, this implies that
$c'$ is either homotopic to $c$ or to its reversed path. But the
tightness of $f$ shows that the length of $c'$ is at least that of
$c$, and the only possibility is that $c'$ is equal to $c$.
\end{proof}

Next, let $M_0$ be a map of genus $g$ with $n$ tight rooted boundary-faces
$f^0_1,\ldots,f^0_n$, and let
$M_i,1\leq i\leq n$ be a sequence of trumpets (a \emph{brassband}),
with respective boundary-faces
$f_i,F_i$, where $f_i$ denotes the mouthpiece. We assume
that the length of $f_i$ and the length of $f_i^0$ are equal for $1\leq i\leq n$. 

For every $i\in \{1,\ldots,n\}$, we orient the contour of the faces $f^0_i,f_i$ in such a way that $f^0_i$
lies to the right, and $f_i$ lies to the left of their oriented
contours. Then, since the boundary of $f_i$ is a simple closed path, and
since $f_i,f_i^0$ have the same length, one may identify these boundaries, in such a way that the root of
$f_i^0$ and the distinguished corner of $f_i$ (obtained from the root of $F_i$ as explained above) are matched, by gluing
edges sequentially as they appear in contour order,
for our choice of orientation. This results
in a map $M=\Phi(M_0,M_1,\ldots,M_n)$ with genus $g$ and with the $n$
rooted boundary faces $F_1,\ldots,F_n$ inherited from
$M_1,\ldots,M_n$.  

\begin{thm}[Trumpet decomposition of maps]\label{sec:decomp-theor}
Let $g,n$ be integers such that $g\geq 0$, $n\geq 1$, and  $(g,n)\neq
(0,1)$ or equivalently $\chi=2-2g-n\leq 0$, and let 
  $L_1,\ell_1,\ldots,L_n,\ell_n$ be  fixed positive integers. Then, the mapping $\Phi$ is a bijection
  between
  \begin{itemize}
\item      sequences $(M_0,M_1,\ldots,M_n)$, where $M_0$ is a map of
  genus $g$ with $n$ rooted tight boundaries of lengths $\ell_i,1\leq i\leq n$
  and for $1\leq i\leq n$, $M_i$ is a trumpet with mouthpiece length $\ell_i$
  and whose other external face has length $L_i$, and
\item maps $M$ of genus $g$ with $n$ rooted boundaries $F_1,\ldots,F_n$ of lengths
  $L_1,\ldots,L_n$, and such that for $1\leq i\leq n$, the minimal length of a closed path
  homotopic to $F_i$ is equal to $\ell_i$.  
\end{itemize}
\end{thm}

\begin{figure}
  \centering
  \includegraphics{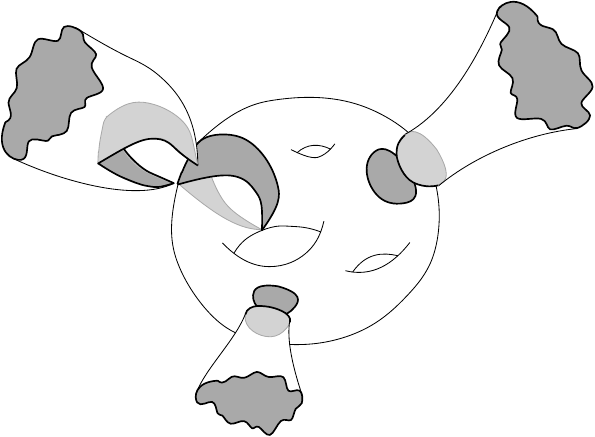}
  \caption[Trumpets decomposition]{An illustration of the
    decomposition of a map of genus $3$ with three external faces
    into a map with tight boundaries and with the same topology, and a brassband of three
    trumpets. We emphasize in this picture that, while the mouthpieces
    of the trumpets always have simple boundaries by Lemma
    \ref{sec:decomp-theor-1}, the (tight) external faces 
    of the central  map need not have simple boundaries.  }
  \label{fig:tubamirum}
\end{figure}

This result is illustrated in Figure \ref{fig:tubamirum}. 
As already mentioned above, this is a generalization of
\cite[Proposition 6.5]{triskell}. The main difficulty comes
from constructing the reverse mapping $\Phi^{-1}(M)$, where $M$ is a
map of genus $g$ with $n$ boundary-faces $F_1,\ldots,F_n$. This consists in
considering, for $i\in \{1,2,\ldots,n\}$, the set $\CCm^{(i)}(M)$ of
closed paths in $M$ that are homotopic to $F_i$, and that have minimal possible length, and to define an
order $\prec^{(i)}$ that makes $\CCm^{(i)}(M)$ a
lattice. In words, we have $c\prec^{(i)}c'$ if $c$ is closer
to $F_i$ than $c'$, although this description
requires some interpretation, which is done in \cite{triskell} by
working in the universal cover of $M$. Since $\CCm^{(i)}(M)$ is clearly a
finite set, this  implies that there is a smallest element $c^i$ in
$\CCm^{(i)}(M)$, that we call the outermost minimal closed path homotopic to
$F_i$. Cutting along these outermost closed paths, for every $i\in
\{1,\ldots,n\}$, produces $n$ trumpets $M_1,\ldots,M_n$ and a remaining
map $M_0$ with $n$ tight boundaries $f^0_1,\ldots,f^0_n$, whose
lengths match with those of the corresponding mouthpieces.  

We will not repeat the detailed argument here and refer the interested
reader to~\cite[Section 6.1]{triskell}.

\subsection{Enumerative consequences}\label{sec:trumpenum}

Theorem \ref{sec:decomp-theor} admits the following immediate
corollary in terms of generating functions (recall that we attach a
weight $t$ to each vertex which is not a boundary, and a weight $t_i$
to each inner face of degree $i$, $i=1,2,3,\ldots$).

\begin{cor}
  \label{cor:funFT}
  For $g\geq 0$ and $L_1,\ldots,L_n$ positive integers ($n \geq 1$),
  let us denote by $F^{(g)}_{L_1,\ldots,L_n}$ the generating function
  of maps of genus $g$ with $n$ rooted boundary-faces of lengths
  $L_1,\ldots,L_n$, and by $\hat{T}^{(g)}_{L_1,\ldots,L_n}$ the
  generating function of those maps whose boundaries are
  tight\footnote{By convention we use a hat symbol to indicate that
    the boundaries are \emph{rooted}. Later on we will consider the
    hatless generating function $T^{(g)}_{L_1,\ldots,L_n}$ where the
    boundaries are \emph{unrooted}, see~\eqref{eq:ThatT}.}. Then, for $(g,n) \neq (0,1)$ or
  equivalently $\chi=2-2g-n\leq 0$, we have
  \begin{equation}
    \label{eq:funFT}
    F^{(g)}_{L_1,\ldots,L_n} = \sum_{\ell_1,\ldots,\ell_n \geq 1}
    A_{L_1,\ell_1} \cdots A_{L_n,\ell_n} \hat{T}^{(g)}_{\ell_1,\ldots,\ell_n}\, ,
  \end{equation}
  where $A_{L,\ell}$ is the generating function of trumpets with
  rooted boundary of length $L$ and mouthpiece of length $\ell$, where
  the vertices incident to the mouthpiece do not receive a weight $t$.
\end{cor}

Note that, by their very definition, $F^{(g)}_{L_1,\ldots,L_n}$ and $ \hat{T}^{(g)}_{L_1,\ldots,L_n}$
are symmetric functions of $L_1,\ldots,L_n$.
The usefulness of the statement above comes from the fact that $A_{L,\ell}$
admits an explicit expression:

\begin{prop}{\cite[Section 9.3]{irredmaps}}\label{sec:enum-coroll}
For any two positive integers $L,\ell$, we have
  \begin{equation}
    \label{eq:Adef}
    A_{L,\ell} = [z^\ell] \left( z + S + \frac{R}z \right)^L\, .
  \end{equation}
\end{prop}

Note that the notation from \cite{irredmaps} differs from that of the
present paper: precisely, Proposition \ref{sec:enum-coroll} is implied
by formula (9.18) in this reference, in the case $d=0$, $d'=\ell$ and
$n=L$. The right-hand side of~\eqref{eq:Adef} is then nothing but the
three-step path generating function $P_\ell(L;R^{(0)},S^{(0)})$,
beware that the variable $z$ used here does not have the same meaning
as in~\cite{irredmaps}.

Let us now introduce the ``grand'' generating function
\begin{equation}\label{eq:Wdef}
  W_n^{(g)}(x_1,\ldots,x_n) := \sum_{L_1,\ldots,L_n \geq 1}
  \frac{F^{(g)}_{L_1,\ldots,L_n}}{x_1^{L_1+1} \cdots x_n^{L_n+1}}
\end{equation}
which is a formal power series in $x_1^{-1},\ldots,x_n^{-1}$ ($n\geq 1$).  Then,
following~\cite[Definition~3.3.1, p.~87]{Eynard2016}, we define for
$(g,n) \neq (0,1), (0,2)$\footnote{Some conventional terms have to be
  added in these cases.}, or equivalently $\chi=2-2g-n<0$, the quantity
\begin{equation}
  \label{eq:omegadef}
  \omega_n^{(g)}(z_1,\ldots,z_n) := W_n^{(g)}\left(x(z_1),\ldots,x(z_n)\right) x'(z_1) \cdots x'(z_n)
\end{equation}
where
\begin{equation}
  x(z) := R^{1/2} \left( z + z^{-1} \right) + S
\end{equation}
is precisely the Zhukovsky transformation discussed in
Section~\ref{sec:context}.

\begin{thm}[Combinatorial interpretation of $\omega_n^{(g)}$]
  \label{thm:enum-coroll-1}
 The series $\omega_n^{(g)}(z_1,\ldots,z_n)$ is a well-defined power series in
  $z_1^{-1},\ldots,z_n^{-1}$ of the form
  \begin{equation}
    \label{eq:omegaexp}
    \omega_n^{(g)}(z_1,\ldots,z_n) = \sum_{\ell_1,\ldots,\ell_n \geq 1}
    \frac{\hat{\tau}^{(g)}_{\ell_1,\ldots,\ell_n}}{z_1^{\ell_1+1} \cdots z_n^{\ell_n+1}}.
  \end{equation}
  Its coefficients are related to the generating functions of maps
  with rooted tight boundaries by
  \begin{equation}
    \label{eq:ttaurel}
    \hat{T}^{(g)}_{\ell_1,\ldots,\ell_n} = \hat{\tau}^{(g)}_{\ell_1,\ldots,\ell_n} R^{(\ell_1+\cdots+\ell_n)/2}.
  \end{equation}
\end{thm}

\begin{proof}
  Observe that
  \begin{equation}
    x(z)^{-1} = \frac{z^{-1}}{R^{1/2} + S z^{-1} + R^{1/2} z^{-2}} = \frac{z^{-1}}{R^{1/2}} - \frac{S z^{-2}}{R} + \cdots
  \end{equation}
  can be seen as a formal power series in $z^{-1}$ without constant
  coefficient, so that for any $L$, $x'(z)/x(z)^{L+1}$ is a
  series containing only powers of $z^{-1}$ larger than or equal to
  $L+1$. This shows that
  \begin{equation}
    \label{eq:omegaFsum}
    \omega_n^{(g)}(z_1,\ldots,z_n) = \sum_{L_1,\ldots,L_n \geq 1} F^{(g)}_{L_1,\ldots,L_n}
    \frac{x'(z_1)}{x(z_1)^{L_1+1}} \cdots \frac{x'(z_n)}{x(z_n)^{L_n+1}}
  \end{equation}
  is a well-defined power series in $z_1^{-1},\ldots,z_n^{-1}$ of the
  form~\eqref{eq:omegaexp}: the only terms of~\eqref{eq:omegaFsum}
  contributing to the coefficient $\hat{\tau}^{(g)}_{\ell_1,\ldots,\ell_n}$ of
  $z_1^{-\ell_1-1} \cdots z_n^{-\ell_n-1}$ are those with
  $1 \leq L_i \leq \ell_i$, for all $i$.
   
  Now we claim that, for any formal Laurent series $F(x)$ in $x^{-1}$,
  we have the identity
  \begin{equation}
    [x^{-1}] F(x) = [z^{-1}] F(x(z)) x'(z).
  \end{equation}
  Indeed, by linearity it suffices to check this identity for
  $F(x)=x^k$, $k \in \Z$: both sides are equal to $1$ for $k=-1$, and
  vanish for $k \neq -1$ (since
  $x(z)^k x'(z) = \frac{d}{dz} \frac{x(z)^{k+1}}{k+1}$ and a
  derivative contains no monomial in $z^{-1}$).  Applying the identity
  for each variable of the multivariate formal Laurent series
  $x_1^{L_1} \cdots x_n^{L_n} W_n^{(g)}(x_1,\ldots,x_n)$, we
  obtain
  \begin{equation}
    \begin{split}
      F^{(g)}_{L_1,\ldots,L_n} &= [x_1^{-1} \cdots x_n^{-1}]
      x_1^{L_1} \cdots x_n^{L_n}
      W_n^{(g)}(x_1,\ldots,x_n)  \\
      & = [z_1^{-1} \cdots z_n^{-1}] x(z_1)^{L_1} \cdots x(z_n)^{L_n} \omega_n^{(g)}(z_1,\ldots,z_n)  \\
      & = \sum_{\ell_1,\ldots,\ell_n \geq 1} \left( [z_1^{\ell_1}]
        x(z_1)^{L_1} \right) \cdots \left( [z_n^{\ell_n}]
        x(z_n)^{L_n} \right) \hat{\tau}^{(g)}_{\ell_1,\ldots,\ell_n} \\
      & = \sum_{\ell_1,\ldots,\ell_n \geq 1} A_{L_1,\ell_1} \cdots
      A_{L_n,\ell_n} R^{(\ell_1+\cdots+\ell_n)/2}
      \hat{\tau}^{(g)}_{\ell_1,\ldots,\ell_n}
    \end{split}
  \end{equation}
  where $A_{L,\ell}$ is as in~\eqref{eq:Adef}. Comparing
  with~\eqref{eq:funFT}, and noting that the matrix
  $(A_{L,\ell})_{L,\ell \geq 1}$ is unitriangular hence invertible, we
  get equality~\eqref{eq:ttaurel}, as desired.
\end{proof}

\begin{rem}
  As we have $x(z)=x(z^{-1})$ and $x'(z)=-x'(z^{-1})/z^2$, we may
  alternatively view $\omega_n^{(g)}(z_1,\ldots,z_n)$ as a formal
  power series in $z_1,\ldots,z_n$ with expansion
  \begin{equation}
    \label{eq:omegaexpbis}
    \omega_n^{(g)}(z_1,\ldots,z_n) = (-1)^n \sum_{\ell_1,\ldots,\ell_n \geq 1}
    \hat{T}^{(g)}_{\ell_1,\ldots,\ell_n} R^{-(\ell_1+\cdots+\ell_n)/2} z_1^{\ell_1-1} \cdots z_n^{\ell_n-1}.
  \end{equation}
  As we will see below, $\omega_n^{(g)}(z_1,\ldots,z_n)$ is actually a
  rational function of $z_1,\ldots,z_n$, and \eqref{eq:omegaexp} and
  \eqref{eq:omegaexpbis} are two different expansions of it (at
  infinity and at zero, respectively).
\end{rem}

\subsection{Allowing boundary-vertices}\label{sec:adding-vert-bound}

Theorem \ref{sec:decomp-theor} deals with the case where all
boundaries of the maps at hand are faces. There is a more general
statement that also deals with both boundary-faces and
boundary-vertices. Let $M_0$ be a map of genus $g$ with $n=m+s$
distinguished tight boundaries, the first $m$ of which are rooted
faces $f_1^0,\ldots,f_m^0$, the remaining $s$ being vertices, for some
$m,s \geq 0$. Let $M_i,1\leq i\leq m$ be a brassband with respective
boundary-faces $f_i,F_i$, where $f_i$ denotes the mouthpiece. We
assume that the lengths of $f_i$ and of $f^0_i$ are equal for
$1\leq i\leq m$. We let $M=\Phi(M_0,M_1,\ldots,M_m)$ be the map
obtained as above by identifying the boundaries of $f_i$ and $f_i^0$
for $1\leq i\leq m$. Then the map $M$ has $m$ rooted boundary-faces
and $s$ boundary-vertices, inherited from $M_0$.

\begin{thm}
  \label{sec:decomp-theor-2}
  Let $g,m,s$ be nonnegative integers such that $n=m+s\geq 1$ and
  $\chi=2-2g-n\leq 0$. Let
  $L_1,\ell_1,\ldots,L_m,\ell_m$ be  fixed positive integers. Then, the mapping $\Phi$ is a bijection
  between
  \begin{itemize}
  \item sequences $(M_0,M_1,\ldots,M_m)$, where $M_0$ is a map of
    genus $g$ with $n$ distinguished tight boundaries, the first $m$
    of which are rooted boundary-faces of lengths
    $\ell_i,1\leq i\leq m$, the remaining $s$ being boundary-vertices,
    and for $1\leq i\leq m$, $M_i$ is a trumpet with mouthpiece length
    $\ell_i$ and whose other external face has length $L_i$, and
  \item maps $M$ of genus $g$ with $n$ distinguished boundaries, the first $m$ of which are rooted
    boundary-faces $F_1,\ldots,F_m$ of lengths $L_1,\ldots,L_m$, the remaining $s$
    being boundary-vertices, and such that for $1\leq i\leq m$, the
    minimal length of a closed path homotopic to $F_i$ is equal to
    $\ell_i$.
\end{itemize}
\end{thm}

The justification for this statement is exactly the same as for
Theorem \ref{sec:decomp-theor}. The
inverse bijection consists in cutting the map $M$ along the outermost minimal
closed paths that are homotopic to the
boundary-faces, considering \emph{both} boundary-faces and 
boundary-vertices as punctures. In particular, the
boundary-vertices are never  
included in the trumpets $M_1,\ldots,M_m$ that are cut away. 

We now consider the enumerative consequences of
Theorem~\ref{sec:decomp-theor-2}. In order to put boundary-faces and
boundary-vertices on a same footing, we denote by
$T^{(g)}_{\ell_1,\ldots,\ell_n}$ the generating function of maps of
genus $g$ with $n$ \emph{unrooted} tight boundaries of lengths
$\ell_1,\ldots,\ell_n$. As the generating function
$\hat{T}^{(g)}_{\ell_1,\ldots,\ell_n}$ considered above, it is symmetric in $\ell_1,\ldots,\ell_n$, and these two quantities are related by
\begin{equation}
  \label{eq:ThatT}
  \hat{T}^{(g)}_{\ell_1,\ldots,\ell_n} = \ell_1 \cdots \ell_n \,
  T^{(g)}_{\ell_1,\ldots,\ell_n}
\end{equation}
as there are $\ell_1 \cdots \ell_n$ ways to root the boundaries, up to the symmetry factor corresponding to the inverse of the automorphism group order\footnote{Note that the objects enumerated by the generating series $T^{(g)}_{\ell_1,\ldots,\ell_n}$ are not rooted, and therefore can have nontrivial automorphism groups. This is why we have to weigh them by the inverse of the automorphism group order, as discussed in the introduction.}. In
particular, there are zero ways if one $\ell_i$ vanishes: as said
in Section~\ref{sec:defs}, a boundary-vertex cannot be rooted. Thus,
$T^{(g)}_{\ell_1,\ldots,\ell_n}$ contains a priori more information
than $\hat{T}^{(g)}_{\ell_1,\ldots,\ell_n}$. In particular, for given integers $m,s\geq 0$ not both equal to $0$, and positive integers $\ell_1,\ldots,\ell_m$, the generating function of maps with $m$ rooted tight boundary-faces of lengths $\ell_1,\ldots,\ell_m$ and with $s$ boundary-vertices is given by $(\ell_1\cdots \ell_m)T^{(g)}_{\ell_1,\ldots,\ell_m,0,\ldots,0}$, with $s$ zero indices. With this remark at hand, we can now state the generating function counterpart to Theorem \ref{sec:decomp-theor-2} as follows. 

\begin{cor}\label{sec:allow-vert-bound}
  Let $g,m,s$ be nonnegative integers such that $n=m+s\geq 1$ and
  $\chi=2-2g-n\leq 0$. Let
  $L_1,\ldots,L_m$ be fixed positive integers. Then, we
  have
  \begin{equation}
    \label{eq:funFTderiv}
   \frac{\partial^{s} F^{(g)}_{L_1,\ldots,L_m}}{\partial t^{s}} = \sum_{\ell_1,\ldots,\ell_m \geq 1}
    A_{L_1,\ell_1} \cdots A_{L_m,\ell_m}\,  (\ell_1\cdots \ell_m)\, T^{(g)}_{\ell_1,\ldots,\ell_m,0,\ldots,0}\, ,
  \end{equation}
  where there are $s$ indices equal to $0$ in the last term, and
  where we reuse the notations from Corollary~\ref{cor:funFT}. For $m=0$, the relation reads
  \begin{equation}
    \label{eq:funFTderiv0}
   \frac{\partial^{s} F^{(g)}}{\partial t^{s}} = T^{(g)}_{0,\ldots,0}\, ,
  \end{equation}
  with $s$ indices equal to $0$, and where $F^{(g)}$ is the generating function of maps of 
  genus $g$ without boundaries.
  \end{cor}

For $g,m,s$ as above, we introduce the following generalization of the grand generating function \eqref{eq:Wdef}: 
\begin{align}\label{eq:Wmsdef}
W^{(g)}_{m,s}(x_1,\ldots,x_m)&:=\sum_{L_1,\ldots,L_m}\frac{\partial^s F^{(g)}_{L_1,\ldots,L_m}}{\partial t^s}\cdot\frac{1}{x_1^{L_1+1}\cdots x_m^{L_m+1}}\\
&=\frac{\partial^s}{\partial t^s} W_{m}^{(g)}(x_1,\ldots,x_m)\, ,
\end{align}
and, assuming further that $\chi<0$, the quantity
\begin{equation}
\label{eq:omegamsdef}
\omega^{(g)}_{m,s}(z_1,\ldots,z_m):=W_{m,s}^{(g)}(x(z_1),\ldots,x(z_m))x'(z_1)\ldots x'(z_m)\, .
\end{equation}
For $m=0$, we have $\omega^{(g)}_{0,s}=W_{0,s}^{(g)}= {\partial^{s} F^{(g)}}/{\partial t^{s}}$.
Corollary \ref{sec:allow-vert-bound} admits the following consequence, analogous to 
Theorem \ref{thm:enum-coroll-1}, whose proof is adapted from the latter in a straightforward way. 

\begin{thm} 
\label{thm:omegams}
For integers $\ell_1,\ldots,\ell_n$, 
define
\begin{equation}
  \label{eq:tauT}
   \tau^{(g)}_{\ell_1,\ldots,\ell_n} := T^{(g)}_{\ell_1,\ldots,\ell_n}
   R^{-(\ell_1+\cdots+\ell_n)/2}
 \end{equation}
 with $T^{(g)}_{\ell_1,\ldots,\ell_n}$ the generating function of maps with unrooted tight boundaries defined above. Then, we have the expansion
 \begin{equation}
   \label{eq:12}
   \omega_{m,s}^{(g)}(z_1,\ldots,z_m)= \sum_{\ell_1,\ldots,\ell_m \geq
     1} \frac{\ell_1 \cdots \ell_m
     \tau^{(g)}_{\ell_1,\ldots,\ell_m,0,\ldots,0}}{z_1^{\ell_1+1}
     \cdots z_m^{\ell_m+1}}
 \end{equation}
 where there are $s$ trailing zeros in the indices of $\tau^{(g)}$.
 For $m=0$, we have $\omega^{(g)}_{0,s}=\tau^{(g)}_{0,\ldots,0}=T^{(g)}_{0,\ldots,0}$.
\end{thm}

Theorems \ref{thm:enum-coroll-1} and \ref{thm:omegams} provide combinatorial interpretations for the coefficients of $\omega^{(g)}_{n}$ and $\omega^{(g)}_{m,s}$ in terms of generating series of maps with tight boundaries. However, they do not provide any information about the structural properties of $\tau^{(g)}_{\ell_1,\ldots,\ell_n}$. This is the goal of the upcoming Section \ref{sec:quasipol}, where we will see that, for $2-2g-n<0$, these are quasi-polynomials in the variables $\ell_1,\ldots,\ell_n$, and that $\tau^{(g)}_{\ell_1,\ldots,\ell_m,0,\ldots,0}$ (with $s$ zero indices) indeed corresponds to the evaluation of this quasi-polynomial function of $n=m+s$ variables when the last $s$ ones are set to $0$.

\begin{rem}
  \label{rem:specialcyl}
  The case $(g,n)=(0,2)$ is special, since we have for $\ell_1>0$
  \begin{equation}
    \label{eq:specialcyl}
    \tau^{(0)}_{\ell_1,\ell_2}=\frac{\delta_{\ell_1,\ell_2}}{\ell_1}, \qquad
    T^{(0)}_{\ell_1,\ell_2}=\frac{R^{\ell_1} \delta_{\ell_1,\ell_2}}{\ell_1}, \qquad
    \hat{T}^{(0)}_{\ell_1,\ell_2}=\ell_1 R^{\ell_1} \delta_{\ell_1,\ell_2}
  \end{equation}
  which are \emph{not} quasi-polynomials in $\ell_1,\ell_2$.  Indeed,
  it is clear that these quantities vanish for $\ell_1\neq \ell_2$,
  and we get $\ell_1 T^{(0)}_{\ell_1,\ell_1}=R^{\ell_1}$ by taking
  $d=0$, $n=d'=\ell_1$ in~\cite[equation~(9.19)]{irredmaps}.
\end{rem}

\subsection{A formula for the series of planar bipartite maps with tight boundaries}
\label{sec:bipcase}

In this section, we give an explicit expression for the generating
function $T^{(0)}_{\ell_1,\ldots,\ell_n}$ of planar maps with $n$
unrooted tight boundaries of lengths $\ell_1,\ldots,\ell_n$, in the
\emph{bipartite case} where the $\ell_i$ are even and where the
weights $t_1,t_3,t_5,\ldots$ for inner faces of odd degrees are set to
zero. To emphasize this latter restriction we will use the notation
$\bip{}$ in the following equations. Recall that the basic generating
function $R$ is here determined by~\eqref{eq:Rbipeq}.

Our input is the Collet-Fusy formula \cite[Theorem~1.1]{CoFu12} for
the generating function of planar bipartite maps with $n$ rooted, not
necessarily tight, boundaries of prescribed lengths. In our present
notations, it reads
\begin{equation}
  \label{eq:CF}
 \bip{F^{(0)}_{2L_1,\ldots,2L_n}}=\frac{L_1\cdots L_n}{L_1+\cdots+L_n}\prod_{i=1}^n\binom{2L_i}{L_i}\frac{\partial^{n-2}}{\partial
 t^{n-2}} R^{L_1+\ldots+L_n}
\end{equation}
for any $n \geq 1$ and positive integers $L_1,\ldots,L_n$. Note that
the formula makes sense for $n=1$, upon understanding that
$\frac{\partial^{-1}}{\partial t^{-1}}$ means integrating over
$t$. For $n=2$, it is a simple exercise to check that the formula is
consistent with Corollary~\ref{cor:funFT} and
\eqref{eq:specialcyl}. We concentrate on the case $n \geq 3$ from now
on.

We will deduce from the Collet-Fusy formula an expression for
$T^{(0)}_{2\ell_1,\ldots,2\ell_n}\rvert_{\mathrm{bip}}$, using
Corollary~\ref{sec:allow-vert-bound}. Interestingly, our expression
involves a certain family of polynomials that we encountered
previously in~\cite{polytightmaps}.  Namely, for any integers
$k\geq 0$ and $n \geq 1$, let us consider the multivariate polynomial in the variables $\ell_1,\ldots,\ell_n$
\begin{equation}
  \label{eq:pkmultdef}
  p_k(\ell_1,\ell_2,\ldots,\ell_n):=\sum_{k_1,k_2,\ldots,k_n\geq
    0\atop k_1+k_2+\cdots+k_n=k}p_{k_1}(\ell_1)q_{k_2}(\ell_2)\cdots
  q_{k_n}(\ell_n)\, ,
\end{equation}
where, on the right-hand side, all factors except the first one are
$q_{k_i}$'s, and where the univariate polynomials $p_k$ and $q_k$ are defined by
\begin{equation}
  \label{eq:defpkqk}
  \begin{split}
   & p_k(\ell):=\frac{1}{(k!)^2}\, \prod_{i=1}^k\left(\ell^2-i^2\right) = \binom{\ell-1}{
k} \binom{\ell+k}{k} \\
   & q_k(\ell):=\frac{1}{(k!)^2}\, \prod_{i=0}^{k-1}\left(\ell^2-i^2\right) = \binom{\ell
}{k} \binom{\ell+k-1}{k} \\
 \end{split}
\end{equation}
with the convention $p_0(\ell)=q_0(\ell)=1$, and with
$\binom{x}{k}=x(x-1)\cdots(x-k+1)/k!$ viewed as a polynomial in
$x$. Clearly, $p_k(\ell_1,\ell_2,\ldots,\ell_n)$ is a polynomial in
the variables $\ell_1^2,\ldots,\ell_n^2$. Moreover, as discussed in
\cite[Proposition 2.1]{polytightmaps}, it is a symmetric polynomial and
it satisfies the consistency relation
\begin{equation}
  \label{eq:pkconsist}
  p_k(\ell_1,\ell_2,\ldots,\ell_n,0)=p_k(\ell_1,\ell_2,\ldots,\ell_n)\, .
\end{equation}
For $k=1,2,3$, we have
\begin{equation}
  \label{eq:firstfewpk}
  \begin{split}
    p_1(\ell_1,\ldots,\ell_n) &= \left(\sum_{i=1}^n \ell_i^2 \right) - 1, \\
    p_2(\ell_1,\ldots,\ell_n) &= \frac14 \left(\sum_{i=1}^n \ell_i^4 \right) +  \sum_{i<j} \ell_i^2 \ell_j^2  - \frac54 \left(\sum_{i=1}^n \ell_i^2 \right) + 1, \\
     p_3(\ell_1,\ldots,\ell_n) &= \frac1{36} \left(\sum_{i=1}^n \ell_i^6 \right) + \frac14 \left( \sum_{i \neq j} \ell_i^4 \ell_j^2 \right) + \sum_{i<j<h} \ell_i^2 \ell_j^2 \ell_h^2 \\
    & \qquad 
    - \frac7{18} \left(\sum_{i=1}^n \ell_i^4 \right) - \frac32 \left( \sum_{i<j} \ell_i^2 \ell_j^2 \right) + \frac{49}{36} \left(\sum_{i=1}^n \ell_i^2 \right) - 1.
  \end{split}
\end{equation}

We will also need the so-called partial exponential Bell polynomials,
defined for any nonnegative integers $n\geq k \geq 1$ by
\begin{multline}
  B_{n,k}(r_1,r_2,\ldots,r_{n-k+1}) := \\
  \sum_{\substack{j_1,j_2,\ldots,j_{n-k+1} \geq 0 \\ j_1+j_2+\cdots+j_{n-k+1}=k \\
      j_1 + 2j_2 + \cdots + (n-k+1)j_{n-k+1}=n}}
  \!\!\!\!\!\!\!\!
  \frac{n!}{j_1!j_2!\cdots j_{n-k+1}!}
  \left(\frac{r_1}{1!}\right)^{j_1} \left(\frac{r_2}{2!}\right)^{j_2} \cdots
  \left(\frac{r_{n-k+1}}{(n-k+1)!}\right)^{j_{n-k+1}}.
\end{multline}
Combinatorially, $B_{n,k}(r_1,r_2,\ldots)$ is the generating series of
partitions of an $n$-element set into $k$ blocks, where we attach a
weight $r_i$ to each block of size $i$. Table~\ref{tab:bellpol} lists
the first few Bell polynomials. They appear in the following classical
formula:

\begin{prop}[Faà di Bruno's formula]
  \label{prop:fdbformula}
  Let $f,g$ be sufficiently smooth functions of one variable, and let
  $n$ be a positive integer. Then, we have
  \begin{equation}
    \frac{d^n}{dt^n} (f \circ g)(t) = \sum_{k=1}^n f^{(k)}(g(t))
    B_{n,k}\left(g'(t),g''(t),\ldots,g^{(n-k+1)}(t)\right)\, .
  \end{equation}
\end{prop}

\begin{table}
  \centering
  \begin{tabular}{|c|c|c|c|c|}
    \hline
    \diagbox{$n$}{$k$} & $1$ & $2$ & $3$ & $4$ \\ \hline
    $1$ & $r_1$ & & & \\ \hline
    $2$ & $r_2$ & $r_1^2$ & & \\ \hline
    $3$ & $r_3$ & $3 r_1 r_2$ & $r_1^3$ & \\ \hline
    $4$ & $r_4$ & $4 r_1 r_3 + 3 r_2^2$ & $6 r_1^2 r_2$ & $r_1^4$ \\ \hline
  \end{tabular}
  \caption{The first few partial exponential Bell polynomials
    $B_{n,k}(r_1,r_2,\ldots)$.}
  \label{tab:bellpol}
\end{table}

See for instance the discussion in~\cite[Example~III.24]{FlSe09}.  We may now state the main result of
this section:

\begin{thm}
  \label{thm:Tgen}
  For $n\geq 3$ and $\ell_1,\ldots,\ell_n$ nonnegative integers, the
  generating function of bipartite planar maps
  with $n$ unrooted tight boundaries of lengths $2\ell_1,\ldots,2\ell_n$ is equal to
  \begin{equation}
    \label{eq:Tgen}
    \bip{T^{(0)}_{2\ell_1,\ldots,2\ell_n}} = R^{\ell_1+\cdots+\ell_n} \sum_{k=0}^{n-3}    k! p_k(\ell_1,\ldots,\ell_n) b_{n-2,k+1}  + \frac{(-1)^n (n-3)!}{t^{n-2}} \delta_{\ell_1+\cdots+\ell_n,0}
  \end{equation}
  with $R$ given by \eqref{eq:Rbipeq}, and where we introduce the shorthand notation
  \begin{equation}
    \label{eq:bnkdef}
    b_{n,k} :=  R^{-k} B_{n,k}(R',R'',\ldots) = B_{n,k} \left(
      \frac{R'}R, \frac{R''}R, \ldots \right)
  \end{equation}
  with $R',R'',\ldots$ the successive derivatives of $R$ with respect
  to the vertex weight $t$.
\end{thm}

For $n=3$ we recover
$T^{(0)}_{2\ell_1,2\ell_2,2\ell_3}\rvert_{\mathrm{bip}} =
R^{\ell_1+\ell_2+\ell_3-1} R' - t^{-1}
\delta_{\ell_1+\ell_2+\ell_3,0}$ consistently
with~\cite[Theorem~1.1]{triskell}, while for $n=4$ we find
\begin{equation}
  \label{eq:Tfour}
  \bip{T^{(0)}_{2\ell_1,2\ell_2,2\ell_3,2\ell_4}} = R^{\ell_1+\ell_2+\ell_3+\ell_4}
  \left( \left( \ell_1^2 + \ell_2^2 + \ell_3^2 + \ell_4^2 -1 \right)
    \frac{R'^2}{R^2} + \frac{R''}{R} \right) + \frac{\delta_{\ell_1+\cdots+\ell_4,0}}{t^2}.
\end{equation}
Longer expressions for $n=5,6$ may be written straightforwardly
from~\eqref{eq:firstfewpk} and Table~\ref{tab:bellpol}.

\begin{rem}
  When setting the weights $t_2,t_4,\ldots$ for inner faces to zero,
  we have $R=t$, $R'=1$ and $R^{(j)}=0$ for $j \geq 2$. Thus, in this
  case we have $b_{n,k}=t^{-k}\delta_{n,k}$, hence
  $T^{(0)}_{2\ell_1,\ldots,2\ell_n}$ reduces to
  $(n-3)!p_{n-3}(\ell_1,\ldots,\ell_n) t^{\ell_1+\ldots+\ell_n-n+2}$
  if at least one $\ell_i$ is non-zero, and vanishes otherwise since
  $p_{n-3}(0,\ldots,0)=(-1)^{n-3}$ so the rightmost term
  in~\eqref{eq:Tgen} cancels the result. This is consistent
  with~\cite[Theorem 2.3]{polytightmaps} since the tight maps considered
  in this reference are nothing but maps with tight boundaries and
  without inner faces. Note that the exponent of $t$ is consistent
  with Euler's relation.
\end{rem}

\begin{proof}[Proof of Theorem~\ref{thm:Tgen}]
  Let us first consider the case where all $\ell_i$ are zero: then
  $T^{(0)}_{0,\ldots,0}$ is simply the generating function of
  bipartite planar maps with $n \geq 3$ marked vertices. For $n=3$ it
  is equal to $\frac{\partial}{\partial t} \ln(R/t)$, see
  e.g.~\cite[Appendix~A]{triskell}, and for larger $n$ we simply have
  to take more derivatives with respect to $t$, namely
  \begin{equation}
    \bip{T^{(0)}_{0,\ldots,0}} = \frac{\partial^{n-2}}{\partial t^{n-2}} \ln(R/t) =
    \frac{\partial^{n-2}}{\partial t^{n-2}} \ln R + \frac{(-1)^n (n-3)!}{t^{n-2}}\, .
  \end{equation}
  Now, by applying Faà di Bruno's formula, we get
  \begin{multline}
    \bip{T^{(0)}_{0,\ldots,0}} = \sum_{k=1}^{n-2} (-1)^{k-1} (k-1)! R^{-k} B_{n-2,k}(R',R'',\ldots) +
    \frac{(-1)^n (n-3)!}{t^{n-2}}\\
    = \sum_{k=0}^{n-3} (-1)^{k} k! b_{n-2,k+1} +
    \frac{(-1)^n (n-3)!}{t^{n-2}}.
  \end{multline}
  This gives the wanted formula for $\ell_1=\cdots=\ell_n=0$ since, by
  \eqref{eq:defpkqk} and \eqref{eq:pkconsist}, we have
  $p_{k}(0,\ldots,0)=p_{k}(0)=(-1)^{k}$.

  Let us now assume that at least one $\ell_i$ is non-zero, and let
  $m \geq 1$ be the number of non-zero $\ell_i$, and $s=n-m$ the number
  of zero $\ell_i$.  Without loss of
  generality, we can assume that $\ell_1,\ldots,\ell_m>0$ and
  $\ell_i=0$ for $i>m$. We note that, in the bipartite case, the
  trumpet generating function of Proposition \ref{sec:enum-coroll}
  vanishes whenever its indices are of different parities (since
  $S=0$), and reads for even indices
  \begin{equation}
    \label{eq:Abip}
    A_{2L,2\ell}=\binom{2L}{L-\ell}R^{L-\ell}\, .
  \end{equation}
  Thus, Corollary~\ref{sec:allow-vert-bound} reduces to
  \begin{equation}
    \label{eq:funneq}
    \frac{\partial^{s}}{\partial t^{s}}
    \bip{F^{(0)}_{2L_1,\ldots,2L_m}}
    =  
    \sum_{\ell_1,\ldots,\ell_m\geq 1}(2\ell_1)A_{2L_1,2\ell_1} \cdots (2\ell_m)A_{2L_m,2\ell_m} \bip{T^{(0)}_{2\ell_1,\ldots,2\ell_m,0,\ldots,0}}
  \end{equation}  
  (where we append $s$ zeros after $2\ell_m$).  On the other hand,
  by the Collet-Fusy formula and by Fa\`a di Bruno's formula, we have
  \begin{multline}
    \label{eq:CFgenexp}
    \frac{\partial^{s}}{\partial t^{s}}
    \bip{F^{(0)}_{2L_1,\ldots,2L_m}} = \frac{L_1 \cdots L_m}{L_1+\cdots+L_m}
    \binom{2L_1}{L_1} \cdots \binom{2L_m}{L_m} \frac{\partial^{n-2}}{\partial t^{n-2}} R^{L_1+\cdots+L_m} \\
    =  \sum_{k=1}^{n-2} L_1 \cdots L_m (L_1+\cdots+L_m-1)_{k-1} \binom{2L_1}{L_1} \cdots \binom{2L_m}{L_m} R^{L_1+\cdots+L_m} b_{n-2,k}
  \end{multline}
  where $(L)_k:=L(L-1)\cdots(L-k+1)$ denotes the falling factorial.
  We will rewrite this expression in the same form as the right-hand
  side of~\eqref{eq:funneq}, so as to identify
  $\bip{T^{(0)}_{2\ell_1,\ldots,2\ell_m,0,\ldots,0}}$. To this end, we
  proceed as in~\cite[Section~3.1]{polytightmaps} and use the
  Chu-Vandermonde identity to write
  \begin{equation}
    (L_1+L_2+\cdots+L_m-1)_{k-1} = (k-1)! \sum_{\substack{k_1,k_2,\ldots,k_m\geq 0 \\ k_1+k_2+\cdots+k_m=k-1}} \binom{L_1-1}{k_1} \binom{L_2}{k_2} \cdots \binom{L_m}{k_m}.
  \end{equation}
  Plugging into~\eqref{eq:CFgenexp} we obtain
  \begin{multline}
    \label{eq:CFexpand}
    \frac{\partial^{s}}{\partial t^{s}}
    \bip{F^{(0)}_{2L_1,\ldots,2L_m}} =
    \sum_{k=1}^{n-2}(k-1)!  R^{L_1+\cdots+L_m} b_{n-2,k} \times \\
    \sum_{\underset{k_1+\cdots+k_m=k-1}{k_1,\ldots,k_m\geq
        0}} 
     L_1 \binom{L_1-1}{k_1} \binom{2L_1}{L_1} L_2 \binom{L_2}{k_2} \binom{2L_2}{L_2} \cdots L_m \binom{L_m}{k_m} \binom{2L_m}{L_m}\, .
  \end{multline}
  Now, by \cite[Lemma 3.1]{polytightmaps} and
  by~\eqref{eq:Abip}\footnote{Beware that the notation $A_{\ell,m}$
    used in \cite{polytightmaps} differs from that of the present
    paper.}, we have for $L>0$
  \begin{equation}
    \label{eq:pqdeconv}
    \begin{split}
      L \binom{L-1}{k}  \binom{2L}{L} &= \sum_{\ell=1}^L(2\ell)\binom{2L}{L-\ell}p_k(\ell)=\sum_{\ell=1}^L(2\ell)A_{2L,2\ell}R^{\ell-L}p_k(\ell)\, ,\\
      L \binom{L}{k}  \binom{2L}{L} &= \sum_{\ell=1}^L(2\ell)\binom{2L}{L-\ell}q_k(\ell)=\sum_{\ell=1}^L(2\ell)A_{2L,2\ell}R^{\ell-L}q_k(\ell)\, .
    \end{split}
  \end{equation}
  Therefore, by the definition \eqref{eq:pkmultdef} of $p_k$, we may rewrite
  \eqref{eq:CFexpand} as
  \begin{multline}
    \label{eq:CFpk}
    \frac{\partial^{s}}{\partial t^{s}}
    \bip{F^{(0)}_{2L_1,\ldots,2L_m}} =
    \sum_{\ell_1,\ldots,\ell_m\geq
      1}(2\ell_1)A_{2L_1,2\ell_1}\cdots
    (2\ell_m)A_{2L_m,2\ell_m}\\
    \times
    R^{\ell_1+\ldots+\ell_m} \sum_{k=1}^{n-2}(k-1)! p_{k-1}(\ell_1,\ldots,\ell_m)
     b_{n-2,k}\, ,
  \end{multline}
  and by \eqref{eq:pkconsist}, we may replace
  $p_{k-1}(\ell_1,\ldots,\ell_m)$ by
  $p_{k-1}(\ell_1,\ldots,\ell_m,0,\ldots,0)$, in order to have $n$
  variables. Comparing this expression with \eqref{eq:funneq}, using
  the invertibility of the unitriangular matrix
  $(A_{2L,2\ell})_{L,\ell\geq1}$ and doing a change of variable
  $k-1 \to k$ finally give the wanted formula~\eqref{eq:Tgen}, since
  the term $\frac{(-1)^n (n-3)!}{t^{n-2}}$ is absent for $m \geq 1$.
\end{proof}

A variant of the Collet-Fusy formula~\eqref{eq:CF} holds in the
\emph{quasi-bipartite case} where exactly two of the boundaries have
odd lengths, i.e.\ when exactly two of the $L_i$, say $L_1$ and $L_2$,
are half-integers. In this case, the factors
$\binom{2L_1}{L_1} \binom{2L_2}{L_2}$ appearing on the right-hand side
should be replaced by
$4\binom{2L_1-1}{L_1-1/2} \binom{2L_2-1}{L_2-1/2}$. Correspondingly,
we have the following analogue of Theorem~\ref{thm:Tgen}:

\begin{prop}
  \label{prop:Tgenquasi}
  For $n\geq 3$, $\ell_1,\ell_2$ positive half-integers and
  $\ell_3,\ldots,\ell_n$ nonnegative integers, the generating function
  of quasi-bipartite planar maps with $n$ unrooted tight boundaries of
  lengths $2\ell_1,\ldots,2\ell_n$ is equal to
  \begin{equation}
    \label{eq:Tgenquasi}
    \bip{T^{(0)}_{2\ell_1,\ldots,2\ell_n}} = R^{\ell_1+\cdots+\ell_n} \sum_{k=0}^{n-3}    k! \tilde{p}_k(\ell_1,\ell_2;\ell_3,\ldots,\ell_n) b_{n-2,k+1}
  \end{equation}
  where, following~\cite[Section~2.1.3]{polytightmaps}, we define
  \begin{equation}
    \label{eq:defptildekmultivariate}
    \tilde{p}_k(\ell_1,\ell_2;\ell_3,\ldots,\ell_n):=\sum_{k_1,k_2,\ldots,k_n\geq 0\atop k_1+k_2+\cdots+k_n=k}\tilde{p}_{k_1}(\ell_1)
    \tilde{p}_{k_2}(\ell_2)q_{k_3}(\ell_3)\cdots q_{k_n}(\ell_n)
  \end{equation}
  with $q_k$ as in~\eqref{eq:defpkqk} and with
  \begin{equation}
    \tilde{p}_k(\ell):=\frac{1}{(k!)^2}\, \prod_{i=1}^k\left(\ell^2-\left(i-\frac{1}{2}\right)^2\right) = \binom{\ell-\frac12}{k}\binom{\ell+k-\frac12}{k}.
  \end{equation}
\end{prop}

\begin{proof}
  Starting with the quasi-bipartite variant of the Collet-Fusy
  formula, we proceed as in the proof of Theorem~\ref{thm:Tgen}.
  Using Fa\`a di Bruno's formula, we obtain the
  expression~\eqref{eq:CFgenexp} with the binomial factors
  $\binom{2L_1}{L_1} \binom{2L_2}{L_2}$ modified as discussed
  above. Now, we expand the falling factorial as follows:
    \begin{equation}
      (L_1+\cdots+L_m-1)_{k-1} = (k-1)! \sum_{\substack{k_1,\ldots,k_m\geq 0 \\ k_1+\cdots+k_m=k-1}} \binom{L_1-\frac12}{k_1}  \binom{L_2-\frac12}{k_2} \binom{L_3}{k_3} \cdots \binom{L_m}{k_m}.
  \end{equation}
  which leads to a variant of~\eqref{eq:CFexpand} in which the factors
  $\binom{L_1-1}{k_1} \binom{2L_1}{L_1} \binom{L_2}{k_2}
  \binom{2L_2}{L_2}$ are replaced by
  $4 \binom{L_1-1/2}{k_1} \binom{2L_1-1}{L_1-1/2} \binom{L_2-1/2}{k_2}
  \binom{2L_2-1}{L_2-1/2}$. By~\cite[Equation~(3.9)]{polytightmaps}, we
  have for $L$ a positive half-integer
  \begin{equation}
    2L \binom{L-\frac12}{k}  \binom{2L-1}{L-\frac12} = \sum_{\ell=\frac12}^L(2\ell)\binom{2L}{L-\ell}\tilde{p}_k(\ell)=\sum_{\ell=\frac12}^L(2\ell)A_{2L,2\ell}R^{\ell-L}\tilde{p}_k(\ell)
  \end{equation}
  where the sum runs over the positive half-integers $\ell$ between
  $\frac12$ and $L$. Using again the second line
  of~\eqref{eq:pqdeconv} to expand the remaining binomial factors, we
  arrive at a variant of~\eqref{eq:CFpk} with $p_{k-1}$ replaced by
  $\tilde{p}_{k-1}$, giving the wanted result.
\end{proof}

Let us consider the case of four boundaries. We have
$\tilde{p}_1(\ell_1,\ell_2;\ell_3,\ell_4)=\ell_1^2+\cdots+\ell_4^2-\frac12$
and $\tilde{p}_0(\ell_1,\ell_2;\ell_3,\ell_4)=1$, hence from
Proposition~\ref{prop:Tgenquasi} we see that, in the quasi-bipartite
case, the term $-1$ appearing in the expression~\eqref{eq:Tfour}
should be replaced by $-\frac12$. In contrast, we found after a
lengthy computation (which we omit here) that~\eqref{eq:Tfour} holds
as is when all $\ell_i$ are half-integers, i.e.\ when all four
boundaries have odd lengths. This is consistent with the discussion of
the lattice count polynomial $N_{0,4}$ at the end
of~\cite[Section~1]{Norbury2010}, since
$T^{(0)}_{2\ell_1,2\ell_2,2\ell_3,2\ell_4}$ reduces to
$N_{0,4}(2\ell_1,2\ell_2,2\ell_3,2\ell_4)$ when we set the weights for
inner faces to zero.

\section{Recursion relations for series of maps with tight boundaries}
\label{sec:recrel}

In this section we give a number of recursion relations for the
generating function $T^{(g)}_{\ell_1,\ldots,\ell_n}$, as obtained by
adding an extra tight boundary to a map with pre-existing
boundaries. We will need the following proposition, proved bijectively
in Appendix~\ref{app:tightpants}:

\begin{prop}
 Let $\ell_1$, $\ell_2$ and $\ell_3$ be nonnegative integers and
  let $T^{(0)}_{\ell_1,\ell_2|\ell_3}$ denote the
  generating function of planar maps with three labeled distinct tight unrooted
  boundaries of lengths $\ell_1,\ell_2,\ell_3$, the third being \emph{strictly tight}, 
  where we attach a weight $t$ per vertex different from a boundary-vertex and not incident to the
  third boundary (and for all $k\geq 1$, a weight $t_{k}$ per inner face of degree $k$).
    Then, for $\ell_1+\ell_2>\ell_3$, we have
  \begin{equation}
    \label{eq:strictpantsformula}
    T^{(0)}_{\ell_1,\ell_2|\ell_3} =
    \begin{cases}
      R^{(\ell_1+\ell_2-\ell_3)/2-1} \displaystyle{\frac{\partial R}{\partial t}}& \text{if $\ell_1+\ell_2-\ell_3$ is even,}\\ \\
      R^{(\ell_1+\ell_2-\ell_3-1)/2} \displaystyle{\frac{\partial S}{\partial t}} & \text{if $\ell_1+\ell_2-\ell_3$ is odd.}
    \end{cases}
  \end{equation}
\end{prop}

\subsection{Adding an extra boundary-vertex}
\label{sec:extrabvtx}

\paragraph{Marking a vertex in a trumpet.}

Take $L,\ell$ positive integers. Recall that $A_{L,\ell}$ is the generating function for trumpets with rooted boundary of
length $L$ and mouthpiece of length $\ell$. The quantity $\frac{\partial A_{L,\ell}}{\partial t}$ is therefore the generating function
for such trumpets with a marked vertex which is not incident to the mouthpiece (recall that vertices of the mouthpiece do not receive the weight $t$).
It has the following expression:
\begin{prop} We have 
\begin{equation}
\begin{split}
\frac{\partial A_{L,\ell}}{\partial t} & = \sum_{m>\ell}A_{L,m}\, m\, T^{(0)}_{m,0|\ell}  \\
& = \sum_{m=\ell+1\atop m=\ell+1\!\!\!\!\mod 2}^L A_{L,m}\, m\, R^{\frac{m-\ell-1}{2}}\frac{\partial S}{\partial t}   
+ \sum_{m=\ell+2\atop m=\ell+2\!\!\!\!\mod 2}^L A_{L,m}\, m\, R^{\frac{m-\ell}{2}-1}\frac{\partial R}{\partial t}\, .
\label{eq:dAdt}
\end{split}
\end{equation}
\label{prop:markedtrumpet}
\end{prop}
\begin{proof}[Combinatorial proof]
Considering the marked vertex as a boundary-vertex, we obtain a planar map with three boundaries which me may decompose as follows:
consider the outermost minimal closed path homotopic to the boundary of length $L$. Cutting along this path  produces a 
trumpet  whose mouthpiece has length $m>\ell$, see Figure~\ref{fig:markedtrumpet}. The remaining map with three tight boundaries is precisely 
a map enumerated by $T^{(0)}_{m,0|\ell}$, endowed canonically with one of its $m$ rootings (inherited from the rooting of the boundary face of length $L$).
\end{proof}
\begin{figure}
  \centering
  \fig{.8}{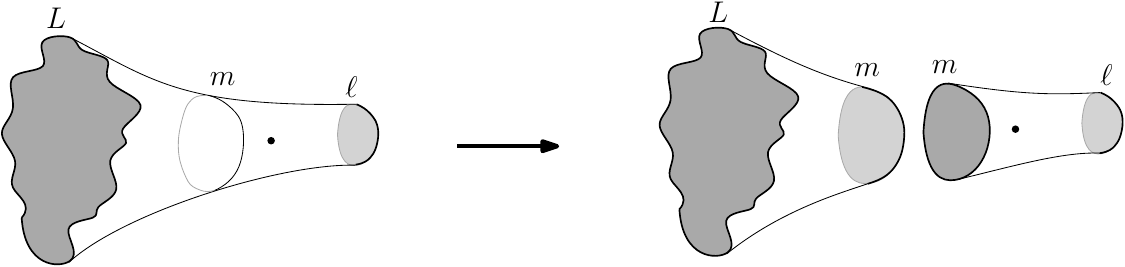}
  \caption{An illustration of the
    decomposition of a trumpet with mouthpiece of length $\ell$ endowed with an extra marked vertex, viewed as a boundary vertex. 
   It results in two pieces: a trumpet with mouthpiece of length $m>\ell$ and a map with three boundaries counted by $m\, T^{(0)}_{m,0|\ell}$. }
  \label{fig:markedtrumpet}
\end{figure}

It is interesting to note that the expression in the second line in Eq.~\eqref{eq:dAdt} may be obtained in a purely computational
way as follows: 
\begin{proof}[Computational proof]
From the explicit expression \eqref{eq:Adef} of $A_{L,\ell}$, we have 
\begin{multline}
\frac{\partial  A_{L,\ell}}{\partial t}  = \frac{\partial \ }{\partial t}
\left([z^\ell]\left(z+S+\frac{R}{z}\right)^L \right) \\
=[z^\ell]\  L \left(z+S+\frac{R}{z}\right)^{L-1}\left(\frac{\partial  S}{\partial t}+\frac{1}{z}\frac{\partial  R}{\partial t}\right)
= L\, A_{L-1,\ell} \frac{\partial  S}{\partial t}+L\, A_{L-1,\ell+1} \frac{\partial  R}{\partial t}.
\end{multline}
To obtain the second line of \eqref{eq:dAdt}, it is enough to prove the equality 
\begin{equation}
L\, A_{L-1,\ell} = \sum_{p\geq 0} (\ell+2p+1) R^p A_{L,\ell+2p+1}\ ,
\end{equation}
which itself follows from iterating the identity
\begin{equation}
L\, A_{L-1,\ell} =  (\ell+1) A_{L,\ell+1}+R\, L\, A_{L-1,\ell+2}
\label{eq:LA}
\end{equation}
$p$ times until $p$ reaches a value such that $\ell+2p>L-1$. As for this last identity \eqref{eq:LA}, upon moving its last term 
to the left hand side, it is simply obtained by extracting the $z^\ell$ coefficient of the straightforward identity
\begin{equation}
\left(1-\frac{R}{z^2}\right)\frac{\partial \ }{\partial S} \left(z+S+\frac{R}{z}\right)^L= \frac{\partial\ }{\partial z} \left(z+S+\frac{R}{z}\right)^L\ .
\end{equation} 
\end{proof}

\paragraph{Marking a vertex in a map with tight boundaries.} 
We start with equation \eqref{eq:funFTderiv} with $m=n$ and $s=0$, which we differentiate
with respect to $t$. We get
 \begin{equation}
    \label{eq:dtfunFT}
    \begin{split}
    \frac{\partial F^{(g)}_{L_1,\ldots,L_n}}{\partial t} & = \sum_{\ell_1,\ldots,\ell_n \geq 1}
    A_{L_1,\ell_1} \cdots A_{L_n,\ell_n}( \ell_1\cdots \ell_n)   \frac{\partial T^{(g)}_{\ell_1,\ldots,\ell_n}}{\partial t}
    \\ & +\sum_{i=1}^n \sum_{\ell_1,\ldots,\ell_n \geq 1}
    A_{L_1,\ell_1} \cdots   \frac{\partial A_{L_i,\ell_i}}{\partial t} \cdots A_{L_n,\ell_n}( \ell_1\cdots \ell_n) T^{(g)}_{\ell_1,\ldots,\ell_n}\ .\\
    \end{split}
  \end{equation}
From \eqref{eq:dAdt}, the general term in the final sum over $i$ may be written as 
\begin{equation}
\begin{split}
\sum_{\ell_1,\ldots,\ell_n \geq 1}
    A_{L_1,\ell_1} \cdots  & \left( \sum_{m_i>\ell_i}A_{L_i,m_i}\, m_i\, T^{(0)}_{m_i,0|\ell_i}   \right) \cdots A_{L_n,\ell_n}( \ell_1\cdots \ell_i\cdots \ell_n) T^{(g)}_{\ell_1,\ldots,\ell_i,\ldots \ell_n}
\\
= \sum_{\ell_1,\ldots,\ell_n \geq 1} A_{L_1,\ell_1} \cdots  & A_{L_i,\ell_i} \cdots A_{L_n,\ell_n}( \ell_1\cdots \ell_i\cdots \ell_n) \left( \sum_{m_i<\ell_i} m_i\, T^{(0)}_{\ell_i,0|m_i} T^{(g)}_{\ell_1,\ldots,m_i,\ldots \ell_n} \right)
\end{split}
\end{equation}
where, to go from the first to the second line, we exchanged the dummy summation variables $\ell_i $ and $m_i$, as well as the order of summation. 
Comparing with the expression for $\frac{\partial F^{(g)}_{L_1,\ldots,L_n}}{\partial t}$ coming from \eqref{eq:funFTderiv} with $m=n$ and $s=1$ , namely
\begin{equation}
    \label{eq:dtfunFTbis}
    \frac{\partial F^{(g)}_{L_1,\ldots,L_n}}{\partial t}  = \sum_{\ell_1,\ldots,\ell_n \geq 1}
    A_{L_1,\ell_1} \cdots A_{L_n,\ell_n}( \ell_1\cdots \ell_n)   T^{(g)}_{\ell_1,\ldots,\ell_n,0}\ ,
\end{equation}
and using the invertibility of the unitriangular matrix
$(A_{L,\ell})_{L,\ell\geq1}$, we obtain the identity
\begin{equation}
  \label{eq:Tgmarkv}
  T^{(g)}_{\ell_1,\ldots,\ell_n,0} =  \frac{\partial T^{(g)}_{\ell_1,\ldots,\ell_n}}{\partial t}  +\sum_{i=1}^n \sum_{m_i<\ell_i} m_i\, T^{(0)}_{\ell_i,0|m_i} T^{(g)}_{\ell_1,\ldots,m_i,\ldots,\ell_n}.
\end{equation}
This identity has the following nice combinatorial interpretation:
  consider a map counted by $T^{(g)}_{\ell_1,\ldots,\ell_n,0}$, and let $v$ be the
  boundary-vertex corresponding to the last $0$. We claim that there
  exists at most one $i=1,\ldots,n$ such that there exists a closed
  path of length $m_i<\ell_i$ separating $v$ and the $i$-th boundary from
  the others. Indeed, if two such $i_1$ and $i_2$ existed, we could
  construct a closed path contradicting the tightness of the $i_1$-th or
  the $i_2$-th boundary (see Figure \ref{fig:exclusion}). If no such $i$ exists, 
  $v$ can be changed into a regular vertex without affecting the tightness of the other
  boundaries, and we get a map counted by
  $\frac{\partial}{\partial t} T^{(g)}_{\ell_1,\ldots,\ell_n}$. Otherwise, if
  exactly one such $i$ exists, we can consider the ``extremal''
  shortest closed
  path separating $v$ and the $i$-th boundary from the others. By cutting along this path,
  we get on the one hand a map counted by $T^{(0)}_{\ell_i,0|m_i}$ (it is a
  pair of pants with one tight boundary of length $\ell_i$, one boundary-vertex and one strictly
  tight boundary of length $m_i$) and on the
  other hand a map counted by
  $T^{(g)}_{\ell_1,\ldots,m_i,\ldots,\ell_n}$. There are $m_i$ ways
  to glue these maps back together.

  Note that the above reasoning also works when $\ell_i$ is allowed to
  take the value $0$, in which case the sum over $m_i$ vanishes. Putting things together, we arrive at the following proposition.
  
    \begin{prop}
  \label{prop:recur0}
 Let $g,n$ be nonnegative integers. We have, for $\chi=2-2g-n\leq 0$ and $\ell_1,\ldots,\ell_n$ nonnegative integers,
\begin{equation}\label{eq:recur0}
\begin{split}
T^{(g)}_{\ell_1,\ldots,\ell_n,0}  & =  \frac{\partial T^{(g)}_{\ell_1,\ldots,\ell_n}}{\partial t}  +\sum_{i=1}^n \sum_{m_i<\ell_i} m_i\, T^{(0)}_{\ell_i,0|m_i} T^{(g)}_{\ell_1,\ldots,m_i,\ldots,\ell_n}  \\ 
&  =   \frac{\partial T^{(g)}_{\ell_1,\ldots,\ell_n}}{\partial t}  +\sum_{i=1}^n \bigg( \sum_{m_i=1 \atop m_i=\ell_i-1\!\!\!\!\mod 2}^{\ell_i-1}\!\!\!\!\!\!\!\! m_i\, R^{\frac{\ell_i-m_i-1}{2}}\frac{\partial  S}{\partial t}T^{(g)}_{\ell_1,\ldots,m_i,\ldots,\ell_n}\\
&\hskip 3.3cm + \sum_{m_i=1 \atop m_i=\ell_i-2\!\!\!\!\mod 2}^{\ell_i-2}\!\!\!\!\!\!\!\! m_i\, R^{\frac{\ell_i-m_i}{2}-1}\frac{\partial  R}{\partial t}T^{(g)}_{\ell_1,\ldots,m_i,\ldots,\ell_n}\bigg)\, .\\
\end{split}
\end{equation}
\end{prop}
Note that, for $n=0$, the relation reads $T^{(g)}_0=\partial T^{(g)}_\varnothing/\partial t=\partial F^{(g)}/\partial t$,
where $T^{(g)}_\varnothing=F^{(g)}$ is the generating function of maps of genus $g$ without boundaries.

\begin{figure}
  \centering
  \fig{.5}{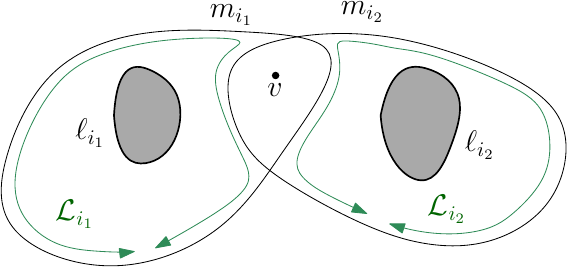}
  \caption{A sketch of why  $m_{i_1}<\ell_{i_1}$ and $m_{i_2}<\ell_{i_2}$ are mutually exclusive: indeed, it would imply $\mathcal{L}_{i_1}<\ell_{i_1}$ or 
   $\mathcal{L}_{i_2}<\ell_{i_2}$, in contradiction with the fact that the boundary-faces of length $\ell_1$
and $\ell_2$ are both tight.}
  \label{fig:exclusion}
\end{figure}

\subsection{Adding an extra boundary-face}
\label{sec:extrabface}

\paragraph{Marking an extra boundary-face in a trumpet.}
Take $L,\ell,M$ positive integers. The quantity $M \frac{\partial A_{L,\ell}}{\partial t_M}$ is the generating function
for trumpets with rooted boundary of length $L$, mouthpiece of length $\ell$ and with an extra marked rooted boundary-face of length $M$. 
\begin{prop} We have 
\begin{equation}
M \frac{\partial A_{L,\ell}}{\partial t_M}  = \sum_{m>\ell}\sum_{\ell_0\geq 1} A_{L,m}\, A_{M,\ell_0}\, m\, \ell_0\, T^{(0)}_{m,\ell_0|\ell} .
\label{eq:dAdtM}
\end{equation}
\end{prop}
The proof is in all points similar to that of Proposition~\ref{prop:markedtrumpet} and amounts to performing a trumpet decomposition
of the map with three boundaries of length $L$, $M$ and $\ell$.

\paragraph{Marking extra boundary-faces via boundary insertion operators.}
We now wish to find a recurrence relation relating
$T^{(g)}_{m_1,\ldots,m_n,m_{n+1}}$ to $T^{(g)}_{m_1,\ldots,m_n}$ for $m_{n+1}$ a
positive integer.

For $m$ a positive integer, let us define the \emph{tight boundary
  insertion operator} $D_m$ by
\begin{equation}
  \label{eq:Dmdefgen}
  D_m := \sum_{M=1}^m (A^{-1})_{m,M} \frac{M}{m} \frac{\partial\ }{\partial t_{M}}
\end{equation}
where $A^{-1}$ is the inverse of the unitriangular matrix $A=(A_{M,m})_{M,m\geq 1}$. It satisfies
\begin{equation}\label{eq:9}
  M\frac{\partial\ }{\partial t_{M}} = \sum_{m=1}^M m  A_{M,m}  D_m\ .
\end{equation}
We start with the identity
\begin{equation}
L_{n+1}\frac{\partial  F^{(g)}_{L_1,\ldots,L_n}}{\partial t_{L_{n+1}}} =F^{(g)}_{L_1,\ldots,L_n,L_{n+1}}
 = \sum_{\ell_1,\ldots,\ell_{n+1} \geq 1}
    A_{L_1,\ell_1} \cdots A_{L_{n+1},\ell_{n+1}}\,  (\ell_1\cdots \ell_{n+1})\, T^{(g)}_{\ell_1,\ldots,\ell_{n+1}}\ .
\label{eq:dFdtL}
\end{equation}
Alternatively, we may write
\begin{equation}
    \label{eq:dFdtLbis}
    \begin{split}
L_{n+1} \frac{\partial F^{(g)}_{L_1,\ldots,L_n}}{\partial t_{L_{n+1}}} & = \sum_{\ell_1,\ldots,\ell_n \geq 1}
    A_{L_1,\ell_1} \cdots A_{L_n,\ell_n}( \ell_1\cdots \ell_n)  L_{n+1} \frac{\partial T^{(g)}_{\ell_1,\ldots,\ell_n}}{\partial t_{L_{n+1}}}
    \\ & +\sum_{i=1}^n \sum_{\ell_1,\ldots,\ell_n \geq 1}
    A_{L_1,\ell_1} \cdots \left( L_{n+1} \frac{\partial A_{L_i,\ell_i}}{\partial t_{L_{n+1}}} \right)\cdots A_{L_n,\ell_n}( \ell_1\cdots \ell_n) T^{(g)}_{\ell_1,\ldots,\ell_n}\ .\\
    \end{split}
  \end{equation}
From \eqref{eq:dAdtM}, the general term in the final sum over $i$ may be written as 
\begin{equation}
\begin{split}
\sum_{\ell_1,\ldots,\ell_n \geq 1}
    A_{L_1,\ell_1} \cdots  \Bigg( \sum_{m_i>\ell_i}\sum_{\ell_{n+1}\geq 1} A_{L_i,m_i} & \, A_{L_{n+1},\ell_{n+1}}m_i\, \ell_{n+1} T^{(0)}_{m_i,\ell_{n+1}|\ell_i}   \Bigg) \cdots A_{L_n,\ell_n}\\
    & \qquad \qquad \qquad \qquad \times ( \ell_1\cdots \ell_n) T^{(g)}_{\ell_1,\ldots, \ell_n}
\\
= \sum_{\ell_1,\ldots,\ell_{n+1} \geq 1} A_{L_1,\ell_1} \cdots  A_{L_{n+1},\ell_{n+1}}( \ell_1\cdots \ell_{n+1}) & \left( \sum_{m_i<\ell_i}  m_i\, T^{(0)}_{\ell_i,\ell_{n+1}|m_i} T^{(g)}_{\ell_1,\ldots,m_i,\ldots \ell_n} \right)
\end{split}
\end{equation}
where, to go from the first to the second line, we exchanged the dummy summation variables $\ell_i $ and $m_i$, as well as the order of summation.  As for the first term
on the right hand side of  \eqref{eq:dFdtLbis}, using \eqref{eq:9}, it
may be rewritten as 
\begin{equation}
  \sum_{\ell_1,\ldots,\ell_{n+1} \geq 1} A_{L_1,\ell_1} \cdots
  A_{L_{n+1},\ell_{n+1}}( \ell_1\cdots \ell_{n+1})
  D_{\ell_{n+1}}T^{(g)}_{\ell_1,\ldots,\ell_n}\ .
\end{equation}
Gathering all the terms in \eqref{eq:dFdtLbis}, comparing with \eqref{eq:dFdtL}, and noting that
we may incorporate Proposition~\ref{prop:recur0} upon defining 
\begin{equation}
\label{eq:D0def}
D_0:=\frac{\partial}{\partial t}\, ,
\end{equation} 
we arrive at the following identification: 
\begin{prop}\label{prop:recurgen}  Let $g,n$ be nonnegative
  integers. We have, for $\chi=2-2g-n\leq 0$ and
  $\ell_1,\ldots,\ell_{n+1}$ nonnegative integers, 
\begin{equation}\label{eq:recurface}
T^{(g)}_{\ell_1,\ldots,\ell_n,\ell_{n+1}}= D_{\ell_{n+1}}T^{(g)}_{\ell_1,\ldots,\ell_n}+\sum_{i=1}^n \sum_{m_i<\ell_i}  m_i\, T^{(0)}_{\ell_i,\ell_{n+1}|m_i} T^{(g)}_{\ell_1,\ldots,m_i,\ldots \ell_n}\ .
\end{equation}
If $\ell_{n+1}$ is even, this reads
\begin{equation}\label{eq:recurfaceeven}
\begin{split}
T^{(g)}_{\ell_1,\ldots,\ell_n,\ell_{n+1}}&=D_{\ell_{n+1}}T^{(g)}_{\ell_1,\ldots,\ell_n}
+\sum_{i=1}^n \bigg( \sum_{m_i=1 \atop m_i=\ell_i-1\!\!\!\!\mod 2}^{\ell_i-1}\!\!\!\!\!\!\!\! m_i\, R^{\frac{\ell_i+\ell_{n+1}-m_i-1}{2}}\frac{\partial  S}{\partial t}T^{(g)}_{\ell_1,\ldots,m_i,\ldots,\ell_n}\\
&\hskip 3.3cm + \sum_{m_i=1 \atop m_i=\ell_i-2\!\!\!\!\mod 2}^{\ell_i-2}\!\!\!\!\!\!\!\! m_i\, R^{\frac{\ell_i+\ell_{n+1}-m_i}{2}-1}\frac{\partial  R}{\partial t}T^{(g)}_{\ell_1,\ldots,m_i,\ldots,\ell_n}\bigg)\, ,\\
\end{split}
\end{equation}
while if $\ell_{n+1}$ is odd, it reads instead
\begin{equation}\label{eq:recurfaceodd}
\begin{split}
T^{(g)}_{\ell_1,\ldots,\ell_n,\ell_{n+1}}&=D_{\ell_{n+1}}T^{(g)}_{\ell_1,\ldots,\ell_n}
+\sum_{i=1}^n \bigg( \sum_{m_i=1 \atop m_i=\ell_i-1\!\!\!\!\mod 2}^{\ell_i-1}\!\!\!\!\!\!\!\! m_i\, R^{\frac{\ell_i+\ell_{n+1}-m_i}{2}-1}\frac{\partial  R}{\partial t}T^{(g)}_{\ell_1,\ldots,m_i,\ldots,\ell_n}\\
&\hskip 3.3cm + \sum_{m_i=1 \atop m_i=\ell_i-2\!\!\!\!\mod 2}^{\ell_i-2}\!\!\!\!\!\!\!\! m_i\, R^{\frac{\ell_i+\ell_{n+1}-m_i-1}{2}}\frac{\partial  S}{\partial t}T^{(g)}_{\ell_1,\ldots,m_i,\ldots,\ell_n}\bigg)\, .\\
\end{split}
\end{equation}
For $n=0$, these formulas read
$T^{(g)}_{\ell_1}=D_{\ell_1}T^{(g)}_\varnothing=D_{\ell_1}F^{(g)}$. 
\end{prop}

Again, the identity \eqref{eq:recurface} has a clear combinatorial
interpretation. Consider a map counted by
$T^{(g)}_{\ell_1,\ldots,\ell_{n+1}}$. Then, there exists at most one $i\in
\{1,\ldots,n\}$ such that there exists a closed path of length
$m_i<\ell_i$ separating the 
$i$-th and the $(n+1)$-th boundaries from the others, for a reason
similar to that illustrated in Figure \ref{fig:exclusion}. If exactly
one such $i$ exists, we can consider the extremal shortest separating closed
path and cut along it. We get, on the one hand, a map counted by
$T^{(0)}_{\ell_i,\ell_{n+1}|m_i}$, i.e.\ a pair of pants with two
tight boundaries of lengths $\ell_i$ and $\ell_{n+1}$, and one
strictly tight boundary of length $m_i$, and, on the other hand, a map
counted by
$T^{(g)}_{\ell_1,\ldots,m_i,\ldots,\ell_n}$. There
are $m_i$ ways to glue back these maps together. 
If no such
$i$ exists, the $(n+1)$-th boundary can be changed into an inner face
without affecting the tightness of the other boundaries. We deduce
that such maps are counted by
$D_{\ell_{n+1}}T^{(g)}_{\ell_1,\ldots,\ell_n}$, which explains, in
retrospect, the name of the operator $D_m$. 

\begin{rem}
  For $g=0$ and $n=2$, notice that Proposition \ref{prop:recurgen}
  relates the generating function of tight cylinders
  $T^{(0)}_{\ell_1,\ell_2}$, given by \eqref{eq:specialcyl} for
  $\ell_1,\ell_2>0$, and $T^{(0)}_{0,0}=\ln(R/t)$, with that of tight
  pairs of pants given by Theorem \ref{thm:triskellnonbip} below.
\end{rem}

\subsection{Consistency with the expression for the planar bipartite case}\label{sec:consistency}

In this section we take $n\geq 3$ and check that the explicit expression \eqref{eq:Tgen} for
$\bip{T^{(0)}_{2\ell_1,\ldots,2\ell_n}}$ in the planar bipartite case
is consistent with the
above relations~\eqref{eq:recur0} and \eqref{eq:recurface} describing the addition of an extra boundary-vertex and an extra boundary-face respectively.

\paragraph{Adding an extra boundary-vertex.}

The relation~\eqref{eq:recur0} of Proposition~\ref{prop:recur0} reduces in the planar bipartite case to
\begin{multline}
  \label{eq:recurgen0}
    \bip{T^{(0)}_{2\ell_1,\ldots,2\ell_n,0}} = \frac{\partial}{\partial t} \bip{T^{(0)}_{2\ell_1,\ldots,2\ell_n}} + \\ \sum_{i=1}^n
    \sum_{0<m_i<\ell_i} (2m_i) R^{\ell_i-m_i-1} R' \bip{T^{(0)}_{2\ell_1,\ldots,2m_i,\ldots,2\ell_n}}
\end{multline}
with $R$ given by \eqref{eq:Rbipeq}. 
Let us see how this relation is indeed satisfied by the expression \eqref{eq:Tgen} for
$\bip{T^{(0)}_{2\ell_1,\ldots,2\ell_n}}$.

By~\eqref{eq:Tgen}, we have indeed
  \begin{equation}
   \bip{T^{(0)}_{2\ell_1,\ldots,2\ell_n,0}} = R^{\ell_1+\cdots+\ell_n} \sum_{k=0}^{n-2} k! p_k(\ell_1,\ldots,\ell_n)  b_{n-1,k+1}+ \frac{(-1)^{n+1} (n-2)!}{t^{n-1}} \delta_{\ell_1+\cdots+\ell_n,0}
  \end{equation}
  while
  \begin{multline}
    \frac{\partial}{\partial t} \bip{T^{(0)}_{2\ell_1,\ldots,2\ell_n} }= 
    R^{\ell_1+\cdots+\ell_n-1}  \sum_{k=0}^{n-3}  k! p_k(\ell_1,\ldots,\ell_n) \times \\
    \left( (\ell_1+\cdots+\ell_n-k-1 ) R' b_{n-2,k+1} + 
      \sum_{j\geq 1}  R^{(j+1)} b_{n-2,k+1}^{[j]} \right)\\
      + \frac{(-1)^{n+1} (n-2)!}{t^{n-1}} \delta_{\ell_1+\cdots+\ell_n,0}\ ,
  \end{multline}
  where $R^{(i)}$ denotes the $i$-th derivative of $R$ with respect to $t$ while $b_{n,k}^{[j]}$ is the derivative of the Bell polynomial
  $B_{n,k}$ with respect to its $j$-th variable, evaluated at
  $(R'/R,R''/R,\ldots)$. 
  By the recursion relation
  \begin{equation}
    b_{n-1,k+1} = \frac{R'}{R} b_{n-2,k} + \sum_{j \geq 1}
    \frac{R^{(j+1)}}R b_{n-2,k+1}^{[j]}\ , 
  \end{equation}
  which is easily derived from the combinatorial interpretation of the
Bell polynomials in terms of partitions (with the convention that
$b_{n,k}=0$ if $k>n$),  we deduce that
  \begin{multline}
    \label{eq:Tsubs}
    \bip{T^{(0)}_{2\ell_1,\ldots,2\ell_n,0}} - \frac{\partial}{\partial t} \bip{T^{(0)}_{2\ell_1,\ldots,2\ell_n}} =
    R^{\ell_1+\cdots+\ell_n-1} R' \times \\
    \sum_{k=0}^{n-2} k! p_k(\ell_1,\ldots,\ell_n) \left( b_{n-2,k} - (\ell_1+\cdots+\ell_n-k-1 ) b_{n-2,k+1} \right).
  \end{multline}
  On the other hand we have
  \begin{multline}
   \label{eq:Tsubsbis}
    \sum_{i=1}^n
    \sum_{0<m_i<\ell_i} (2m_i) R^{\ell_i-m_i-1} R' \bip{T^{(0)}_{2\ell_1,\ldots,2m_i,\ldots,2\ell_n}} = R^{\ell_1+\cdots+\ell_n-1} R' \times \\
    \sum_{k=0}^{n-2} \sum_{i=1}^n \sum_{0<m_i<\ell_i} (2m_i) k!
    p_k(\ell_1,\ldots,m_i,\ldots,\ell_n) b_{n-2,k+1}.
  \end{multline}
 Comparing \eqref{eq:Tsubs} and \eqref{eq:Tsubsbis}, we see that the relation~\eqref{eq:recurgen0} is then a direct consequence of the following
  recursion relation for $p_k(\ell_1,\ldots,\ell_n)$:
  \begin{prop} (String equation) For nonnegative integers $\ell_1,\ldots,\ell_n$, we have
  \begin{multline}
  \label{eq:recpk}
  (k+1)p_{k+1}(\ell_1,\ldots,\ell_n)=(\ell_1+\cdots+\ell_n-k-1)p_k(\ell_1,\ldots,\ell_n)\\+\sum_{i=1}^n \sum_{0<m_i<\ell_i} (2m_i)p_k(\ell_1,\ldots,m_i,\ldots,\ell_n)\ .
  \end{multline}
  \end{prop}
This latter recursion is proved in \cite[Proposition 2.2]{polytightmaps}.

\paragraph{Adding an extra boundary-face of even length.}

We now wish to check the consistency of the expression \eqref{eq:Tgen} for
$\bip{T^{(0)}_{2\ell_1,\ldots,2\ell_n}}$ with the recursion relation~\eqref{eq:recurface} of Proposition~\ref{prop:recurgen} which,
in the planar bipartite case, reduces to 
  \begin{multline}
    \label{eq:recurgenm}
   \bip{T^{(0)}_{2\ell_1,\ldots,2\ell_n,2\ell_{n+1}}} = D_{2\ell_{n+1}} \bip{T^{(0)}_{2\ell_1,\ldots,2\ell_n}}+ \\
    \sum_{i=1}^n
    \sum_{0<m_i<\ell_i} (2m_i) R^{\ell_i-m_i+\ell_{n+1}-1} R' \bip{T^{(0)}_{2\ell_1,\ldots,2m_i,\ldots,2\ell_n}}.
  \end{multline}
As a prerequisite to prove \eqref{eq:recurgenm} from \eqref{eq:Tgen},  we note that $D_{2m}$ satisfies
the Leibniz product rule, namely $D_{2m}(XY)=(D_{2m} X)Y+X(D_{2m} Y)$, and moreover that
the action of $D_{2m}$ on derivatives of $R$ is given by
\begin{equation}
  \label{eq:DmRn}
  D_{2m} R^{(j)} = T^{(0)}_{2m,2,\!\!\!\!\underbrace{\scriptstyle 0,\ldots,0}_{j+1 \text{ times}}}\Big\vert_{\mathrm{bip}} =  R^{m+1} \sum_{k=0}^j k! q_k(m) b_{j+1,k+1}.
\end{equation}
Here, we may justify the first equality by noting that, in the
bipartite case, $R^{(j)}$ can be seen as the generating function of
maps with one unrooted boundary-face of length $2$ and $j+1$ boundary-vertices.
Since $D_{2m}$ creates an extra tight boundary-face of length $2m$, 
and since the other
boundaries (which have lengths $2$ and $0$ only) are automatically tight, 
we indeed obtain the maps counted by $T^{(0)}_{2m,2,\!\!\!\!\underbrace{\scriptstyle 0,\ldots,0}_{j+1 \text{ times}}}\Big\vert_{\mathrm{bip}}$. As for the second equality, it is a particular instance 
of \eqref{eq:Tgen}, simplified via the identity $p_k(m,1,0,\ldots,0)=p_k(m,1)=p_k(m)+p_{k-1}(m)=q_k(m)$.

By \eqref{eq:Tgen}, we now have
  \begin{multline}
  \label{eq:comp1}
    \bip{T^{(0)}_{2\ell_1,\ldots,2\ell_n,2\ell_{n+1}}} = R^{\ell_1+\cdots+\ell_n+\ell_{n+1}} \sum_{k=0}^{n-2} k !p_k(\ell_1,\ldots,\ell_n,\ell_{n+1})  b_{n-1,k+1} \\
    = R^{\ell_1+\cdots+\ell_n+\ell_{n+1}} \sum_{k' \geq 0} \sum_{k'' \geq 0} (k'+k'')! p_{k''}(\ell_1,\ldots,\ell_n) q_{k'}(\ell_{n+1}) b_{n-1,k'+k''+1}
  \end{multline}
  (throughtout this computation, we will leave the summation ranges as
  implicit as possible, since they are naturally enforced by the
  vanishing of the summands). On the other hand, using \eqref{eq:DmRn}, we have
  \begin{multline}
    D_{2\ell_{n+1}} \bip{T^{(0)}_{2\ell_1,\ldots,2\ell_n}}=
    R^{\ell_1+\cdots+\ell_n-1}  \sum_{k \geq 0} k! p_{k}(\ell_1,\ldots,\ell_n) \times \\ \allowdisplaybreaks
    \left( (\ell_1+\cdots+\ell_n-k-1 ) (D_{2\ell_{n+1}} R) b_{n-2,k+1} +
      \sum_{j\geq 1}  (D_{2\ell_{n+1}} R^{(j)}) b_{n-2,k+1}^{[j]} \right) \\ 
    = R^{\ell_1+\cdots+\ell_n+\ell_{n+1}} \sum_{k \geq 0} k! p_{k}(\ell_1,\ldots,\ell_n) \times \hfill \\
    \left( (\ell_1+\cdots+\ell_n-k-1 ) \frac{R'}{R} b_{n-2,k+1} +
      \sum_{j\geq 1} \sum_{k' \geq 0} k'! q_{k'}(\ell_{n+1}) b_{j+1,k'+1}
      b_{n-2,k+1}^{[j]} \right) \label{eq:DTexp}\ .
  \end{multline}
  To compare the
  last two displays \eqref{eq:comp1} and \eqref{eq:DTexp}, we may write
  \begin{multline}
    \binom{k'+k''}{k'} b_{n-1,k'+k''+1} = \sum_{j \geq 0} \binom{n-2}j b_{j+1,k'+1} b_{n-2-j,k''} \\
    = \frac{R'}{R} \delta_{k',0} b_{n-2,k''}+ \sum_{j \geq 1} b_{j+1,k'+1} b_{n-2,k''+1}^{[j]}
  \end{multline}
  upon interpreting the lhs as a sum over partitions of
  $\{1,\ldots,n-1\}$ into $k'+k''+1$ blocks, $k'+1$ of which are marked,
  the element $n-1$ being always in a marked block: the first equality
  is obtained by summing over the number $j$ of elements other than
  $n-1$ that are in a marked block, and the second equality is
  obtained by noting that
  $\binom{n-2}j b_{n-2-j,k''} = b_{n-2,k''+1}^{[j]}$ for $j \geq 1$.
  We deduce that the sum appearing at the end of~\eqref{eq:DTexp},
  with $k$ replaced by $k''$, can be rewritten as
  \begin{multline}
    \sum_{j\geq 1} \sum_{k' \geq 0} k'! q_{k'}(\ell_{n+1}) b_{j+1,k'+1}
    b_{n-2,k''+1}^{[j]} = \sum_{k' \geq 0} \binom{k'+k''}{k'}
    k'! q_{k'}(\ell_{n+1}) b_{n-1,k'+k''+1}\\ - \frac{R'}{R} b_{n-2,k''}
  \end{multline}
  and hence
  \begin{multline}
    \bip{T^{(0)}_{2\ell_1,\ldots,2\ell_n,2\ell_{n+1}}} - D_{2\ell_{n+1}} \bip{T^{(0)}_{2\ell_1,\ldots,2\ell_n}} =
    R^{\ell_1+\cdots+\ell_n+\ell_{n+1}-1} R' \times \\
    \sum_{k \geq 0} k! p_{k}(\ell_1,\ldots,\ell_n) \left( b_{n-2,k} - (\ell_1+\cdots+\ell_n-k-1 ) b_{n-2,k+1} \right).
  \end{multline}
  By comparing with~\eqref{eq:Tsubs} we see that
  \begin{equation}
    \bip{T^{(0)}_{2\ell_1,\ldots,2\ell_n,2\ell_{n+1}}} - D_{2\ell_{n+1}} \bip{T^{(0)}_{2\ell_1,\ldots,2\ell_n}} =
    R^{\ell_{n+1}} \left( \bip{T^{(0)}_{2\ell_1,\ldots,2\ell_n,0}} - \frac{\partial}{\partial t} \bip{T^{(0)}_{2\ell_1,\ldots,2\ell_n}} \right)
  \end{equation}
  and we immediately deduce  that \eqref{eq:recurgenm} is again a  direct consequence of the string equation \eqref{eq:recpk}.
 
 \begin{rem}
We may similarly check that the explicit expression \eqref{eq:Tgenquasi} for
$\bip{T^{(0)}_{2\ell_1,\ldots,2\ell_n}}$ in the planar quasi-bipartite case is also consistent with 
the relations~\eqref{eq:recur0} and \eqref{eq:recurface} for the addition of an extra boundary.
This property is now a direct consequence of the string equation relating $\tilde{p}_{k+1}(\ell_1,\ell_2;\ell_3,\ldots,\ell_n)$
to $\tilde{p}_k(\ell_1,\ell_2;\ell_3,\ldots,\ell_n)$, see \cite[Proposition 2.10]{polytightmaps}.
\end{rem}

\section{The quasi-polynomiality phenomenon}
\label{sec:quasipol}

We have seen in Section \ref{sec:bipcase} that the generating function
of planar bipartite maps with (at least three) tight boundaries is an
even polynomial in the boundary lengths times a power of the basic
generating function $R$.  As mentioned in the introduction, this is a
particular case of a more general quasi-polynomiality
phenomenon.

Precisely, given a positive integer $n$, let us call
\emph{parity-dependent quasi-polynomial in $n$ variables} a function
$f$ of $n$ integer variables, such that there exists a (necessarily
unique) family of polynomials
$(f_I)_I$ indexed by subsets $I$ of $\{1,\ldots,n\}$, such that
$f(m_1,\ldots,m_n)=f_{I}(m_1,\ldots,m_n)$ whenever the $m_i$ with
$i\in I$ are odd integers, and the $m_i$ with $i\notin I$ are even
integers.
By \emph{coefficients} of $f$, we mean the collection of coefficients
of the polynomials $f_I$, which may belong to an arbitrary ring;
by \emph{total degree} of $f$, we mean the maximum of the total
degrees of the $f_I$.  It is convenient to allow the value $n=0$,
in which case it is understood that $f$ is just a constant element of
the ring at hand.
A parity-dependent quasi-polynomial is said \emph{symmetric} if it is invariant under any permutation of its variables.

In order to state the main result of this section, we also introduce
the quantity
\begin{equation}\label{eq:chimgn}
  \chi(\mathcal{M}_{g,n}):=
  \begin{cases}
    (-1)^{n-1} (n-3)! & \text{if $g=0$,}\\
    (-1)^{n-1} \frac{(2g+n-3)!}{(2g-2)!} \zeta(1-2g) & \text{if $g>0$,}
  \end{cases}
\end{equation}
which, for $2-2g-n<0$, is the Euler characteristic of the moduli space
$\mathcal{M}_{g,n}$ \cite{Harer1986}.

\begin{thm}
  \label{thm:quasipol}
  Let $g\geq 0,n\geq 1$ be integers such that $2-2g-n<0$. Then, there
  exists a symmetric parity-dependent quasi-polynomial $\mathfrak{T}^{(g)}_n$ in
  $n$ variables, of total degree $3g-3+n$, such that
  for any nonnegative integers $\ell_1,\ldots,\ell_n$, we
  have  
  \begin{equation}
    \label{eq:quasipol}
    T^{(g)}_{\ell_1,\ldots,\ell_n}=
    R^{\frac{\ell_1+\cdots+\ell_n}{2}}\mathfrak{T}^{(g)}_n(\ell_1^2,\ldots,\ell_n^2) - \chi(\mathcal{M}_{g,n}) t^{2-2g-n} \delta_{\ell_1+\cdots+\ell_n,0}\, .
  \end{equation}
  The coefficients of $\mathfrak{T}^{(g)}_n$ are rational functions
  with rational coefficients of the quantities $\frac{R^{(k)}}{R}$ and
  $\frac{S^{(k)}}{\sqrt{R}}$, $k=1,\ldots,3g-2+n$, where
  $R^{(k)},S^{(k)}$ denote the $k$-th derivatives of the series $R,S$
  with respect to the vertex weight $t$.
\end{thm}

Recall from Remark \ref{rem:specialcyl} that there is no
quasi-polynomiality phenomenon for $(g,n)=(0,2)$. 
The quasi-polynomiality phenomenon was first observed in the language
of topological recursion by Norbury and Scott \cite{NoSc13}, who
showed that the coefficients $\hat{\tau}^{(g)}_{\ell_1,\ldots,\ell_n}$
appearing in the series expansion \eqref{eq:omegaexp} of
$\omega^{(g)}_n(z_1,\ldots,z_n)$ are quasi-polynomials in
$\ell_1,\ldots,\ell_n$ which are odd in each variable. We review their
result in Section \ref{sec:toprec}: combining it with the trumpet
decomposition yields~\eqref{eq:quasipol} in the case where $n$
and all $\ell_i$ are non-zero.

However, the quasi-polynomiality of
$\hat{\tau}^{(g)}_{\ell_1,\ldots,\ell_n}$ does not imply a priori that
of $\tau^{(g)}_{\ell_1,\ldots,\ell_n}$, as there may be pathologies
when \emph{some} variables are set to $0$.  To circumvent this
problem, and show that $\tau^{(g)}_{\ell_1,\ldots,\ell_n}$ indeed
coincides with an even quasi-polynomial except when \emph{all}
variables are set to $0$, we will make use in Section
\ref{sec:quasipolvia} of the tight boundary insertion relations
established in Section~\ref{sec:recrel}. The proof is by induction on
the number of boundaries $n$, and we still use results from
topological recursion to initialise the induction.

\begin{rem}
  \label{rem:quasipoln0}
  In fact, Theorem~\ref{thm:quasipol}  also makes sense for $n=0$,
  and states that for any $g\geq
2$, the generating function $T^{(g)}_\varnothing=F^{(g)}$
  of maps of genus $g$ without boundaries is equal to  
  \begin{equation}
    F^{(g)} = \mathfrak{F}^{(g)} - \chi(\mathcal{M}_{g,0}) t^{2-2g}\, ,
  \end{equation}
  with $\mathfrak{F}^{(g)}$ a rational function with rational coefficients
  of $\frac{R^{(k)}}{R},\frac{S^{(k)}}{\sqrt{R}}$,
  $k=1,\ldots,3g-2$. As we will see below, this result is obtained by
  combining the property known from topological recursion that $F^{(g)}$
  can be expressed in terms of the so-called moments, with an
  expression of the moments in terms of the derivatives of $R$ and
  $S$. In the case $g=1$, we no longer have rational functions
  but rather logarithms, see Theorem \ref{sec:torus-1}. 
\end{rem}

\begin{rem}
  \label{rem:zeroweights}
  When setting the weights $t_1,t_2,\ldots$ for inner faces to zero,
  i.e.\ when we consider \emph{tight maps} for which every face is a
  tight boundary, then
  $\mathfrak{T}^{(g)}_n(\ell_1^2,\ldots,\ell_n^2)$ reduces to
  $t^{2-2g-n}$ times Norbury's lattice count (quasi-)polynomial
  $N_{g,n}(\ell_1,\ldots,\ell_n)$ \cite{Norbury2010}. It is shown in
  this reference that
  $N_{g,n}(0,\ldots,0)=\chi(\mathcal{M}_{g,n})$. This is consistent,
  in view of~\eqref{eq:quasipol}, with the fact that
  $T^{(g)}_{0,\ldots,0}$ vanishes when we set the weights for inner
  faces to zero (there are no tight maps having only
  boundary-vertices and no faces).
\end{rem}

\subsection{Maps without boundary-vertices}\label{sec:toprec}

In this section, we establish the quasi-polynomiality
property~\eqref{eq:quasipol} in the case where $n$ and all $\ell_i$
are non-zero. We will do so by combining
Theorem~\ref{thm:enum-coroll-1} with known results coming from the
theory of topological recursion.  One reason behind the fact that this
theory is formulated in terms of $\omega^{(g)}_n$ instead of
$W^{(g)}_n$ is that the former are rational functions of their $n$
variables with a simple pole structure, see for instance \cite[Section
3.3.1]{Eynard2016}. It was recognized by Norbury and Scott that this
simple pole structure amounts to the following proposition:

\begin{prop}
  \label{prop:quasipol}
  For $g \geq 0$, $n \geq 1$ and $2-2g-n<0$, the
  coefficients $\hat{\tau}^{(g)}_{\ell_1,\ldots,\ell_n}$ appearing in
  the expansion~\eqref{eq:omegaexp} of $\omega^{(g)}_n$ are of the
  form
  \begin{equation}
    \label{eq:hattauquasipol}
    \hat{\tau}^{(g)}_{\ell_1,\ldots,\ell_n} = \ell_1 \cdots \ell_n
    \mathfrak{T}^{(g)}_n(\ell_1^2,\ldots,\ell_n^2)
  \end{equation}
  with $\mathfrak{T}^{(g)}_n$ a parity-dependent quasi-polynomial in
  $n$ variables of total degree $3g-3+n$.  
\end{prop}

The above proposition is a reformulation of~\cite[Theorem
1]{NoSc13}\footnote{Precisely, this theorem considers the expansion of
  $\omega^{(g)}_n(z_1,\ldots,z_n)$ around $z_i=0$. Here, our
  $\hat{\tau}_{\ell_1,\ldots,\ell_n}^{(g)}$ are rather defined in
  terms of the expansion around $z_i=\infty$. But the two expansions
  can be related through antisymmetry properties of $\omega^{(g)}_n$,
  see \eqref{eq:51}.}, in which the quantity
$\mathfrak{T}^{(g)}_n(\ell_1^2,\ldots,\ell_n^2)$ is denoted by
$N^g_n(\ell_1,\ldots,\ell_n)$. Note that we do not characterize the
structure of the coefficients of $\mathfrak{T}^{(g)}_n$ at this
stage. The relation~\eqref{eq:hattauquasipol} implies that
\eqref{eq:quasipol} holds when $n$ and all $\ell_i$ are non-zero, by
\eqref{eq:ttaurel} and \eqref{eq:ThatT}.

For completeness, let us sketch a proof of
Proposition~\ref{prop:quasipol} based on the results available in
\cite{Eynard2016}. We start with a technical lemma.

\begin{lem}\label{sec:maps-with-bound}
  Let $P(n)$ be a polynomial of degree at most $2d+1$, which is equivalent to the 
  series $F(z) := \sum_{n \geq 0} \frac{P(n)}{z^{n+1}}$ being a rational
  function of the form $N(z)/(z-1)^{2d+2}$, with $N$ a polynomial of
  degree at most $2d+1$. Then, the following properties are
  equivalent:
  \begin{itemize}
  \item[(i)] $P$ is odd,
  \item[(ii)] $F$ is such that $F(z)=z^{-2} F(z^{-1})$,
  \item[(iii)] $N$ satisfies $N(z)=z^{2d}N(z^{-1})$, i.e. it is a
    self-reciprocal (palindromic) polynomial.
  \end{itemize}
\end{lem}

\begin{proof}
  By~\cite[Proposition~4.2.3]{StanleyEC1}\footnote{Short rederivation:
    we consider the vector space $\mathcal{L}$ of formal Laurent
    series of the form $\sum_{n \in \Z} a_n z^n$.  We cannot multiply
    two elements of this space in general, but we can multiply an
    element of $\mathcal{L}$ by a polynomial in $z$, and this
    operation is linear. Let us consider in particular the series
    $\sum_{n \in \Z} \frac{P(n)}{z^{n+1}} \in \mathcal{L}$, and the
    polynomial $(z-1)^{2d+2}$. Their product vanishes since $P$ is a
    polynomial of degree at most $2d+1$; hence, it is annihilated by the
    $(2d+2)$-th power of the forward difference operator. We deduce
    that
    $(z-1)^{2d+2} \sum_{n \geq 0} \frac{P(n)}{z^{n+1}} = -
    (z-1)^{2d+2} \sum_{n \leq -1} \frac{P(n)}{z^{n+1}}$. The left-hand
    side may be interpreted as a product in the ring
    $\mathbb{C}((z^{-1}))$ and is thus equal to
    $N(z) \in \mathbb{C}[z]$. Therefore the right-hand side is also
    equal to $N(z)$, and we obtain an identity valid in
    $\mathbb{C}[[z]]$. Multiplying by the inverse of $(z-1)^{2d+2}$ in
    $\mathbb{C}[[z]]$ (which differs from that in
    $\mathbb{C}((z^{-1}))$!) and changing $n$ in $-n$, we obtain the
    wanted series identity.}, the rational function $F(z)$ admits
  around $0$ the expansion $F(z) = \allowbreak - \sum_{n \geq 1} P(-n) z^{n-1}$
  hence we have
  \begin{equation}
    z^{-2} F(z^{-1}) = - \sum_{n \geq 1} \frac{P(-n)}{z^{n+1}}. 
  \end{equation}
  This implies the equivalence between (i) and (ii): indeed, if $P$ is
  odd, then
  $z^{-2}F(z^{-1})=\sum_{n\geq 1}z^{-n-1}P(n)=F(z)$, since $P(0)=0$,
  and conversely, if $F(z)=z^{-2}F(z^{-1})$ then by comparing the
  expansions, we obtain $P(0)=0$ and $-P(-n)=P(n)$ for every $n\geq
  1$, implying that $P$ is odd. 
  The equivalence between (ii) and (iii) is straightforward.
\end{proof}

\begin{proof}[Proof of Proposition~\ref{prop:quasipol}]
  According to~\cite[Lemma~3.3.3 page 89,
  Theorem~3.5.1 page 121, and last equation on page
  140]{Eynard2016}\footnote{There seems to be a discrepancy between
    Theorem~3.5.1 and the equation on page 140: $\sum d_i$ is at most
    $6g-6+2n$ in the former, and $6g-4+2n$ in the latter. We assume
    that the former is correct since it holds for pairs of pants ($g=0,n=3$) and
    lids ($g=1,n=1$).}, $\omega^{(g)}_n$ has a partial fraction decomposition of
  the form
  \begin{equation}
    \label{eq:omegaexpeyn}
    \omega^{(g)}_n(z_1,\ldots,z_n) =
    \sum_{\epsilon \in \{+1,-1\}^n}
    \sum_{\substack{d \in \mathbb{N}^n \\ d_1 + \cdots + d_n \leq 6g-6+2n}}
    \frac{c_{\epsilon,d}}{(z_1-\epsilon_1)^{d_1+2} \cdots (z_n-\epsilon_n)^{d_n+2}}
  \end{equation}
  where the $c_{\epsilon,d}$ are some coefficients depending on the
  spectral curve. By the negative binomial formula, we have
  \begin{equation}
  \frac{1}{(z-\epsilon)^{d+2}}=\sum_{\ell\geq 1}\binom{\ell}{d+1} \frac{\epsilon^{\ell-d-1}}{z^{\ell+1}}
  . 
  \end{equation}
  Viewing
  $\binom{\ell}{d+1}$ as a polynomial in $\ell$, this implies that, in the
  expansion~\eqref{eq:omegaexp} of $\omega^{(g)}_n$, the $\hat{\tau}$
  coefficients are of the form
  \begin{equation}
    \label{eq:tauepsilon}
    \hat{\tau}^{(g)}_{\ell_1,\ldots,\ell_n} = \sum_{\epsilon \in \{+1,-1\}^n}
    Q^{(g,\epsilon)}(\ell_1,\ldots,\ell_n) \epsilon_1^{
     \ell_1} \cdots \epsilon_n^{\ell_n}.
  \end{equation}
  where $Q^{(g,\epsilon)}$ is a polynomial of degree at
  most
  $6g-6+3n$. It remains to check that these polynomials are indeed odd
  in each variable, so that we
  may write, for some polynomials $P^{(g,\epsilon)}$ of degree at
  most $3g-3+n$, 
  \begin{equation}
    \label{eq:61}
      Q^{(g,\epsilon)}(\ell_1,\ldots,\ell_n)= \ell_1\cdots\ell_n\, 
      P^{(g,\epsilon)}(\ell_1^2,\ldots,\ell_n^2)\, .
    \end{equation}
   Thanks to the symmetry in $\ell_1,\ldots,\ell_n$, it suffices to check that $Q^{(g,\epsilon)}$ is odd in $\ell_1$.
   For this we use the antisymmetry property
  of~\cite[Lemma~3.3.2, page 88]{Eynard2016}, which entails that
  \begin{equation}\label{eq:51}
    \omega^{(g)}_n(z_1,z_2,\ldots,z_n) = z_1^{-2} \omega^{(g)}_n(z_1^{-1},z_2,\ldots,z_n).
  \end{equation}
  In the
  expansion~\eqref{eq:omegaexpeyn}, each term of the sum over
  $\epsilon$ must be separately invariant under the antisymmetry
  mapping $f(z_1) \mapsto z_1^{-2} f(z_1^{-1})$, since it corresponds to the contribution to $\omega^{(g)}_n$ 
  having poles at $z_1=\epsilon_1, \ldots, z_n=\epsilon_n$. 
  Such term corresponds to the term
  $Q^{(g,\epsilon)}(\ell_1,\ldots,\ell_n) \epsilon_1^{\ell_1} \cdots
 \epsilon_2^{\ell_2}$ in~\eqref{eq:tauepsilon} and we claim that
  $Q^{(g,\epsilon)}$ is odd in $\ell_1$. By
  the transformation $z_i \mapsto \epsilon_i z_i$, it suffices to
  consider the contribution of $\epsilon=(1,\ldots,1)$ to $\omega^{(g)}_n$, namely
  \begin{equation}
  F(z_1):=
   \sum_{\substack{d \in \mathbb{N}^n \\ d_1 + \cdots + d_n \leq 6g-6+2n}}
    \frac{c_{(1,\ldots,1),d}}{(z_1-1)^{d_1+2} \cdots (z_n-1)^{d_n+2}},
 \end{equation} 
 satisfying $F(z_1)=z_1^{-2}F(z_1^{-1})$. By Lemma~\ref{sec:maps-with-bound}, 
 it expands as $\sum\limits_{\ell_1\geq 1}\frac{P(\ell_1)}{z_1^{\ell_1+1}}$ with $P$ an odd polynomial in $\ell_1$,
 which implicitly depends on the variables $z_2,\ldots,z_n$. Further expanding in these extra variables, 
the coefficient of $\frac{1}{z_2^{\ell_2+1}\cdots z_n^{\ell_n+1}}$, which
is nothing but $Q^{(g,(1,\ldots,1))}(\ell_1,\ldots,\ell_n)$, must be odd in $\ell_1$ as wanted, hence the oddness
property is established.
As
  $m_i=\ell_i^2$ has the same parity as $\ell_i$, we conclude
  that~\eqref{eq:hattauquasipol} holds with
  \begin{equation}
    \mathfrak{T}^{(g)}_n(m_1,\ldots,m_n) =
    \sum_{\epsilon \in \{+1,-1\}^n}
    P^{(g,\epsilon)}(m_1,\ldots,m_2)
    \epsilon_1^{m_1} \cdots \epsilon_n^{m_n}\, ,
  \end{equation}
  which is indeed a parity-dependent quasi-polynomial in $n$ variables
  of total degree $3g-3+n$.
\end{proof}

\subsection{Maps with boundary-vertices}
\label{sec:quasipolvia}

The purpose of this section is to complete the proof of Theorem
\ref{thm:quasipol}, by including the case where some of the boundary
lengths $\ell_1,\ldots,\ell_n$ vanish, and by showing the structure
properties of the coefficients of the quasi-polynomial
$\mathfrak{T}^{(g)}_n$. We proceed by induction on $n$, by using the
recursion relations of Section~\ref{sec:recrel} for the induction
step. As a warmup, let us first consider the easier case of maps with
even face degrees. Recall the relation~\eqref{eq:tauT} between the
$T$'s and the $\tau$'s, and the notation $\bip{}$ from
Section~\ref{sec:bipcase} which indicates that we set the weights
$t_1,t_3,\dots$ for inner faces of odd degrees to zero.

\begin{prop}
  \label{prop:quasipolbip}
  Let $g,n$ be nonnnegative integers such that $2-2g-n<0$.  There
  exist
  $Q^{(g)}_n (x_1,\ldots,x_{3g-2+n};u_1,\ldots,u_n)\in
  \mathbb{Q}(x_1,\ldots,x_{3g-2+n})[u_1,\ldots,u_n]$ and a constant
  $c(g,n)\in \mathbb{Q}$ such that, for any integers
  $\ell_1,\ldots,\ell_n$, we have
  \begin{equation}
  \label{eq:54}
  \bip{\tau^{(g)}_{2\ell_1,\ldots,2\ell_n}}=Q^{(g)}_n\left(\frac{R'}{R},\ldots,\frac{R^{(3g-2+n)}}{R};\ell_1^2,\ldots,\ell_n^2\right)-c(g,n) t^{2-2g-n} \delta_{\ell_1+\cdots+\ell_n,0}\, .
\end{equation}
\end{prop}

\begin{proof}
  As mentioned above, we proceed by induction on $n$ for a fixed value
  of $g$. The induction step relies on the recursion relations of
  Propositions \ref{prop:recur0} and \ref{prop:recurgen} which, for
  $t_1=t_3=\cdots=0$ (so that $S=0$), reduce to
   \begin{multline}
    \label{eq:recurgenmbip}
   \bip{T^{(g)}_{2\ell_1,\ldots,2\ell_n,2\ell_{n+1}}} = D_{2\ell_{n+1}} \bip{T^{(g)}_{2\ell_1,\ldots,2\ell_n}}+ \\
    \sum_{i=1}^n
    \sum_{0<m_i<\ell_i} (2m_i) R^{\ell_i -m_i+\ell_{n+1}-1} R'
    \bip{T^{(g)}_{2\ell_1,\ldots,2m_i,\ldots,2\ell_n}}\, .
  \end{multline}
Here, $D_{m}$ is the tight boundary insertion operator defined at
\eqref{eq:Dmdefgen}, reading for $S=0$
\begin{equation}
  \label{eq:53}
  D_{2\ell}=
  \begin{cases}
      \sum_{i=1}^\ell(-1)^{\ell+i}\binom{\ell+i-1}{\ell-i}R^{\ell-i}\frac{\partial}{\partial
        t_{2i}} & \text{if $\ell> 0$,}\\
      \frac{\partial}{\partial t} & \text{if $\ell=0$.}
    \end{cases}
  \end{equation}
In the second line of \eqref{eq:recurgenmbip}, it is understood that, if $\ell_i=0$, then
the corresponding sum over $0<m_i<\ell_i$ vanishes. In terms of the $\tau$'s, this relation is rewritten as
\begin{multline}
  \label{eq:55}
   \bip{\tau^{(g)}_{2\ell_1,\ldots,2\ell_n,2\ell_{n+1}}} =
   \frac{D_{2\ell_{n+1}}\!\!\bip{\tau^{(g)}_{2\ell_1,\ldots,2\ell_n}}}{R^{\ell_{n+1}}}\\
   + 
     \frac{R'}{R}\sum_{i=1}^n\left(\ell_{i}\bip{\tau^{(g)}_{2\ell_1,\ldots,2\ell_n}}+
    \sum_{0<m_i<\ell_i} (2m_i)\!\!
    \bip{\tau^{(g)}_{2\ell_1,\ldots,2m_i,\ldots,2\ell_n}}\right)\, .
\end{multline}
Here, we have used the Leibniz product rule for the differential
operator $D_{2\ell_{n+1}}$, and the relation $D_{2\ell}R=R^\ell R'$
which is the case $j=0$ of \eqref{eq:DmRn}, recalling the
definition~\eqref{eq:bnkdef} of $b_{n,k}$.

Let us now sketch the rough idea of the induction, ignoring the
pathological term of \eqref{eq:54} for now:
\begin{itemize}
\item from the recursion relation~\eqref{eq:55}, it is relatively
  straightforward that if $\tau^{(g)}_{2\ell_1,\ldots,2\ell_n}$ is
  polynomial in $\ell_1^2,\ldots,\ell_n^2$, then
  $\tau^{(g)}_{2\ell_1,\ldots,2\ell_n,2\ell_{n+1}}$ is also polynomial
  in $\ell_1^2,\ldots,\ell_n^2$: it is clear for the first term since
  the linear operator $D_{2\ell_{n+1}}$ does not act on these
  variables, while the second term involves a summation which turns
  out to preserve the wanted polynomiality property,
  \item regarding the new variable $\ell_{n+1}$, we use 
    the assumption that the coefficients of
    $\tau^{(g)}_{2\ell_1,\ldots,2\ell_n}$ are rational in the
    derivatives of $R$: acting on this quantity with the differential
    operator $D_{2\ell_{n+1}}$ and using the chain rule, it follows
    from \eqref{eq:DmRn} that 
    $R^{-\ell_{n+1}} D_{2\ell_{n+1}}\tau^{(g)}_{2\ell_1,\ldots,2\ell_n}$
    is polynomial in $\ell_{n+1}^2$, also with rational coefficients
    in the derivatives of $R$. 
\end{itemize}
Let us now make this argument precise.
Assuming that the expression~\eqref{eq:54} holds at rank $n$, we plug
it into the right-hand side of \eqref{eq:55}, and analyse the two
terms separately. The second term is easier to deal with, since it
does not depend on $\ell_{n+1}$ nor receives a contribution from the
pathological term of \eqref{eq:54}. It equals
\begin{equation}
  \label{eq:56}
 \frac{R'}{R} \sum_{i=1}^n \left(\ell_i
   Q^{(g)}_n\left(\ldots,\ell_i^2,\ldots\right)+\sum_{0<m_i<\ell_i}(2m_i)
   Q^{(g)}_n\left(\ldots,m_i^2,\ldots\right)\right)\, ,
\end{equation}
where, for readability, we only display the $u_i$ variable of
$Q^{(g)}_n$. Now, it is an exercise (solved in Lemma
\ref{sec:discr-integr-lemma} below) that for every univariate
polynomial $\mathfrak{q}$, the quantity
$\ell\mathfrak{q}(\ell^2)+\sum_{0<m<\ell}(2m)\mathfrak{q}(m^2)$ is
again a polynomial in $\ell^2$. Therefore, the second term of
\eqref{eq:55} is of the wanted form. To deal with the first term, we
apply the chain rule to obtain
\begin{multline}
  \label{eq:57}
   \frac{D_{2\ell_{n+1}}\!\!\bip{\tau^{(g)}_{2\ell_1,\ldots,2\ell_n}}}{R^{\ell_{n+1}}}=\sum_{j=1}^{3g-2+n}\frac{D_{2\ell_{n+1}}\left(\frac{R^{(j)}}{R}\right)}{R^{\ell_{n+1}}}\frac{\partial Q^{(g)}_n\left(\ldots\right)}{\partial
   x_j}- \\
 c(g,n) \frac{D_{2\ell_{n+1}}t^{2-2g+n}}{R^{\ell_{n+1}}}
 \delta_{\ell_1+\cdots+\ell_n,0}\, .
\end{multline}
Here, the first term is clearly polynomial in
$\ell_1^2,\ldots,\ell_n^2$. Moreover, it is also a polynomial in
$\ell_{n+1}^2$ by \eqref{eq:DmRn}, which also entails that the
coefficients are rational in $R^{(k)}/R$ for $1\leq k\leq
3g-2+(n+1)$. Finally, the second term in the last display is zero
unless $\ell_1,\ldots,\ell_{n+1}$ all vanish, in which case it is
equal to
$-c(g,n) \frac{\partial}{\partial t}t^{2-2g-n}=-c(g,n+1) t^{2-2g-(n+1)}$,
where $c(g,n+1)=(2-2g+n)c(g,n)$.  This completes the proof of the
induction step: if the statement of Proposition~\ref{prop:quasipolbip}
is true for $(g,n)$, then it is also true for $(g,n+1)$.

It remains to initialise the induction. For $g=0$, we initialise the
induction at $n=3$ by using the explicit formula of \cite[Theorem
1.1]{triskell}, which is the bipartite version of Theorem
\ref{thm:triskellnonbip} in Appendix \ref{app:tightpants}.  For
$g\geq 1$, we rely on the consequences of the topological recursion that
are described in \cite{budd2020irreducible}. Assume first that
$g\geq 2$, in which case we initialise the induction at $n=0$: in view
of Remark~\ref{rem:quasipoln0}, we only have to check that the
generating function $\bip{F^{(g)}}=\bip{\tau^{(g)}_\varnothing}$ of
essentially bipartite maps of genus $g$ without boundary is a rational
function of $R^{(k)}/R,1\leq k\leq 3g-2$.  By \cite[Equations (31) and
(32)]{budd2020irreducible}, there exists a polynomial $\tilde{P}_g$
with rational coefficients in $3g-3$ variables
such that for $t=1$ we have
\begin{equation}
  \label{eq:58}
  \bip{F^{(g)}}[t=1]=\tilde{P}_g(0,\ldots,0)-\left(\frac{1}{\bar{M}_0}\right)^{2g-2}\tilde{P}_g\left(\frac{\bar{M}_1}{\bar{M}_0},\ldots,\frac{\bar{M}_{3g-3}}{\bar{M}_0}\right)\, ,
\end{equation}
where $\bar{M}_p$ are the so-called \emph{moments}, defined at
\cite[Equation (30)]{budd2020irreducible}. Using Euler's relation, we
can restore the dependence on $t$ by the formula
\begin{equation}
  \label{eq:74}
  \bip{F^{(g)}}=t^{2-2g}\bip{F^{(g)}}[t=1](t_2,tt_4,t^2t_6,\ldots)\, .
\end{equation}
Consequently,
letting $\bar{M}_p^{(0)}=t\bar{M}_p(t_2,tt_4,t^2t_6,\ldots)$, in
accordance with \cite[Equation (43)]{budd2020irreducible}, we have
\begin{equation}
  \label{eq:59}
  \bip{F^{(g)}}=t^{2-2g}\tilde{P}_g(0,\ldots,0)-\left(\frac{1}{\bar{M}_0^{(0)}}\right)^{2g-2}\tilde{P}_g\left(\frac{\bar{M}_1^{(0)}}{\bar{M}_0^{(0)}},\ldots,\frac{\bar{M}_{3g-3}^{(0)}}{\bar{M}_0^{(0)}}\right)\, .
\end{equation}
Finally, \cite[Lemma 9]{budd2020irreducible} shows that
$\bar{M}_0^{(0)}=R/R'$, while we have 
\begin{equation}
  \label{eq:66}
  \frac{\bar{M}_p^{(0)}}{\bar{M}_0^{(0)}}=\frac{\mathsf{T}_p\left(\frac{R'}{R},\ldots,\frac{R^{(p+1)}}{R}\right)}{(R'/R)^{2p}}
\end{equation}
for some homogeneous polynomial $\mathsf{T}_p$ of degree $2p$. This entails
that $\bip{F^{(g)}}$ is of the wanted form \eqref{eq:54}: the second term
in \eqref{eq:59} is of the form
$Q^{(g)}_0\left(\frac{R'}{R},\ldots,\frac{R^{(3g-2)}}{R}\right)$ with
$Q^{(g)}_0\in \mathbb{Q}(x_1,\ldots,x_{3g-2})$ (there are no variables
$u_i$ for $n=0$), and the first term is of the form $t^{2-2g}c(g,0)$
with $c(g,0)=-\tilde{P}_g(0,\ldots,0)$.

In the case $g=1$, we rely on \cite[Equation
(31)]{budd2020irreducible}, stating, in our notation, that
$F^{(1)}=\ln(tR'/R)/12$. This entails, again by the boundary insertion
formula and by~\eqref{eq:DmRn}, that
\begin{equation}
  \label{eq:60}
 \bip{\tau^{(1)}_{2\ell}}=\frac{D_{2\ell}F^{(1)}}{R^\ell}=\frac{D_{2\ell}\left(\frac{R'}{R}\right)}{12 \frac{R'}{R} R^{\ell}}+\frac{\delta_{\ell,0}}{12t}\, =  \frac{1}{12}\left((\ell^2-1)\frac{R'}{R}+\frac{R''}{R'}\right)   +\frac{\delta_{\ell,0}}{12t}
\end{equation}
which is of the wanted form~\eqref{eq:54}, with $c(1,1)=-1/12$.  This
concludes the proof of Proposition~\ref{prop:quasipolbip}.
\end{proof}

Let us record in passing the following interesting consequence
of~\eqref{eq:60}, which calls for a bijective interpretation.
\begin{prop}
  The generating function of essentially bipartite maps of genus $1$
  with one tight boundary of length $2\ell$ is given by
  \begin{equation}
    \label{eq:65}
    \bip{T^{(1)}_{2\ell}}
    =  \frac{R^\ell}{12}\left((\ell^2-1)\frac{R'}{R}+\frac{R''}{R'}\right)
    +\frac{\delta_{\ell,0}}{12t}\, . 
  \end{equation}
\end{prop}

We now state the non bipartite analogue of
Proposition~\ref{prop:quasipolbip}.

\begin{prop}
  \label{prop:quasipolnonbip}
  Let $g,n$ be nonnnegative integers such that $2-2g-n<0$.  There
  exist a parity-dependent quasi-polynomial $\mathfrak{T}^{(g)}_n$ in
  $n$ variables, whose coefficients are
  rational functions with rational coefficients of the quantities
  $\frac{R^{(k)}}{R}$ and $\frac{S^{(k)}}{\sqrt{R}}$ for
  $k=1,\ldots,3g-2+n$, and a constant $c(g,n)\in \mathbb{Q}$ such
  that, for any integers $\ell_1,\ldots,\ell_n$, we have
    \begin{equation}
    \label{eq:quasipolnonbip}
    \tau^{(g)}_{\ell_1,\ldots,\ell_n}=
    \mathfrak{T}^{(g)}_n(\ell_1^2,\ldots,\ell_n^2)
    - c(g,n) t^{2-2g-n} \delta_{\ell_1+\cdots+\ell_n,0}\, .
  \end{equation}
\end{prop}

The proof, by induction, is in the same spirit as that of
Proposition~\ref{prop:quasipolbip}, but involves extra
difficulties. For this reason, we offload it to the appendices:
Appendix~\ref{sec:proof-prop-refpr-1} is devoted to the induction step
and Appendix~\ref{sec:proof-prop-refpr} to the initialisation.

\begin{proof}[End of the proof of Theorem~\ref{thm:quasipol}]
  In view of Proposition~\ref{prop:quasipolnonbip} and of the
  relation~\eqref{eq:tauT} between the $T$'s and the $\tau$'s, the
  only tasks left to show are that
  $c(g,n)=\chi(\mathcal{M}_{g,n})$ and that $\mathfrak{T}^{(g)}_n$ has
  total degree $3g-3+n$ for every $n\geq 1$. As discussed in
  Remark~\ref{rem:zeroweights}, the first assertion is a consequence
  of~\cite[Theorem~2]{Norbury2010}, which states that the lattice
  count polynomial $N_{g,n}$ evaluates to $\chi(\mathcal{M}_{g,n})$
  when all its variables are set to zero. The second assertion follows
  from Proposition \ref{prop:quasipol}. We believe that it should also
  be possible to keep track of the degree of $\mathfrak{T}^{(g)}_n$ in
  our inductive proof, at the price of a refinement of
  Proposition \ref{prop:quasipolnonbip}, by looking at the joint
  homogeneity properties of $\mathfrak{T}^{(g)}_n$ in the variables
  $\ell_i^2$ as well as the derivatives of $R$ and $S$, similarly to
  what is stated in Theorem \ref{sec:high-moments-high} for $n=0$. 
\end{proof}

\section{Conclusion and discussion}\label{sec:conclusion}

In this paper, we have seen that, within the theory of topological
recursion applied to the enumeration of maps, the  expansion of the
fundamental quantity $\omega^{(g)}_n(z_1,\ldots,z_n)$ has a simple
relation with the
generating functions $T^{(g)}_{\ell_1,\ldots,\ell_n}$
of maps of genus $g$ with $n$ tight boundaries of lengths $\ell_1,\ldots,\ell_n$. This coincidence is combinatorially explained
by the trumpet decomposition. We have found recursion relations, also of combinatorial nature, expressing 
$T^{(g)}_{\ell_1,\ldots,\ell_{n+1}}$ in terms of $T^{(g)}_{\ell_1,\ldots,\ell_n}$. 
Finally, we have seen that, when at least one $\ell_i$ is non-zero, 
$T^{(g)}_{\ell_1,\ldots,\ell_n}$ is equal to a parity-dependent
quasi-polynomial in the variables $\ell_1^2,\ldots,\ell_n^2$ times a simple power of $R$.

At this stage, a tantalizing question is whether one can derive
\emph{bijectively} an expression for
$T^{(g)}_{\ell_1,\ldots,\ell_n}$. So far, only the situation
$(g,n)=(0,3)$ has been treated in \cite{triskell} for the essentially
bipartite case, and in Appendix~\ref{app:tightpants} for the general
case. In particular, the remarkably simple formulas \eqref{eq:Tfour}
for $(g,n)=(0,4)$ and \eqref{eq:65} for $(g,n)=(1,1)$, in the case of
maps with faces of even degree only, are still waiting for a bijective
proof. One may also attempt to find a combinatorial interpretation
of~\eqref{eq:Tgen}, given that its constituents
$p_k(\ell_1,\ldots,\ell_n)$ and $b_{n-2,k+1}$ have themselves
interpretations in terms of tight maps~\cite{polytightmaps} and set
partitions, respectively.

\paragraph{Acknowledgements.} Many thanks are due to the anonymous
referee for suggesting valuable improvements to this paper. 
This work is supported by the
Agence Nationale de la Recherche via the grants ANR-18-CE40-0033
``Dimers'', ANR-19-CE48-0011 ``Combiné'' and ANR-23-CE48-0018
``CartesEtPlus''. JB acknowledges the hospitality of the Laboratoire
de Physique of ENS de Lyon during the initial stages of this work.
This work was completed while GM was holding a visiting professor
position at the Research Institute for Mathematical Sciences, an
International Joint Usage/Research Center located in Kyoto University.

\appendix

\section{Bijective enumeration of tight pairs of pants}
\label{app:tightpants}
The purpose of this appendix is to establish the following theorem.

\begin{thm}
  \label{thm:triskellnonbip}
  Let $a$, $b$ and $c$ be nonnegative integers or half-integers. Then, the
  generating function $T^{(0)}_{2a,2b,2c}$ of planar maps with three labeled
  distinct tight boundaries of lengths $2a$, $2b$, $2c$, counted with
  a weight $t$ per vertex different from a boundary-vertex and, for
  all $k\geq 1$, a weight $t_{k}$ per inner face of degree $k$, is
  given by
  \begin{equation}
    \label{eq:Tabc}
    T^{(0)}_{2a,2b,2c} =
    \begin{cases}
      R^{-1} \frac{\partial R}{\partial t} - t^{-1} & \text{if $a=b=c=0$,} \\
      R^{a+b+c-1} \frac{\partial R}{\partial t} & \text{if $a+b+c \in \Z_{>0}$,}\\
      R^{a+b+c-\frac12} \frac{\partial S}{\partial t} & \text{if $a+b+c \in \Z_{>0}-\frac12$,}
    \end{cases}
  \end{equation}
  where $R$ and $S$ are the formal power series in
  $t,t_1,t_2,t_3,\ldots$ defined in Section~\ref{sec:defs}.
\end{thm}

This theorem is an extension of~\cite[Theorem~1.1]{triskell} to non
bipartite maps. Note that, in this reference (which the reader is
invited to consult for any definition not found in the current paper),
we used the notation $T_{a,b,c}$ instead of the present notation
$T^{(0)}_{2a,2b,2c}$. The (essentially) bipartite case is recovered
upon taking $t_1=t_3=t_5=\cdots=0$ which, as noted above, implies that
$R$ satisfies~\eqref{eq:Rbipeq} and that $S$ vanishes. In that case,
$T^{(0)}_{2a,2b,2c}$ is non-zero if and only if $a+b+c$ is an integer,
as wanted.

To prove Theorem~\ref{thm:triskellnonbip}, we will use the well-known
fact that any map $M$ can be transformed into a bipartite map by doing
a \emph{subdivision} of every edge, that is adding a new vertex at the
midpoint of every edge, which is thus split into two. We denote by
$\phi(M)$ the resulting \emph{subdivided map}. The vertices of
$\phi(M)$ consist of \emph{regular vertices} (those originally in $M$)
and \emph{midpoint vertices} (those added by the subdivision, which
have degree $2$). The faces of $\phi(M)$ are those of $M$, only their
degree is doubled.

\begin{figure}
  \centering
  \fig{.8}{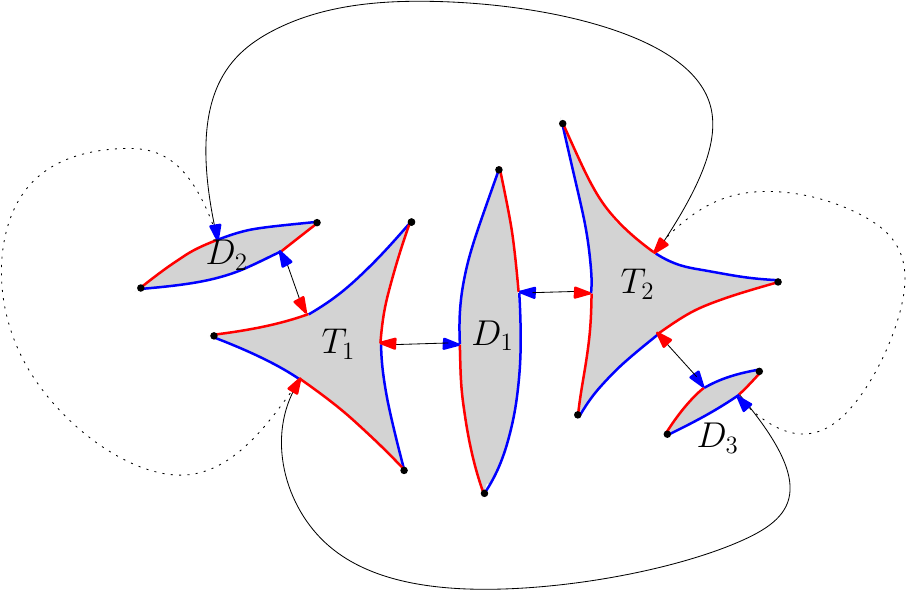}
  \caption{A summary representation of the bijection of
    \cite{triskell} between maps with three tight boundaries and
    quintuples $(D_1,D_2,D_3,T_1,T_2)$ made of three bigeodesic
    diangles $D_1,D_2,D_3$ and two bigeodesic triangles $T_1,T_2$. The
    solid lines indicate the identification between attachment points
    for type I. For type II, the two outer identification lines have
    to be replaced by the dotted lines.}
  \label{fig:assembling_unified}
\end{figure}

We then proceed by ``conjugating'' the bijective decomposition
of~\cite{triskell} by $\phi$. Indeed, let $M$ be a planar map with
three tight boundaries of lengths $2a,2b,2c$: clearly, $\phi(M)$ is a
bipartite planar map with three tight boundaries. We may then apply
one of two possible decompositions described in [\emph{ibid.},
Sections~5.6 and 5.7], namely:
\begin{itemize}
\item the decomposition of type I if $a \leq b + c$ and
  $b \leq c + a$ and $c \leq a + b$,
\item the decomposition of type II if $a \geq b + c$ or $b \geq c + a$
  or $c \geq a + b$.
\end{itemize}
In both cases, we obtain a quintuple consisting of five pieces: two
so-called bigeodesic triangles, and three bigeodesic diangles. The
vertices of these pieces are those of $\phi(M)$, and we see that a
midpoint vertex of $\phi(M)$ remains of degree $2$ in each piece to
which it belongs: indeed, the decomposition is made by cutting along
bigeodesics, and distinct bigeodesics may only merge at vertices of
degree at least $3$, which are necessarily regular. As a consequence
of this remark, each piece admits a (unique) preimage by $\phi$, and
so our decomposition amounts to cutting $M$ itself into five pieces,
denoted by $(D_1,D_2,D_3,T_1,T_2)$ in
Figure~\ref{fig:assembling_unified}.

While this approach seems very natural, the main difficulty is to
characterize precisely which sorts of pieces we may obtain in the
decomposition. Recall from [\emph{ibid.}, Sections~2.3 and 2.4] that
bigeodesic diangles and triangles are planar maps with one
boundary-face having several (four and six, respectively)
distinguished incident corners, half of them being the so-called
\emph{attachment points}. Let us here call \emph{tips} the other
distinguished corners: by the above remark, we see that the tips of
the pieces obtained in the decomposition of $\phi(M)$ are always
incident to regular vertices. However, the attachment points may be
incident to either regular or midpoint vertices. We are thus led to
introduce the following definitions.

In a map $M$, a \emph{midpoint corner} is a corner of the subdivided
map $\phi(M)$ which is incident to a midpoint vertex. A \emph{regular
  corner} is a corner in the usual sense, it corresponds to a corner
of $\phi(M)$ incident to a regular vertex. We define the distance
$d_M(c,c')$ between two (regular or midpoint) corners $c,c'$ of $M$ as
\emph{half} the graph distance between their incident vertices in
$\phi(M)$. The distance between two regular corners, or between two
midpoint corners, is thus an integer, while the distance between a
regular corner and a midpoint corner is a half-integer. The notion of
geodesic boundary interval of [\emph{ibid.}, Section~2.1] extends
naturally to midpoint corners.

\begin{figure}[h]
  \centering
  \fig{}{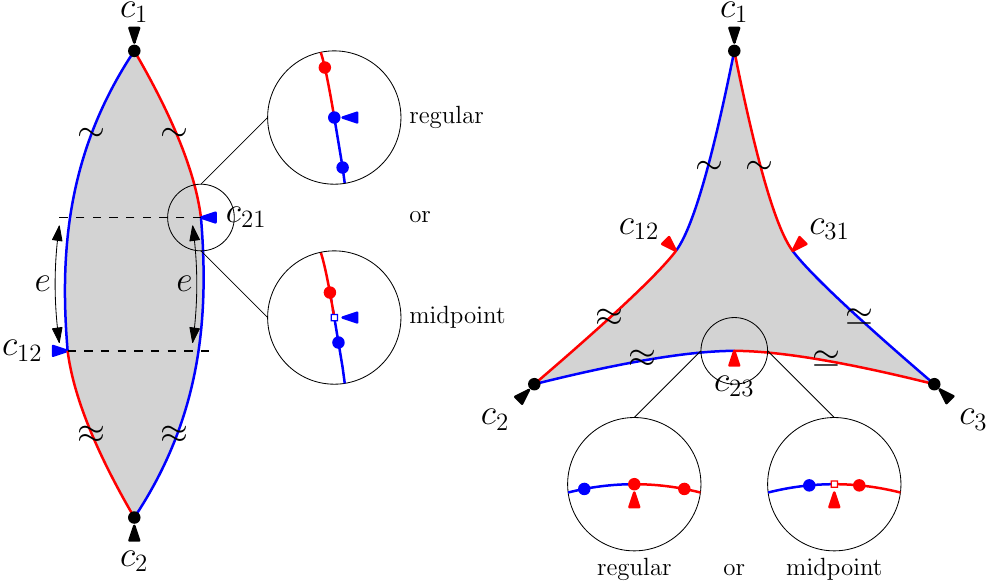}
  \caption{Schematic representation of a generalized bigeodesic
    diangle (left) and triangle (right). The tips $c_1,c_2,c_3$ are
    regular corners, incident to vertices (shown as solid disks). The
    attachment points $c_{12},...$ are either regular corners or
    midpoint corners, incident to midpoint vertices of the subdivided
    map (shown as small squares). We use the same red/blue coloring
    convention as~\cite[Figures~2.2 and 2.5]{triskell}.}
  \label{fig:diangle_triangle_nonbip}
\end{figure}

A \emph{generalized bigeodesic diangle} $D$ is defined as in
[\emph{ibid.}, Section~2.3] except that we no longer require $D$ to be
bipartite, and that we allow that zero, one or two among the
attachment points $c_{12},c_{21}$ are midpoint corners. We still
require that the tips $c_1,c_2$ are regular corners of $D$. See the
left of Figure~\ref{fig:diangle_triangle_nonbip} for an
illustration. The \emph{exceedance} of $D$ is defined as
\begin{equation}
  e(D) := d_D(c_{12},c_1)-d_D(c_{21},c_1)=d_D(c_{21},c_2)-d_D(c_{12},c_2).
\end{equation}
This quantity is half the exceedance of $\phi(D)$, which is clearly a
bigeodesic diangle in the original sense. When $e(D)$ is an integer,
the two attachment points are of the same kind (regular or midpoint),
and they are of different kinds when $e(D)$ is a half-integer.

A \emph{generalized bigeodesic triangle} $T$ is defined as in
[\emph{ibid.}, Section~2.4] except that we no longer require $T$ to be
bipartite, and that we allow the attachment points
$c_{12},c_{23},c_{31}$ to be midpoint corners. We still require that
the tips $c_1,c_2,c_3$ are regular corners of $T$. Recall that the
definition of a bigeodesic triangle implies that
$d_T(c_{12},c_1)=d_T(c_{31},c_1)$ (and circular permutations thereof),
and hence all the attachment points must be of the same kind. The
triangle is said \emph{odd} if $c_{12},c_{23},c_{31}$ are midpoint
corners, and \emph{even} if they are regular corners. See the right of
Figure~\ref{fig:diangle_triangle_nonbip} for an illustration. Clearly,
$\phi(T)$ is a bigeodesic triangle in the usual sense.

With these new definitions at hand, by applying [\emph{ibid.}, Theorem
3.1] to $\phi(M)$, we arrive at the following proposition.

\begin{prop}
  \label{prop:biject-enum-tight}
  Let $M$ be a planar map with three tight boundaries of lengths
  $2a,2b,2c$.

  If $a \leq b + c$, $b \leq c + a$, and $c \leq a + b$,
  then by the type I decomposition of~\cite[Section~5.6]{triskell},
  $M$ is in one-to-one correspondence with 
  a quintuple $(D_1,D_2,D_3,T_1,T_2)$ where $D_1,D_2,D_3$
  are generalized bigeodesic triangles of respective exceedances
  $e(D_1)=a+c-b$, $e(D_2)=a+b-c$, $e(D_3)=b+c-a$, and $T_1,T_2$ are
  generalized bigeodesic triangles, with the constraints that the
  attachment points which are identified together on
  Figure~\ref{fig:assembling_unified} must be of the same (regular or
  midpoint) kind, and that not all elements of this quintuple are
  reduced to the vertex-map. 

  If $b \geq a+c$, say, then the type II decomposition
  of~[ibid., Section~5.7] produces a quintuple
  $(D_1,D_2,D_3,T_1,T_2)$ with the same properties, except that the
  exceedances are now $e(D_1)=b-a-c$, $e(D_2)=2a$, $e(D_3)=2c$.

  For both type I and type II, we have the relation
  \begin{equation}
    \label{eq:excrel}
    e(D_1)+e(D_2)+e(D_3)=a+b+c\, .
  \end{equation}
\end{prop}

The above proposition encompasses a number of different cases. When
$a+b+c$ is an integer, all the exceedances are integers, and therefore
all twelve attachment points are of the same kind. If they are all
regular (resp.\ midpoint) corners, then $T_1$ and $T_2$ are both even
(resp.\ odd) triangles. When $a+b+c$ is a half-integer, $e(D_1)$ is a
half-integer and hence $T_1$ and $T_2$ have different
``parities''. The other exceedances $e(D_2)$ and $e(D_3)$ are
half-integers for type I, and integers for type II.

Let us now discuss the enumerative consequences of
Proposition~\ref{prop:biject-enum-tight}.  To this end, it is
necessary to introduce the generating functions of the various objects
that may arise in the situations described above. Generally speaking,
we assign a weight $t_k$ per inner face of degree $k$ (before
subdivision), and a weight $t$ per vertex (of the map, i.e.\ per
regular vertex of the subdivided map). Following the conventions of
[\emph{ibid.}], vertices that lie on strictly geodesic boundaries
(displayed in red in the figures), and that are not incident to
attachment points, do not receive the weight $t$.  We denote by $Y$
(resp.\ $\widetilde{Y}$) the generating function of generalized
bigeodesic triangles that are even (resp.\ odd). We then treat the
case of diangles in the following proposition.

\begin{prop}\label{prop:biject-enum-tight-1}
 Let $X$ denote the generating function of generalized bigeodesic
diangles with exceedance $0$, whose attachment points are both
regular corners. Then the following holds. 

The generating function of generalized bigeodesic diangles with
exceedance $e\in \Z_{\geq 0}$, whose attachment points are both
regular corners, is equal to $R^eX$.

The generating function of generalized bigeodesic diangles with
exceedance $e\in \Z_{\geq 0}+\frac 1 2$, where the first attachment
point is a regular corner and the second a midpoint corner, is equal
to $R^{e+\frac{1}{2}}X/t$.

  The generating function of generalized bigeodesic diangles with
  exceedance $e\in \Z_{\geq 0}$, whose
  attachment points are both midpoint corners, is equal to $R^{e+1}X/t^2$.
\end{prop}

\begin{proof}
  The first statement is obtained by a straightforward adaptation of
  the proof of [\emph{ibid.}, Proposition 2.2], noting that the
  arguments do not require that the maps be bipartite, and that $R$
  can still be interpreted as the generating function for 
  elementary slices.

  The second statement follows from the first one, by observing that
  there is a bijection between the generalized bigeodesic diangles
  whose exceedance is $e\in \Z_{\geq 0}+\frac12$ and whose second
  attachment point is a midpoint corner, and diangles whose exceedance
  is $e+\frac12$ and whose both attachment points are regular
  corners. Indeed, looking at
  Figure~\ref{fig:diangle_triangle_nonbip}, we simply have to move the
  second attachment point $c_{21}$ to the nearest (red) regular vertex
  in the direction of $c_1$, and we need to divide by $t$ to unweigh
  this regular vertex.

Finally, the third statement follows from a similar observation: we
now move the two attachment points $c_{12}$ and $c_{21}$ to their
nearest regular vertices in the direction of $c_2$ and $c_1$
respectively, resulting in a diangle of exceedance $e+1$, and the
weight $1/t^2$ is needed to unweigh these regular vertices. 
\end{proof}

\begin{proof}[Proof of Theorem \ref{thm:triskellnonbip}. ]
  We apply Proposition~\ref{prop:biject-enum-tight}, discussing separately the
  different cases that appear there.
   Let us first consider the case $a+b+c\in \Z_{> 0}$: we may write
  \begin{align}
    \nonumber
    T^{(0)}_{2a,2b,2c}&=
    \frac{1}{t^6}(R^{e(D_1)}X)(R^{e(D_2)}X)(R^{e(D_3)}X)Y^2+\frac{R^{e(D_1)+1}X}{t^2}\frac{R^{e(D_2)+1}X}{t^2}\frac{R^{e(D_3)+1}X}{t^2}\widetilde{Y}^2
    \\
    &=R^{a+b+c}\frac{X^3(Y^2+R^3\widetilde{Y}^2)}{t^6}\, . \label{eq:TXYeven}
  \end{align}
  Here, the term in $Y^2$ (resp.\ $\widetilde{Y}^2$) corresponds to
  the the situation where both $T_1$ and $T_2$ are even (resp.\ odd)
  triangles, and the division by $t^6$ in the first term arises from
  the identification between the vertices incident to the attachment
  points. Note that the second equality holds both for type I and type
  II, thanks to~\eqref{eq:excrel}.

  The formula~\eqref{eq:TXYeven} remains valid in the case $a=b=c=0$,
  except that we have to remove a spurious term $t^{-1}$ corresponding
  to the contribution of the quintuple whose all elements are reduced
  to the vertex map.

  Let us now consider the case $a+b+c\in \Z_{\geq 0}+\frac 12$: then
  the discussion differs slightly between type I and type II. For type
  I, we may write
  \begin{equation}
    \label{eq:TXYoddI}
    T^{(0)}_{2a,2b,2c}=
    \frac{1}{t^3}\frac{R^{e(D_1)+\frac12}X}{t}\frac{R^{e(D_2)+\frac12}X}{t} \frac{R^{e(D_3)+\frac12}X}{t}(2Y\widetilde{Y})
    =R^{a+b+c-\frac12}\frac{2R^2X^3Y\widetilde{Y}}{t^6}\, .
  \end{equation}
  Here, the factor $2Y\widetilde{Y}$ enumerates the pairs $(T_1,T_2)$
  of different parities. Note that in this situation, all diangles
  have exactly one regular attachment point, and the factor $1/t^3$
  arises from the identification between the vertices at these
  attachment points. For type II, assuming without loss of generality
  that $b\geq a+c$ as in Proposition \ref{prop:biject-enum-tight}, we
  may write
  \begin{align}
    \nonumber
    T^{(0)}_{2a,2b,2c}&=
    \frac{1}{t^3}\frac{R^{e(D_1)+\frac12}X}{t}\left( (R^{e(D_2)}X) \frac{R^{e(D_3)+1}X}{t^2} Y\widetilde{Y}+
   \frac{R^{e(D_2)+1}X}{t^2}(R^{e(D_3)}X)\widetilde{Y}Y\right)\\
 &=R^{a+b+c-\frac12}\frac{2R^2X^3Y\widetilde{Y}}{t^6}\, .     \label{eq:TXYoddII}
  \end{align}
  Here, the first (resp.\ second) term corresponds to the case where
  $T_1$ is even (resp.\ odd), so that $D_2$ has two (resp.\ zero)
  regular attachment points, $D_3$ has zero (resp.\ two) regular
  attachment points. In all cases, $D_1$ has exactly one regular
  attachment point.

  Gathering all these cases together, we arrive at the unified formula
   \begin{equation}
    \label{eq:Tabcunited}
    T^{(0)}_{2a,2b,2c} +t^{-1}\delta_{a+b+c=0}=
    \begin{cases}
       R^{a+b+c-1} \frac{RX^3(Y^2+R^3\widetilde{Y}^2)}{t^6} & \text{if $a+b+c \in \Z_{\geq 0}$,}\\
      R^{a+b+c-\frac12} \frac{2R^2X^3Y\widetilde{Y}}{t^6} & \text{if $a+b+c \in \Z_{\geq 0}+\frac12$.}
    \end{cases}
  \end{equation}
  We conclude by noting that 
  \begin{equation}
    \label{eq:13000}
    \frac{\partial
      R}{\partial t}=T^{(0)}_{1,1,0}=\frac{RX^3(Y^2+R^3\widetilde{Y}^2)}{t^6}\quad
    \mbox{and}\quad \frac{\partial S}{\partial t}=T^{(0)}_{1,0,0}=\frac{2R^2X^3Y\widetilde{Y}}{t^6} \, ,
  \end{equation}
  since $R$ (resp.\ $S$) counts planar maps with two distinguished
  boundaries both of length $1$ (resp.\ one of length $1$ and the
  other of length $0$): differentiating with respect to $t$ has the
  effect of adding a third boundary of length $0$, and boundaries of
  lengths $0$ or $1$ are tautologically tight.
\end{proof}

We now state a variant of Theorem~\ref{thm:triskellnonbip} allowing to
count tight pairs of pants with some boundaries strictly tight, under
certain assumptions.

\begin{prop}
  \label{prop:strictpants}
  Let $a$, $b$ and $c$ be nonnegative integers or half-integers, and
  let $T^{(0)}_{2a,2b|2c}$ (resp.\ $T^{(0)}_{2a|2b,2c}$) denote the
  generating function of planar maps with three labeled distinct tight
  boundaries of lengths $2a,2b,2c$, the third (resp.\ the second and
  the third) being \emph{strictly tight}, where we attach a weight $t$
  per vertex different from a boundary-vertex and not incident to the
  third boundary (resp.\ to the second nor to the third boundary) and,
  for all $k\geq 1$, a weight $t_{k}$ per inner face of degree $k$.

  Then, for $a+b>c$, we have
  \begin{equation}
    \label{eq:strictpants}
    T^{(0)}_{2a,2b|2c} =
    \begin{cases}
      R^{a+b-c-1} \frac{\partial R}{\partial t} & \text{if $a+b+c \in \Z_{>0}$,}\\
      R^{a+b-c-\frac12} \frac{\partial S}{\partial t} & \text{if $a+b+c \in \Z_{>0}-\frac12$,}
    \end{cases}
  \end{equation}
  while, for $a>b+c$, we have
  \begin{equation}
    \label{eq:doublestrictpants}
    T^{(0)}_{2a|2b,2c} =
    \begin{cases}
      R^{a-b-c-1} \frac{\partial R}{\partial t} & \text{if $a+b+c \in \Z_{>0}$,}\\
      R^{a-b-c-\frac12} \frac{\partial S}{\partial t} & \text{if $a+b+c \in \Z_{>0}-\frac12$.}
    \end{cases}
  \end{equation}  
\end{prop}

Observe that, in these two expressions, the assumptions $a+b>c$ and
$a>b+c$ are respectively crucial to avoid having terms involving
negative powers of $t$. It would be interesting to obtain formulas for
$T^{(0)}_{2a,2b|2c}$ and $T^{(0)}_{2a|2b,2c}$ without these
assumptions, but this is beyond the scope of this paper. Note that,
whenever $a>b+c$, the boundaries of lengths $b,c$ cannot touch each
other.

\begin{proof}[Proof of Proposition~\ref{prop:strictpants}]
  In view of Theorem~\ref{thm:triskellnonbip}, it suffices to
  establish that
  \begin{equation}
    \label{eq:strictpantsalt}
    T^{(0)}_{2a,2b,2c} =
    \begin{cases}
      R^{2c} T^{(0)}_{2a,2b|2c} & \text{if $a+b>c$,} \\
      R^{2b+2c} T^{(0)}_{2a|2b,2c} & \text{if $a>b+c$.} \\
    \end{cases}
  \end{equation}
  Let us first assume $a+b>c$ and establish the formula for
  $T^{(0)}_{2a,2b|2c}$. If $c=0$ then there is nothing to prove, as a
  boundary-vertex is always considered strictly tight. Assuming $c>0$,
  let $M$ be a planar map with three labeled distinct tight boundaries
  of lengths $2a,2b,2c$. The third boundary is a face, which we denote
  by $F_3$. We then cut $M$ along the \emph{innermost} minimal closed path
  $C$ homotopic to $F_3$. This path can be precisely defined as the
  \emph{maximal} element of the lattice $\CCm^{(3)}$, discussed
  in~\cite[Section~6.1]{triskell}. We will justify in
  Lemma~\ref{lem:simplepantssum} below that $C$ is simple, hence by
  cutting we obtain two pieces: a planar map with three tight
  boundaries of lengths $2a,2b,2c$, the third being strictly tight,
  and a planar map with two tight boundaries both of length
  $2c$. Moreover, the boundary resulting from this cutting operation
  can be canonically rooted: indeed, since planar maps with three boundaries have a trivial automorphism
group, we may canonically select a distinguished corner incident to $C$. This
  decomposition is clearly bijective. Turning to generating functions,
  we decide to assign the weight $t$ for vertices on $C$ to
  the second map. It is known~\cite[Section~9.3]{irredmaps} that
  planar maps with two tight boundaries both of length $2c$, one of
  which is rooted, have
  generating function $R^{2c}$. We conclude that
  $T^{(0)}_{2a,2b,2c} = R^{2c} T^{(0)}_{2a,2b|2c}$ as wanted.

  For $a>b+c$, we may perform a similar decomposition with the
  boundary of length $2b$, giving the relation
  $T^{(0)}_{2a,2b|2c}=R^{2b} T^{(0)}_{2a|2b,2c}$.
\end{proof}

Let us now justify that the closed path $C$ used in the above proof is
simple:

\begin{lem}
  \label{lem:simplepantssum}
  Let $M$ be a planar map with three labeled distinct tight boundaries
  of lengths $\ell_1,\ell_2,\ell_3$. If $\ell_1+\ell_2>\ell_3>0$ then
  every closed path on $M$ of length $\ell_3$ which is homotopic to
  the third boundary is simple.
\end{lem}

\begin{proof}
  Let $C$ be a closed path on $M$ of length $\ell>0$. We may decompose
  $C$ into a (finite) sequence of \emph{simple} closed paths
  $\hat{C}=(C_1,C_2,\ldots,C_n)$ with \emph{positive lengths}, for
  instance by doing a ``loop-erasure'' (in this construction, a closed
  path reduced to a single edge followed once in both ways is regarded
  as simple).

  To be explicit, choose an arbitrary vertex $v_0$ on $C$ and an
  orientation, and let us denote by
  $v_0,e_1,v_1,e_2,v_2,\ldots,e_\ell,v_\ell=v_0$ the successive
  vertices and edges of $M$ visited by $C$. We then apply the
  following iterative procedure. Initialize $\gamma$ as the path
  reduced to $v_0$, and $\hat{C}$ as an empty sequence. Then, for every step
  $i$ from $1$ to $\ell$, append the edge $e_i$ and $v_i$ to $\gamma$:
  \begin{itemize}
  \item if $\gamma$ is simple, proceed to the next step,
  \item if $\gamma$ is not simple, then it necessarily consists of a
    simple path $\gamma'$ from $v_0$ to $v_i$ (possibly of length zero
    if $v_i=v_0$) and of a simple closed path $C'$ with positive
    length: set $\gamma=\gamma'$, append $C'$ to the sequence
    $\hat{C}$, and proceed to the next step.
  \end{itemize}
  At the last step, the path $\gamma$ is reduced to $v_0=v_\ell$, and
  $\hat{C}$ is the sequence we are looking for. There is an
  interesting combinatorial structure of \emph{heap} underlying this
  decomposition, see e.g.~\cite[Proposition~6.3]{Viennot1986}, but for
  our purposes we need only remark that the sum of the lengths of the
  $C_i$'s is $\ell$.

  Now, let us assume that $C$ is homotopic to the third boundary. We
  claim that at least one of the following assertions is true:
  \begin{itemize}
  \item[(i)] there exists $i$ such that $C_i$ is homotopic to the
    third boundary,
  \item[(ii)] there exists $i,j$ such that $C_i$ and $C_j$ are
    homotopic to the first and second boundaries, respectively.
  \end{itemize}
  Indeed, each $C_i$ is either contractible or homotopic to one of the
  boundaries, as we justify in Lemma~\ref{lem:simplepants} below. It
  is not possible that all $C_i$ are contractible or homotopic to the
  first boundary, as otherwise $C$ would be homotopic to a multiple of
  the latter. Similarly for the second boundary. Hence, either (i) or
  (ii) must hold.

  Let us now assume $\ell=\ell_3$. As the boundaries are assumed
  tight, in case (i) we find that $\hat{C}$ consists of a single
  element ($n=1$) so that $C=C_1$ is simple. In case (ii) $C$ would
  have length at least $\ell_1+\ell_2$, but this is excluded by the
  assumption $\ell_1+\ell_2>\ell_3$.
\end{proof}

\begin{rem}
  In the case where $\ell_1$, say, vanishes,
  Lemma~\ref{lem:simplepantssum} and the above proof still hold, upon
  using the convention discussed in Section~\ref{sec:defs} and
  \cite[Figure~2]{triskell} of replacing the boundary-vertex by a
  small circle made of edges of length zero. Observe that the
  boundary-vertex cannot be incident to the boundary-face of length
  $\ell_3$, as otherwise following the contour of that face but
  circumventing the boundary-vertex in the other direction would yield
  a path of length $\ell_3$ that is homotopic to the
  boundary-face of length $\ell_2$, contradicting the tightness
  assumption.
\end{rem}

The following fact, used in the proof of
Lemma~\ref{lem:simplepantssum}, seems well known in some communities,
but we rederive it for convenience.

\begin{lem}
  \label{lem:simplepants}
  Let $M$ be a planar map with three boundaries. Then, every simple
  closed path on $M$ is either contractible or homotopic to one of the
  boundaries.
\end{lem}

\begin{proof}
  We may draw $M$ in the plane, with the puncture of one of the
  boundaries sent to infinity.  Let $C$ be a simple closed path on
  $M$. By the Jordan-Schoenflies theorem, $C$ delimits a bounded
  domain $D$ homeomorphic to a disk, which may contain zero, one or
  two punctures. If $D$ contains zero puncture, then $C$ is
  contractible.  If $D$ contains one puncture, then $C$ is homotopic
  to the corresponding boundary. Finally, if $D$ contains two
  punctures, then $C$ is homotopic to the boundary corresponding to
  the puncture at infinity.
\end{proof}

\section{Quasi-polynomiality in the general case: induction step}
\label{sec:proof-prop-refpr-1}

This appendix and the next one are devoted to the proof of
Proposition~\ref{prop:quasipolnonbip} by induction on $n$. Here, we
check that if the statement of the proposition is true for $n$
boundaries, then it is also true for $n+1$ boundaries. We proceed in
three steps.

\subsection{A discrete integration lemma for polynomials}

We first record an elementary discrete integration
lemma. It is certainly well known, but the derivation is a nice
variation on the Bernoulli-Faulhaber formulas so we give a full proof.

\begin{lem}
  \label{sec:discr-integr-lemma}
  Fix an integer $k\geq 0$. Then the sums
  \begin{equation}
    \label{eq:1}
    \sum_{0<m<\ell\atop m\in 2\Z}m^{2k+1}+\frac{\ell^{2k+1}}{2}\,
    ,\qquad \sum_{0<m<\ell \atop m\in 2\Z+1}m^{2k+1}\, ,
  \end{equation}
  are even polynomials in the \textbf{even} integer variable $\ell$, and the
  sums
  \begin{equation}
    \label{eq:200}
     \sum_{0<m<\ell\atop m\in 2\Z}m^{2k+1}\,
    ,\qquad \sum_{0<m<\ell \atop m\in
      2\Z+1}m^{2k+1}+\frac{\ell^{2k+1}}{2}\, ,
  \end{equation}
are even polynomials in the \textbf{odd} integer variable $\ell$. 
\end{lem}

\begin{proof}
  First suppose that $\ell=2n$ is even. 
  By the Bernoulli-Faulhaber formula, we have
  \begin{align}
    \sum_{0<m<\ell\atop m\in
      2\Z}m^{2k+1}+\frac{\ell^{2k+1}}{2}&=2^{2k+1}\left(\sum_{q=1}^nq^{2k+1}-n^{2k+1}+\frac{n^{2k+1}}{2}\right)\\
    &=2^{2k+1}\left(\frac{1}{2k+2}\sum_{i=0}^{2k+1}\binom{2k+2}{i}B_in^{2k+2-i}
    -\frac{n^{2k+1}}{2}\right)\, .
  \end{align}
Since $B_1=1/2$ and all other Bernoulli numbers with odd indices
vanish, we obtain that this is a polynomial in
$n^2=(\ell/2)^2$. Similarly, a simple consequence of the
Bernoulli-Faulhaber formula, subtracting the terms with even indices, is that 
\begin{equation}
  \label{eq:4}
  \sum_{0<m<\ell\atop m\in 2\Z+1}m^{2k+1}=\frac{2^{2k+2}}{k+1}\sum_{i=0}^{2k+1}\binom{2k+2}{i}B_in^{2k+2-i}(2^{1-i}-1)
\end{equation}
and we see that the term involving $B_1$ vanishes. Using again the
vanishing properties of Bernoulli numbers of odd indices, this is a
polynomial in $n^2=(\ell/2)^2$.

We now study the case where $\ell=2n-1$ is odd. Note that 
\begin{equation}
  \label{eq:5}
  \sum_{0<m<\ell\atop m\in 2\Z}m^{2k+1}=2^{2k+1}\sum_{0<q<n}q^{2k+1}
\end{equation}
which, by another formula of 
Faulhaber, is a polynomial in the
variable $n(n-1)=(\ell^2-1)/4$, hence an even polynomial in $\ell$.
Finally, simply note that, since $\ell$ is odd, 
\begin{equation}
  \label{eq:6}
  \sum_{0<m<\ell\atop m\in 2\Z+1}m^{2k+1}+\frac{\ell^{2k+1}}{2}=
  \sum_{0<m\leq \ell\atop m\in 2\Z+1}m^{2k+1}-\frac{\ell^{2k+1}}{2}
  =\frac{1}{2}\left(\sum_{0<m<\ell\atop m\in 2\Z+1}m^{2k+1}+
  \sum_{0<m\leq \ell\atop m\in 2\Z+1}m^{2k+1}\right)\, ,
\end{equation}
since the last term is the average of the first two equal terms. By a
final use of the Bernoulli-Faulhaber formula, this quantity if of the form
$(P(n)+P(n+1))$ where $P$ is a polynomial and $2n-1=\ell-2$, so it can be
put in the form $P((\ell-1)/2)+P((\ell+1)/2)$, which is an even polynomial
in $\ell$ by Newton's formula. 
\end{proof}

\begin{cor}
  \label{sec:discr-integr-lemma-1}
  Let $P\in \mathbb{Q}[x]$ be a polynomial. Then
  \begin{equation}
    \label{eq:8}
    \sum_{0<m<\ell\atop m\in 2\Z}mP(m^2)+\frac{\ell}{2}P(\ell^2)\,
    ,\qquad \sum_{0<m<\ell\atop m\in 2\Z+1}mP(m^2)\, ,
  \end{equation}
  are even polynomials of the even integer variable $\ell$, and
  \begin{equation}
    \label{eq:9bis}
     \sum_{0<m<\ell\atop m\in 2\Z}mP(m^2)\, ,\qquad
     \sum_{0<m<\ell\atop m\in 2\Z+1}mP(m^2) +\frac{\ell}{2}P(\ell^2)\, ,
   \end{equation}
    are even polynomials of the odd integer variable $\ell$. 
\end{cor}

\subsection{Quasi-polynomiality for tight boundary insertions}

 Next, we  prove the following result on the boundary
insertion operator $D_\ell$.  

\begin{prop}
  \label{sec:quasi-polyn-tight}
  For every $i\geq 0$, the quantities
  \begin{equation}
    \label{eq:11}
    R^{-\frac{\ell}{2}}\frac{D_\ell R^{(i)}}{R}\, \quad ,\qquad
    R^{-\frac{\ell}{2}}\frac{D_\ell S^{(i)}}{\sqrt{R}}
  \end{equation}
  are parity-dependent quasi-polynomials in the variable $\ell^2$
  whose coefficients are polynomial functions of
  $R^{(j)}/R,S^{(j)}/\sqrt{R},1\leq j\leq i+1$.
\end{prop}

\begin{proof}
We prove this statement by induction on $i$, using the fact that for
every $\ell\geq 0$, 
\begin{equation}
  \label{eq:1200}
  D_\ell R^{(i)}=T^{(0)}_{\ell,1,1,0_i}\, ,\qquad D_\ell
  S^{(i)}=T^{(0)}_{\ell,1,0_{i+1}}\, ,
\end{equation}
where we use the simplifying notation $0_i=0,\ldots,0$ for a string of
$i$ consecutive zeros.  We initialize the induction at $i=0$ by noting
that, by our formula \eqref{eq:Tabc} for tight pairs of pants,
\begin{equation}
  \label{eq:13}
  D_\ell R=R^{\frac{\ell}{2}+1}\times
\begin{cases}\displaystyle{\frac{R'}{R}} & \text{if $\ell$ is
                                         even}\\
           \displaystyle{\frac{S'}{\sqrt R}} & \text{ if $\ell$ is 
             odd,}
         \end{cases}
         \qquad D_\ell S=R^{\frac{\ell+1}{2}}\times 
         \begin{cases}  \displaystyle{\frac{S'}{\sqrt R}} & \text{if $\ell$ is
                                         even}\\
      \displaystyle{\frac{R'}{R}}    & \text{ if $\ell$ is 
        odd.}
    \end{cases}
\end{equation}
Suppose that the statement holds
up to  some given value of $i$: namely, we may
find polynomials $P_{j,2},P_{j,3},Q_{j,1},Q_{j,2}$ in $\mathbb{Q}[x_1,\ldots,x_{j+1},y_1,\ldots,y_{j+1},u]$
(the second index will refer to the number of boundary-faces of odd
degree in the expressions below) such 
that, for $0\leq j\leq i$, and letting $\epsilon$ be $0$ if $\ell$ is
even and $1$ if $\ell$ is odd, we have
\begin{equation}
  \label{eq:14}
  \begin{split}
    D_\ell R^{(j)}&= R^{\frac{\ell}{2}+1}
    P_{j,2+\epsilon}\left(\frac{R^{(\leq j+1)}}{R},\frac{S^{(\leq j+1)}}{\sqrt{R}},\ell^2\right)\, , \\
    D_\ell S^{(j)}&= R^{\frac{\ell+1}{2}}
    Q_{j,1+\epsilon}\left(\frac{R^{(\leq j+1)}}{R},\frac{S^{(\leq
    j+1)}}{\sqrt{R}},\ell^2\right)\, .
  \end{split}
\end{equation}
Here, we use the shorthand notation $R^{(\leq j+1)},S^{(\leq j+1)}$
for the variables $(R^{(k)},1\leq k\leq j+1)$ and
$(S^{(k)},1\leq k\leq j+1)$.

Now, to prove the statement at rank $i+1$, we write
\begin{align}
  \label{eq:15bis}
  D_\ell R^{(i+1)}&=T^{(0)}_{\ell,1,1,0_{i+1}}=\frac{\partial}{\partial
  t}T^{(0)}_{\ell,1,1,0_i}+\sum_{0<m<\ell}mT^{(0)}_{\ell,0|m}T^{(0)}_{m,1,1,0_i}\, , \\
  \label{eq:19}
  D_\ell S^{(i+1)}&=T^{(0)}_{\ell,1,0_{i+2}}=\frac{\partial}{\partial
  t}T^{(0)}_{\ell,1,0_{i+1}}+\sum_{0<m<\ell}mT^{(0)}_{\ell,0|m}T^{(0)}_{m,1,0_{i+1}}\, ,
\end{align}
where the second equality on each line is a consequence of Proposition
\ref{prop:recur0}. Let us first analyse the right-hand side
of~\eqref{eq:15bis}. Its first term is equal to
\begin{align}
  \label{eq:10}
\frac{\partial}{\partial t}T^{(0)}_{\ell,1,1,0_i}&=
                            \left(\frac{\ell}{2}+1\right)R^{\frac{\ell}{2}+1}\frac{R'}{R}P_{i,2+\epsilon}\left(\frac{R^{(\leq
                            i+1)}}{R},\frac{S^{(\leq
                            i+1)}}{\sqrt{R}},\ell^2\right)\\
                          &+R^{\frac{\ell}{2}+1}\sum_{k=1}^{i+1}\left(\frac{R^{(k)}}{R}\right)'\frac{\partial
                            P_{i,2+\epsilon}}{\partial x_k}\left(\frac{R^{(\leq
                            i+1)}}{R},\frac{S^{(\leq
                            i+1)}}{\sqrt{R}},\ell^2\right)\\
  &+R^{\frac{\ell}{2}+1}\sum_{k=1}^{i+1}\left(\frac{S^{(k)}}{\sqrt{R}}\right)'\frac{\partial
                            P_{i,2+\epsilon}}{\partial y_k}\left(\frac{R^{(\leq
                            i+1)}}{R},\frac{S^{(\leq
                            i+1)}}{\sqrt{R}},\ell^2\right)\, .    
\end{align}
Note that $(R^{(k)}/R)'=(R^{(k+1)}/R)-(R^{(k)}/R)(R'/R)$ and
$(S^{(k)}/\sqrt{R})'=(S^{(k+1)}/\sqrt{R})-(1/2)(S^{(k)}/\sqrt{R})(R'/R)$,
and therefore, the last two sums are quasi-polynomial in $\ell^2$ with
coefficients in $\mathbb{Q}[R^{(\leq i+2)}/R,S^{(\leq
  i+2)}/\sqrt{R}]$. 
In order to study the last sum in \eqref{eq:15bis}, it is easier to
split cases.
\paragraph{Assume that $\ell$ is even ($\epsilon=0$).} The last sum in \eqref{eq:15bis}
then equals 
\begin{align}
  \label{eq:3}
 \sum_{0<m<\ell}mT^{(0)}_{\ell,0|m}T^{(0)}_{m,1,1,0_i}&=\sum_{0<m<\ell\atop m\in 2\Z}mR^{\frac{\ell-m}{2}}\frac{R'}{R}R^{\frac{m}{2}+1}P_{i,2}\left(\frac{R^{(\leq
                            i+1)}}{R},\frac{S^{(\leq
  i+1)}}{\sqrt{R}},m^2\right)    \\
  &+\sum_{0<m<\ell\atop m\in 2\Z+1}mR^{\frac{\ell-m}{2}}\frac{S'}{\sqrt{R}}R^{\frac{m}{2}+1}P_{i,3}\left(\frac{R^{(\leq
                            i+1)}}{R},\frac{S^{(\leq
  i+1)}}{\sqrt{R}},m^2\right)\, .\label{eq:16}
\end{align}
By Corollary \ref{sec:discr-integr-lemma-1}, the term \eqref{eq:16},
and the sum of the term \eqref{eq:3} and \eqref{eq:10} (for $\epsilon=0$) is of the
wanted form of $R^{\frac{\ell}{2}+1}$ times an even polynomial in the
variable $\ell^2$, with coefficients in $\mathbb{Q}[R^{(\leq i+1)}/R,S^{(\leq
  i+1)}/\sqrt{R}]$. Putting things together, we see that \eqref{eq:15bis}
indeed yields an even polynomial expression in the even variable
$\ell$.
\paragraph{Assume that $\ell$ is odd ($\epsilon=1$).} This time, we
have
\begin{align}
  \label{eq:17}
 \sum_{0<m<\ell}mT^{(0)}_{\ell,0|m}T^{(0)}_{m,1,1,0_i}&=\sum_{0<m<\ell\atop m\in 2\Z}mR^{\frac{\ell-m}{2}}\frac{S'}{\sqrt{R}}R^{\frac{m}{2}+1}P_{i,2}\left(\frac{R^{(\leq
                            i+1)}}{R},\frac{S^{(\leq
  i+1)}}{\sqrt{R}},m^2\right)    \\
  &+\sum_{0<m<\ell\atop m\in 2\Z+1}mR^{\frac{\ell-m}{2}}\frac{R'}{R}R^{\frac{m}{2}+1}P_{i,3}\left(\frac{R^{(\leq
                            i+1)}}{R},\frac{S^{(\leq
  i+1)}}{\sqrt{R}},m^2\right)\, .\label{eq:18}
\end{align}
We then apply Corollary \ref{sec:discr-integr-lemma-1} to the term
\eqref{eq:17} on the one hand, and to the sum of the term
\eqref{eq:18} and \eqref{eq:10} (for $\epsilon=1$) on the other
hand. These are of the wanted form of $R^{\frac{\ell}{2}+1}$ times an
even polynomial in $\ell^2$, with coefficients in
$\mathbb{Q}[R^{(\leq i+1)}/R,S^{(\leq i+1)}/\sqrt{R}]$. Putting things
together, we see that \eqref{eq:15bis} indeed yields an even
polynomial expression in the odd variable $\ell$. This completes the
induction step for $D_\ell R^{(i+1)}$.

We now analyse the right-hand side of~\eqref{eq:19} in a similar
way. Its first term is equal to
\begin{align}
  \label{eq:21}
\frac{\partial}{\partial t}T^{(0)}_{\ell,1,0_{i+1}}&=
                            \frac{\ell+1}{2}\cdot R^{\frac{\ell+1}{2}}\frac{R'}{R}Q_{i,1+\epsilon}\left(\frac{R^{(\leq
                            i+1)}}{R},\frac{S^{(\leq
                            i+1)}}{\sqrt{R}},\ell^2\right)\\
                          &+R^{\frac{\ell+1}{2}}\sum_{k=1}^{i+1}\left(\frac{R^{(k)}}{R}\right)'\frac{\partial
                            Q_{i,1+\epsilon}}{\partial x_k}\left(\frac{R^{(\leq
                            i+1)}}{R},\frac{S^{(\leq
                            i+1)}}{\sqrt{R}},\ell^2\right)\\
  &+R^{\frac{\ell+1}{2}}\sum_{k=1}^{i+1}\left(\frac{S^{(k)}}{\sqrt{R}}\right)'\frac{\partial
                            Q_{i,1+\epsilon}}{\partial y_k}\left(\frac{R^{(\leq
                            i+1)}}{R},\frac{S^{(\leq
                            i+1)}}{\sqrt{R}},\ell^2\right)    
\end{align}
Again, the last two sums are quasi-polynomial in the variable $\ell^2$
with coefficients in
$\mathbb{Q}[R^{(\leq i+2)}/R,S^{(\leq i+2)}/\sqrt{R}]$.  To analyse
the last sum in \eqref{eq:19} we again split cases. Assume that $\ell$
is even ($\epsilon=0$). Then, we have
\begin{align}
\label{eq:22}
 \sum_{0<m<\ell}mT^{(0)}_{\ell,0|m}T^{(0)}_{m,1,0_{i+1}}&=\sum_{0<m<\ell\atop m\in 2\Z}mR^{\frac{\ell-m}{2}}\frac{R'}{R}R^{\frac{m+1}{2}}Q_{i,1}\left(\frac{R^{(\leq
                            i+1)}}{R},\frac{S^{(\leq
  i+1)}}{\sqrt{R}},m^2\right)    \\
  &+\sum_{0<m<\ell\atop m\in 2\Z+1}mR^{\frac{\ell-m}{2}}\frac{S'}{\sqrt{R}}R^{\frac{m+1}{2}}Q_{i,2}\left(\frac{R^{(\leq
                            i+1)}}{R},\frac{S^{(\leq
  i+1)}}{\sqrt{R}},m^2\right)\, .\label{eq:23}
\end{align}
We apply Corollary \ref{sec:discr-integr-lemma-1} to the term
\eqref{eq:23}, and to the sum of the term \eqref{eq:22} and
\eqref{eq:21} (for $\epsilon=0$), showing that they are of the wanted
form of $R^{\frac{\ell+1}{2}}$ times an even polynomial in the
variable $\ell^2$, with coefficients in
$\mathbb{Q}[R^{(\leq i+1)}/R,S^{(\leq i+1)}/\sqrt{R}]$. Finally,
assume that $\ell$ is odd ($\epsilon=1$). This time, we have
\begin{align}
\label{eq:24} 
 \sum_{0<m<\ell}mT^{(0)}_{\ell,0|m}T^{(0)}_{m,1,0_{i+1}}&=\sum_{0<m<\ell\atop m\in 2\Z}mR^{\frac{\ell-m}{2}}\frac{S'}{\sqrt{R}}R^{\frac{m+1}{2}}Q_{i,1}\left(\frac{R^{(\leq
                            i+1)}}{R},\frac{S^{(\leq
  i+1)}}{\sqrt{R}},m^2\right)    \\
  &+\sum_{0<m<\ell\atop m\in 2\Z+1}mR^{\frac{\ell-m}{2}}\frac{R'}{R}R^{\frac{m+1}{2}}Q_{i,2}\left(\frac{R^{(\leq
                            i+1)}}{R},\frac{S^{(\leq
  i+1)}}{\sqrt{R}},m^2\right)\, .\label{eq:25}
\end{align}
We then apply Corollary \ref{sec:discr-integr-lemma-1} to the term
\eqref{eq:24} on the one hand, and to the sum of the term
\eqref{eq:25} and \eqref{eq:21} (for $\epsilon=1$) on the other
hand. These are of the wanted form of $R^{\frac{\ell}{2}+1}$ times an
even polynomial in the variable $\ell^2$, with coefficients in
$\mathbb{Q}[R^{(\leq i+1)}/R,S^{(\leq i+1)}/\sqrt{R}]$. Putting things
together, we see that \eqref{eq:19} indeed yields an even polynomial
expression in the odd variable $\ell$. This completes the induction
step for $D_\ell S^{(i+1)}$, and the proof of
Proposition~\ref{sec:quasi-polyn-tight}.
\end{proof}

From  the fact that $D_m$ is
a differential operator, note that
\begin{equation}
  \label{eq:26}
  D_\ell\left(\frac{R^{(i)}}{R}\right)=\frac{D_\ell
  R^{(i)}}{R}-\frac{R^{(i)}}{R}\frac{D_\ell R}{R}\, ,\qquad
D_\ell\left(\frac{S^{(i)}}{\sqrt{R}}\right)=\frac{D_\ell
  S^{(i)}}{\sqrt{R}}-\frac{1}{2}\frac{S^{(i)}}{\sqrt{R}}\frac{D_\ell R}{R}
\end{equation}
are of the form of $R^{\ell/2}$ times a parity-dependent quasi-polynomial in
$\ell^2$, and with coefficients in $\mathbb{Q}[R^{(\leq i+1)}/R,S^{(\leq
  i+1)}/\sqrt{R}]$. 

\subsection{End of the proof of the induction step}

Let us now assume that the statement of
Proposition~\ref{prop:quasipolnonbip} is true for $(g,n)$, we will
show that it is then true for $(g,n+1)$. Let us introduce the
shorthand notation
\begin{equation}
  d_{g,n} := 3g-2+n.
\end{equation}
By the induction hypothesis, for every subset $I$ of $\{1,\ldots,n\}$
there exists
\begin{equation}
  \label{eq:7}
  P^{(g)}_{n,I}\in
  \mathbb{Q}(x_1,\ldots,x_{d_{g,n}},y_1,\ldots,y_{d_{g,n}})[u_1,\ldots,u_n]\, ,
\end{equation}
such that, whenever the $\ell_i$ with $i\in I$ are odd integers, and
the $\ell_i$ with $i\notin I$ are even integers, we have
\begin{equation}
  \label{eq:28}
  \tau^{(g)}_{\ell_1,\ldots,\ell_n}=P^{(g)}_{n,I}\left(\frac{R^{(\leq
        d_{g,n})}}{R},\frac{S^{(\leq
        d_{g,n})}}{\sqrt{R}},\ell_1^2,\ldots,\ell_n^2\right)
  - c(g,n) t^{2-2g-n} \delta_{\ell_1+\cdots+\ell_n,0}\, .
\end{equation}
The $P^{(g)}_{n,I}$ correspond to the family of polynomials associated
with the parity-dependent quasi-polynomial $\mathfrak{T}^{(g)}_n$, but
the extra variables $x_1,\ldots,x_{d_{g,n}},y_1,\ldots,y_{d_{g,n}}$
are useful to keep track of the dependency in the derivatives of $R$
and $S$. In what follows, to lighten notation, we omit the mention of
these extra variables, except when we differentiate with respect to
them.

Recall that Propositions~\ref{prop:recur0} and~\ref{prop:recurgen}
assert that
\begin{align}
  \label{eq:27}
  T^{(g)}_{\ell_1,\ldots,\ell_{n+1}}=D_{\ell_{n+1}}T^{(g)}_{\ell_1,\ldots,\ell_n}
  +\sum_{i=1}^n\sum_{0<m_i<\ell_i}m_i\,
  T^{(0)}_{\ell_i,\ell_{n+1}|m_i}T^{(g)}_{\ell_1,\ldots,m_i,\ldots,\ell_n}
\end{align}
with $D_0=\frac{\partial}{\partial t}$ and $D_m$ as in \eqref{eq:Dmdefgen} for $m>0$. We plug~\eqref{eq:28} into the
above relation. Since $D_{\ell_{n+1}}$ is a differential operator that
annihilates the variable $t$ except when $\ell_{n+1}=0$, the first
term on the right-hand side of \eqref{eq:27} is equal to
\begin{multline}
  \label{eq:29}
      D_{\ell_{n+1}}\left(R^{\frac{\ell_1+\cdots+\ell_n}{2}}P^{(g)}_{n,I}(\ell_1^2,\ldots,\ell_n^2)-c(g,n)t^{2-2g-n}\delta_{\ell_1+\cdots+\ell_n,0}\right)=\\
    \frac{\ell_1+\cdots+\ell_n}{2}\cdot R^{\frac{\ell_1+\cdots+\ell_n}{2}}\frac{D_{\ell_{n+1}}R}{R}P^{(g)}_{n,I}(\ell_1^2,\ldots,\ell_n^2)\\
  + R^{\frac{\ell_1+\cdots+\ell_n}{2}}
    \sum_{r=1}^{d_{g,n}}\left(D_{\ell_{n+1}}\left(\frac{R^{(r)}}{R}\right)\frac{\partial
                            P^{(g)}_{n,I}}{\partial x_r}(\ell_1^2,\ldots,\ell_n^2)
 +D_{\ell_{n+1}}\left(\frac{S^{(r)}}{\sqrt{R}}\right)\frac{\partial
                            P^{(g)}_{n,I}}{\partial
                            y_r}(\ell_1^2,\ldots,\ell_n^2)\right)\\
                        - c(g,n+1) t^{2-2g-(n+1)} \delta_{\ell_1+\cdots+\ell_n+\ell_{n+1},0}
\end{multline}
where we set $c(g,n+1)=(2-2g-n)c(g,n)$.
 By Proposition \ref{sec:quasi-polyn-tight} and its consequence
 discussed around \eqref{eq:26}, the third line of \eqref{eq:29} is
 of the wanted form $R^{\frac{\ell_1+\cdots+\ell_{n+1}}{2}}$ times a
 quasi-polynomial in $\ell_1^2,\ldots,\ell_{n+1}^{2}$.
 We then split the last term of \eqref{eq:27} into pieces.

 Assume
first that 
that $\ell_{n+1}$ is even, in which case, the quantity
$D_{\ell_{n+1}}R/R$ appearing in the second line of \eqref{eq:29}
equals $R^{\ell_{n+1}/2}R'/R$. The last term of \eqref{eq:27} is then 
\begin{align}
  \label{eq:30} 
& \sum_{i=1\atop \ell_i\in 2\Z}^n \sum_{0<m_i<\ell_i\atop m_i\in
                   2\Z}m_iR^{\frac{\ell_i+\ell_{n+1}-m_i}{2}}\frac{R'}{R}R^{\frac{\ell_1+\cdots+m_i+\cdots+\ell_n}{2}}P^{(g)}_{n,I}(\ell_1^2,\ldots,m_i^2,\ldots,\ell_n^2)\\
  \label{eq:31}
  &+\sum_{i=1\atop \ell_i\in 2\Z}^n \sum_{0<m_i<\ell_i\atop m_i\in
    2\Z+1}m_iR^{\frac{\ell_i+\ell_{n+1}-m_i}{2}}\frac{S'}{\sqrt{R}}R^{\frac{\ell_1+\cdots+m_i+\cdots+\ell_n}{2}}P^{(g)}_{n,I\cup\{i\}}(\ell_1^2,\ldots,m_i^2,\ldots,\ell_n^2)\\
  \label{eq:32}
 & +\sum_{i=1\atop \ell_i\in 2\Z+1}^n \sum_{0<m_i<\ell_i\atop m_i\in
   2\Z}m_iR^{\frac{\ell_i+\ell_{n+1}-m_i}{2}}\frac{S'}{\sqrt{R}}R^{\frac{\ell_1+\cdots+m_i+\cdots+\ell_n}{2}}P^{(g)}_{n,I\setminus \{i\}}(\ell_1^2,\ldots,m_i^2,\ldots,\ell_n^2)\\
  \label{eq:33}
  &+ \sum_{i=1\atop \ell_i\in 2\Z+1}^n \sum_{0<m_i<\ell_i\atop m_i\in
  2\Z+1}m_iR^{\frac{\ell_i+\ell_{n+1}-m_i}{2}}\frac{R'}{R}R^{\frac{\ell_1+\cdots+m_i+\cdots+\ell_n}{2}}P^{(g)}_{n,I}(\ell_1^2,\ldots,m_i^2,\ldots,\ell_n^2)\, .
\end{align}
By the discrete integration Corollary \ref{sec:discr-integr-lemma-1}, we
see that \eqref{eq:31} and \eqref{eq:32} are polynomial expressions in
their respectively even and odd variables $\ell_i^2$, while we should
combine \eqref{eq:30} and \eqref{eq:33} with the corresponding first terms
of \eqref{eq:29} (which is made possible by the fact that
$D_{\ell_{n+1}}$ has a factor $R'/R$) to get similar polynomial
expressions.

If, on the other hand, $\ell_{n+1}$ is odd, then $D_{\ell_{n+1}}R/R$ 
equals $R^{\ell_{n+1}/2}S'/\sqrt{R}$. The last term of \eqref{eq:27} is then 
\begin{align}
\label{eq:34}
& \sum_{i=1\atop \ell_i\in 2\Z}^n \sum_{0<m_i<\ell_i\atop m_i\in
                   2\Z}m_iR^{\frac{\ell_i+\ell_{n+1}-m_i}{2}}\frac{S'}{\sqrt{R}}R^{\frac{\ell_1+\cdots+m_i+\cdots+\ell_n}{2}}P^{(g)}_{n,I}(\ell_1^2,\ldots,m_i^2,\ldots,\ell_n^2)\\
\label{eq:35}
  &+\sum_{i=1\atop \ell_i\in 2\Z}^n \sum_{0<m_i<\ell_i\atop m_i\in
    2\Z+1}m_iR^{\frac{\ell_i+\ell_{n+1}-m_i}{2}}\frac{R'}{R}R^{\frac{\ell_1+\cdots+m_i+\cdots+\ell_n}{2}}P^{(g)}_{n,I\cup\{i\}}(\ell_1^2,\ldots,m_i^2,\ldots,\ell_n^2)\\
\label{eq:36}
 & +\sum_{i=1\atop \ell_i\in 2\Z+1}^n \sum_{0<m_i<\ell_i\atop m_i\in
   2\Z}m_iR^{\frac{\ell_i+\ell_{n+1}-m_i}{2}}\frac{R'}{R}R^{\frac{\ell_1+\cdots+m_i+\cdots+\ell_n}{2}}P^{(g)}_{n,I\setminus\{i\}}(\ell_1^2,\ldots,m_i^2,\ldots,\ell_n^2)\\
\label{eq:37}
  &+ \sum_{i=1\atop \ell_i\in 2\Z+1}^n \sum_{0<m_i<\ell_i\atop m_i\in
  2\Z+1}m_iR^{\frac{\ell_i+\ell_{n+1}-m_i}{2}}\frac{S'}{\sqrt{R}}R^{\frac{\ell_1+\cdots+m_i+\cdots+\ell_n}{2}}P^{(g)}_{n,I}(\ell_1^2,\ldots,m_i^2,\ldots,\ell_n^2)\, .
\end{align}
By the discrete integration Corollary \ref{sec:discr-integr-lemma-1}, we
see that \eqref{eq:35} and \eqref{eq:36} are polynomial expressions in
their respectively even and odd variables $\ell_i^2$, while we should
combine \eqref{eq:34} and \eqref{eq:37} with the corresponding first terms
of \eqref{eq:29} to get similar polynomial expressions.

From this discussion we conclude that the statement of
Proposition~\ref{prop:quasipolnonbip} holds for $(g,n+1)$, which
concludes the proof of the induction step.

\section{Quasi-polynomiality in the general case: initialisation}
\label{sec:proof-prop-refpr}

We now initialise the induction for proving
Proposition~\ref{prop:quasipolnonbip}: for $g=0$, $g=1$ and $g \geq 2$
we should respectively show that the statement holds for $n=3$, $n=1$
and $n=0$.

The fact that the statement holds in the planar case $(g,n)=(0,3)$ is a direct
consequence of Theorem~\ref{thm:triskellnonbip} and of the
relation~\eqref{eq:tauT} between the $T$'s and the $\tau$'s.  To treat
the case of higher genera, we will need some results from the theory
of topological recursion presented in \cite{Eynard2016}, and more
precisely, we will need the expression of the generating function of
maps in genus $g\geq 1$ based on the so-called \emph{method of
  moments}~\cite{Ambjoern1995}. Our approach can be seen as an
extension of the analysis of Budd
\cite[Section~3.4]{budd2020irreducible} to the non bipartite case (but
only considering the non-irreducible case).  Recall that
$F^{(g)}=F^{(g)}_\varnothing =T^{(g)}_\varnothing=\tau^{(g)}_\varnothing$
denote the generating series of maps of genus $g$ without boundaries.

\subsection{Genus one}\label{sec:torus}

Let us first deal with the case $g=1$. Our main input will be
the following. 

\begin{thm}
  \label{sec:torus-1}
  The generating function for maps of genus one without boundary is
  given by
  \begin{equation}
    \label{eq:2700}
  F^{(1)}=\frac{1}{24}\ln\left(\left(\frac{R'}{R}\right)^2-\left(\frac{S'}{\sqrt{R}}\right)^2\right)+\frac{\ln(t)}{12}\,
  .
\end{equation}
\end{thm}

From this, we can easily deduce Proposition \ref{prop:quasipolnonbip}
in the case $(g,n)=(1,1)$. Indeed, note that by
Propositions~\ref{prop:recur0} and~\ref{prop:recurgen}, we have
$T^{(1)}_\ell=D_\ell F^{(1)}$, which equals
\begin{equation}
  \label{eq:38}
  T^{(1)}_\ell=\frac{\frac{R'}{R}D_\ell\left(\frac{R'}{R}\right)-\frac{S'}{\sqrt{R}}D_\ell\left(\frac{S'}{\sqrt{R}}\right)}{12\left(\left(\frac{R'}{R}\right)^2-\left(\frac{S'}{\sqrt{R}}\right)^2\right)}+\frac{\delta_{\ell,0}}{12t}
\end{equation}
and by Proposition \ref{sec:quasi-polyn-tight}---see also
\eqref{eq:26}---this is of the wanted form of $R^{\ell/2}$ times a quasipolynomial in
$\ell$ with quasiperiod $2$, whose coefficients are rational functions
of $R^{(i)}/R,S^{(i)}/\sqrt{R},i\in \{1,2\}$, minus a pathological term for $\ell=0$.

\begin{proof}[Proof of Theorem \ref{sec:torus-1}]
  Let us recall Eynard's notation\footnote{Eynard puts a bound on the degrees,
and excludes faces of degrees $1$ and $2$, but this restriction is in
fact inessential, as it results from Eynard's choice of considering
the map generating functions as formal power series in the
vertex weight $t$ whose coefficients are polynomials in the face
weights. Working instead with a weight per edge lifts this restriction
since the set of maps with a given number of edges is finite, and we
believe that the reasonings of \cite{Eynard2016} extend to this
setting.} from \cite{Eynard2016}. 
We set (p.58 and Theorem 3.1.2 p.62) $x(z)=S+\sqrt{R}(z+1/z)$
($S=\alpha,R=\gamma^2$), and (pp.59-60, Definition 3.1.1 p.63 and
Theorem 3.1.2 p.62) 
\begin{align}
\label{eq:400}
  y(z)&=W_1^{(0)}(x(z))-\frac{V'(x(z))}{2}\\
 &=       -\frac{1}{2}\sum_{k\geq 1}u_k(z^k-z^{-k})\, ,
\end{align}
where 
\begin{equation}
  \label{eq:300}
  u_k=S\delta_{k,0}+\sqrt{R}\delta_{k,1}-\sum_{i\geq
    k+1}t_i\sum_{j=k}^{\lfloor(i+k-1)/2\rfloor}\binom{i-1}{j,j-k,i-1+k-2j}R^{j-k/2}S^{i-1+k-2j}\, .
\end{equation}
Note that the latter is related to the generating function $V_{k-1}$ for $(k-1)$-slices
introduced in \cite{irredmaps}: for $k\geq 2$, we have $V_{k-1}=-R^{k/2}u_k$. On the other hand, 
we have $u_0=0$ and $u_1=t/\sqrt{R}$, as seen by applying the multinomial formulas in 
\eqref{eq:defRS}. 
We define the $0$-th order \emph{moments} by 
\begin{equation}
  \label{eq:500}
  M_{+,0}=\frac{-y'( 1)}{\sqrt{R}}\, ,\qquad M_{-,0}=\frac{-y'(-
    1)}{\sqrt{R}}. 
\end{equation}
Now, \cite[Theorem 3.4.6]{Eynard2016}  expresses the generating
function of toric maps in terms of $M_{\pm,0}$ as 
\begin{equation}
  \label{eq:40}
  F^{(1)}=-\frac{1}{24}\ln\left(\frac{\gamma^2y'(1)y'(-1)}{t^2}\right)=-\frac{1}{24}\ln\left(\frac{R^2M_{+,0}M_{-,0}}{t^2}\right)\,
  ,
\end{equation}
On the other hand,   
\cite[Theorem 3.3.4]{Eynard2016} gives 
\begin{equation}
 \begin{split} \omega_3^{(0)}(z_1,z_2,z_3)=&
    \frac{1}{2RM_{+,0}}\frac{1}{(z_1-1)^2(z_2-1)^2(z_3-1)^2}\\
 &-\frac{1}{2RM_{-,0}}\frac{1}{(z_1+1)^2(z_2+1)^2(z_3+1)^2}\, ,\end{split}
\end{equation}
which we now know to count pairs of pants. Namely, this is
\begin{equation}
  \label{eq:69}
  \sum_{\ell_1,\ell_2,\ell_3\geq
  1}\left(\frac{1}{2RM_{+,0}}+\frac{(-1)^{\ell_1+\ell_2+\ell_3}}{2RM_{-,0}}\right)\frac{\ell_1\ell_2\ell_3}{z_1^{\ell_1+1}z_2^{\ell_2+1}z_3^{\ell_3+1}}\,
.
\end{equation}
Comparing with our formula for tight pairs of pants \eqref{eq:Tabc}, we obtain 
\begin{equation}
  \label{eq:70}
  \tau^{(0)}_{\ell_1,\ell_2,\ell_3}=\left\{\begin{split}
\frac{1}{2RM_{+,0}}+\frac{1}{2RM_{-,0}}&=\frac{R'}{R} \mbox{ if
}\ell_1+\ell_2+\ell_3\in 2\Z\\
\frac{1}{2RM_{+,0}}-\frac{1}{2RM_{-,0}}&=\frac{S'}{\sqrt{R}} \mbox{ if
}\ell_1+\ell_2+\ell_3\in 2\Z+1
\end{split}
\right. ,
\end{equation}
where $\ell_1,\ell_2,\ell_3>0$. 
This shows that
\begin{equation}
  \label{eq:39}
  \frac{1}{RM_{+,0}}=-\frac{1}{\gamma y'(1)}=\frac{R'}{R}+\frac{S'}{\sqrt{R}}\, ,\qquad
\frac{1}{RM_{-,0}}=-\frac{1}{\gamma
  y'(-1)}=\frac{R'}{R}-\frac{S'}{\sqrt{R}}\, .
\end{equation}
Plugging this into \eqref{eq:40} yields Theorem \ref{sec:torus-1}. 
\end{proof}

\subsection{Higher moments, higher genera}

The generating functions for maps in higher genera are given by higher
order moments, which are given (\cite[p.64]{Eynard2016}), for $h\geq 1$, by
\begin{equation}
  \label{eq:600}
  M_{\pm,h}=\frac{-1}{R^{(h+1)/2}M_{\pm,0}}\sum_{k\geq h+1}(\pm
  1)^{k+h+1}u_k\binom{k+h}{2h+1}\, .
\end{equation}
We define the renormalized version
\begin{equation}
  \label{eq:700}
  \bar{M}_{\pm,h}=R^{(h+1)/2}M_{\pm,h}M_{\pm,0}=-\sum_{k\geq h+1}(\pm
  1)^{k+h+1}u_k\binom{k+h}{2h+1}\, .
\end{equation}
One should note that \cite[Corollary 3.5.1]{Eynard2016} introduces the
quantities $R^{h/2}M_{\pm,h}$ that are called ``dimensionless'', we
will see that these are indeed natural quantities to consider, due to
the following result\footnote{There seems to be an inaccuracy in the
  statement of \cite[Corollary 3.5.1]{Eynard2016}. Indeed, the latter
  is stated without the boundary term involving $t^{2-2g}$ only, which
  is however present in other similar statements such as
  [\emph{ibid.}, Theorems~3.4.3 and 3.4.9]. In these statements, the
  coefficient $\frac{B_{2g}}{2g(2-2g)}$ is indeed equal to
  $\chi(\mathcal{M}_{g,0})$ as wanted.
  It seems that this boundary term got omitted at some stage.}:

\begin{thm}{\cite[Corollary 3.5.1]{Eynard2016}}
\label{sec:high-moments-high}
  There exists $\mathcal{P}_g\in
  \mathbb{Q}[x_+,x_-,(y_{+,h},y_{-,h},1\leq h\leq 3g-3)]$ that is homogeneous
  of degree $2g-2$ in the first two variables, and such that
  $\mathcal{P}_g(x_+,x_-,(\mu^h,\mu^h,1\leq h\leq 3g-3))$ is of degree
  $3g-3$ in $\mu$, and $c(g,0) \in \mathbb{Q}$, such that
  \begin{equation}
    \label{eq:68}
      F^{(g)}=\mathcal{P}_g\left(\frac{1}{RM_{+,0}},\frac{1}{RM_{-,0}},(R^{h/2}M_{+,h},R^{h/2}M_{-,h},1\leq
        h\leq 3g-3)\right)-c(g,0) t^{2-2g}\, .
    \end{equation}
  \end{thm}

We would like to show that, for $h\geq 1$, the ``dimensionless'' moments $R^{h/2}M_{\pm,h}$ are
rational functions of $R^{(k)}/R,S^{(k)}/\sqrt{R}$ for
$1\leq k\leq h+1$.  Our goal will be to get rid of the variables $t_i$
appearing implicitly in the $u_k$ in the previous expression, and
express this purely in terms of $R,S$ and their derivatives. Our
approach is inspired by that of \cite[Section~3.4]{budd2020irreducible}.

First, for  $h\geq 1$, we rewrite, using \eqref{eq:700} and \eqref{eq:300}, 
\begin{multline}
  \label{eq:2}
  \bar{M}_{\pm,h}
 =\!\!\sum_{k\geq h+1}\!\!(\pm 1)^{k+h+1} \sum_{i\geq 1}t_i \binom{k+h}{2h+1}\sum_{j}\binom{i-1}{j,j-k,i-1+k-2j}R^{j-k/2}S^{i-1+k-2j}\\
  =(\pm 1)^h\sum_{i\geq h+2}t_i\sum_{l}(\pm 1)^{i-l}R^{\frac{i-l-1}{2}}S^l\sum_{k}\binom{k+h}{2h+1}\binom{i-1}{\frac{i-1+k-l}{2},\frac{i-1-k-l}{2},l}
\end{multline}
where the second expression comes by changing the variable $j$ into $l=i-1+k-2j$ and permuting the sums. To alleviate notations, we do not specify some summation ranges as they are naturally enforced by the vanishing of the binomial and multinomial coefficients and by
the requirement that their arguments are nonnegative integers. In
particular, in the second expression, $k+1$ has the same parity as $i-l$, which allows to
reorganize the term $(\pm 1)^{k+h+1}$ as displayed. 

\subsubsection{A binomial identity}
\label{sec:binomial-identity}

Following \cite[Lemma
7]{budd2020irreducible}, it is convenient to simplify the sum over $k$
appearing in \eqref{eq:2}. Let us first rewrite
\begin{equation}
  \label{eq:20}
  \sum_{k}\binom{k+h}{2h+1}\binom{i-1}{\frac{i-1+k-l}{2},\frac{i-1-k-l}{2},l}=
  \binom{i-1}{l}\sum_{k}\binom{k+h}{2h+1}\binom{i-l-1}{\frac{i-l-1+k}{2}}\, ,
\end{equation}
where we recall that the sum is over nonnegative integers $k$  with
the same parity as $i-l-1$. Writing $j=i-l$, we treat the sum over $k$ on the
right-hand side by the following lemma. 

\begin{lem}
  \label{sec:higher-values-p}
  For any nonnegative integer $h$, there exist two polynomials $Q_h^{[0]}$ and $Q_h^{[1]}$ of
degree $h+1$ such that, denoting by $\epsilon=\epsilon(j)\in \{0,1\}$ the parity of
  $j$, we have
  \begin{equation}
    \label{eq:1100}
    \sum_{k}\binom{k+h}{2h+1}\binom{j-1}{\frac{j-1+k}{2}}
    =\binom{j-1}{\frac{j-\epsilon}{2}}Q_h^{[\epsilon]}\left(\frac{j-\epsilon}{2}\right)
  \end{equation}
where we sum over nonnegative $k$ such that $k+\epsilon$ is odd. 
\end{lem}

Thanks to this lemma, the sum over $k$ in \eqref{eq:2} is expressed as 
\begin{equation}
    \label{eq:1100bis}
    \sum_{k}\binom{k+h}{2h+1}\binom{i-1}{\frac{i-1+k-l}{2},\frac{i-1-k-l}{2},l}
    =\binom{i-1}{\frac{i-l-\epsilon}{2},\frac{i-l+\epsilon-2}{2},l}Q_h^{[\epsilon]}\left(\frac{i-l-\epsilon}{2}\right)\, ,
  \end{equation}
which leads to the expression 
\begin{equation}\label{eq:52}
  \bar{M}_{\pm,h}  =(\pm 1)^h\sum_{\epsilon\in \{0,1\}}(\pm 1)^\epsilon
  \sum_{i\geq 2}t_i\sum_{l=i +\epsilon\, \mathrm{mod}\, 2}
               R^{\frac{i-l-1}{2}}S^l\binom{i-1}{\frac{i-l-\epsilon}{2},\frac{i-l+\epsilon-2}{2},l}Q^{[\epsilon]}_h\left(\frac{i-l-\epsilon}{2}\right)
\end{equation}
that we will further study in the next subsection. In passing, we note
that
\begin{equation}
  \label{eq:50}
  Q_h^{[0]}(1)=0\quad \mbox{ for every }\quad h\geq 1\, ,
\end{equation}
as seen by taking $j=2$ in \eqref{eq:1100}: on the left-hand side, the
sum is over all positive odd $k$, and each term vanishes.

\begin{proof}[Proof of Lemma \ref{sec:higher-values-p}]
    Assume first that $j$ is even, that is, $\epsilon=0$, so that the summation is over
  odd $k=2k'+1$. Then, writing $2j'=j$, Equation~\eqref{eq:1100} is rewritten as
  \begin{equation}
    \label{eq:1300}
    \sum_{k'\geq 0}\binom{2k'+h+1}{2h+1}\binom{2j'-1}{j'+k'}=\binom{2j'-1}{j'}Q_h^{[0]}(j')\, ,
  \end{equation}
where $Q_h^{[0]}=Q_h(0,\cdot)$ is the polynomial of degree $h+1$ of \cite[Lemma~7]{budd2020irreducible}
specialized at $b=0$. Alternatively, we
may recover the expression of $Q_h^{[0]}$ by the following independent argument. First
note that, in the notation of \cite{polytightmaps}, we have
\begin{equation}
  \label{eq:1800}
\binom{x+h}{2h+1}=\frac{h!^2}{(2h+1)!}xp_h(x)\, , 
\end{equation}
so that, using the
notation
\begin{equation}
  \label{eq:71}
  \mathsf{A}_{\ell,k}=\frac{2k}{2\ell}\binom{2\ell}{\ell+k}\, ,\quad \ell,k\in
  \frac{1}{2}\mathbb{Z}
\end{equation}
from \cite[Equation (3.2)]{polytightmaps}, we can rewrite the left hand side
of \eqref{eq:1300} in
the form
\begin{multline}
  \label{eq:1500}
  \frac{h!^2}{(2h+1)!}\sum_{k'}2(k'+1/2)p_h(2(k'+1/2))\binom{2(j'-1/2)}{j'-1/2+k'+1/2}\\
  =
   (2j'-1)\frac{h!^2}{(2h+1)!}\sum_{k'}\mathsf{A}_{j'-1/2,k'+1/2}p_h(2(k'+1/2))\, .
 \end{multline}
Taking $m=k'+1/2$, the polynomial $p_h(2m)$ is of degree $h$ in $m^2$, so it can be
 expressed in the basis of the polynomials $\tilde{p}_r(m),0\leq r\leq
 h$ defined in \cite[Equation (2.18)]{polytightmaps}, in
 the form $p_h(2m)=\sum_{r=0}^h\tilde{\alpha}_{h,r}\tilde{p}_r(m)$, which, based on
\cite[Equation (3.9)]{polytightmaps}, yields the
 expression
\begin{multline}
   \label{eq:1600}
   (2j'-1)\frac{h!^2}{(2h+1)!}\sum_{k'}\sum_{r=0}^h
   \tilde{\alpha}_{h,r}\mathsf{A}_{j'-1/2,k'+1/2}\, \tilde{p}_r(k'+1/2)\\
   =(2j'-1)\frac{h!^2}{(2h+1)!}\sum_{r=0}^h\tilde{\alpha}_{h,r}\sum_{k'}
     \mathsf{A}_{j'-1/2,k'+1/2}\, \tilde{p}_r(k'+1/2)\\
  =(2j'-1)\frac{h!^2}{(2h+1)!}\sum_{r=0}^h \tilde{\alpha}_{h,r}
    \binom{j'-1}{r}\binom{2j'-2}{j'-1}\\
   =\binom{2j'-1}{j'} \frac{h!^2}{(2h+1)!}
j'\sum_{r=0}^h \tilde{\alpha}_{h,r}\, 
       \binom{j'-1}{r}
\end{multline}
which, since $j'=j/2$, is indeed of the wanted form
$\binom{j-1}{j/2}Q_h^{[0]}(j/2)$
for some polynomial $Q^{[0]}_h$ of degree $h+1$. 

For $j$ odd, i.e.\ $\epsilon=1$, we argue in a similar (and slightly simpler) way.
The summation is now over even $k=2k'$, and we write $2j'=j-1$, so
that the left-hand side of \eqref{eq:1100} reads 
\begin{equation}
  \label{eq:15000}
   \sum_{k'}\binom{2k'+h}{2h+1}\binom{2j'}{j'+k'}=
   2j'\frac{h!^2}{(2h+1)!}\sum_{k'}\mathsf{A}_{j',k'}p_h(2k')\, .
 \end{equation}
 where we use \eqref{eq:1800} and \eqref{eq:71} to pass to the right-hand side.
 This time, we take $m=k'$ and we express the polynomial $p_h(2m)$ in the basis of the polynomials $p_r(m),0\leq r\leq h$, in
 the form $p_h(2m)=\sum_{r=0}^h\alpha_{h,r}p_r(m)$, which yields, now
 using \cite[Equation (3.7)]{polytightmaps}, the
 expression
\begin{multline}
   \label{eq:1600bis}
   2j'\frac{h!^2}{(2h+1)!}\sum_{k'}\sum_{r=0}^h \alpha_{h,r}\mathsf{A}_{j',k'}p_r(k')\\
   =2j'\frac{h!^2}{(2h+1)!}\sum_{r=0}^h\alpha_{h,r}\sum_{k'}
     \mathsf{A}_{j',k'}p_r(k')\\
  =2j'\frac{h!^2}{(2h+1)!}\sum_{r=0}^h\alpha_{h,r}
    \binom{j'-1}{r}\binom{2j'-1}{j'}\\
   =\binom{2j'}{j'} \frac{h!^2}{(2h+1)!}
     j'\sum_{r=0}^h\alpha_{h,r}
    \binom{j'-1}{r}\, ,
\end{multline}
which is again of the wanted form
$\binom{j-1}{(j-1)/2}Q_h^{[1]}((j-1)/2)$ for some polynomial
$Q_h^{[1]}$ of degree $h+1$. 
\end{proof}

\subsubsection{Relating the moments with the recursions determining $R,S$}
\label{sec:expr-moments-terms}

The expression \eqref{eq:52} for the moments can be compared to the
equations \eqref{eq:defRS} defining
$R$ and $S$, which we may write in the more compact and symmetric form
\begin{equation}
  \label{eq:800}
  \bZ(\bU(\bt))=\bt
\end{equation}
where $\bt=(t_0,t_1)$, $t_0=t$, $\bU=(R,S)$ and $\bZ=(Z_0,Z_1)$ is the pair of bivariate series defined by
\begin{equation}
  \label{eq:900}
  Z_0(r,s)
  =r-\sum_{i\geq 2}t_i\sum_{l=i\, \mathrm{mod}\, 2}\binom{i-1}{\frac{i-l}{2},\frac{i-l}{2}-1,l}
  r^{\frac{i-l}{2}}s^{l}
\end{equation}
and
\begin{equation}
  \label{eq:1000}
  Z_1(r,s)
 =s-\sum_{i\geq
            2}t_i\sum_{l=i +1\, \mathrm{mod}\, 2}\binom{i-1}{\frac{i-l-1}{2},\frac{i-l-1}{2},l}r^{\frac{i-l-1}{2}}s^{l} \, .
\end{equation}
Combining this with \eqref{eq:52},  we obtain that, for every $h\geq
1$, 
\begin{equation}
  \label{eq:62}
   (\pm 1)^h\bar{M}_{\pm,h}  =-\sum_{\epsilon\in \{0,1\}}(\pm
   1)^\epsilon R^{\frac{\epsilon-1}{2}}
  Q_h^{[\epsilon]}(r\partial_r)Z_\epsilon(r,s)|_{r=R,s=S}\, ,
             \end{equation}
             where we note that the term $r$ in \eqref{eq:900} is
             annihilated by the differential operator $Q_h^{[0]}(r\partial_r)$ thanks to
             \eqref{eq:50}. 
By Leibniz's formula, this means that the renormalized moments are
linear combinations of the expressions
\begin{equation}
  \label{eq:2100}
  R^{k+\frac{\epsilon-1}{2}}\frac{\partial^kZ_\epsilon}{\partial r^k}(R,S)\, ,\qquad 1\leq k\leq
  h+1\, ,\quad \epsilon\in \{0,1\}.
\end{equation}
We now discuss the structure of these
expressions.

\subsubsection{Structure of derivatives of $\bZ$}\label{sec:closer-look-at}

To this end, we need to better
understand the partial derivatives of $R,S$ with respect to
$t_0,t_1$. We recall that $R^{(k)}$ and
$S^{(k)}$ correspond to the derivatives with respect to $t=t_0$.

\begin{prop}
  \label{sec:expr-deriv-bz}
  There exists a family of polynomials $P^{(a,b)}(x_k,y_k,1\leq k\leq
  a+b)$ such that, for any integers $a,b \geq 0$, we have
  \begin{equation}
    \label{eq:2200}
 \frac{\partial^{a+b}
R}{\partial t_0^a\partial
t_1^b}=R^{\frac{b}{2}+1}P^{(a,b)}\left(\frac{R^{(k)}}{R},\frac{S^{(k)}}{\sqrt{R}},1\leq
  k\leq a+b\right)
\end{equation}
  and, if moreover $(a,b) \neq (0,0)$,
  \begin{equation}
    \label{eq:2300}
 \frac{\partial^{a+b}
S}{\partial t_0^a\partial
t_1^b}=R^{\frac{b+1}{2}}P^{(a+1,b-1)}\left(\frac{R^{(k)}}{R},\frac{S^{(k)}}{\sqrt{R}},1\leq
  k\leq a+b\right)\, .
\end{equation}
It is defined by induction as follows: we have
$P^{(a,0)}=x_a$ for every $a\geq 0$, with the convention that $x_0=1$,
$P^{(a+1,-1)}=y_a$ for every $a\geq 1$, and, for every $a,b\geq 0$,
\begin{equation}
  \label{eq:78}
      P^{(a,b+1)}
      = \left(\left(\frac{b}{2}+1\right)y_1 
        +
\sum_{k=1}^{a+b}
\left(\sum_{i=0}^{k-1}\binom{k}{i}x_iy_{k+1-i}\partial_{x_k}
+\left(x_{k+1}-\frac{y_1y_k}{2}\right)\partial_{y_k} \right)\right)
P^{(a,b)}\, .
\end{equation}
Moreover, the polynomial $P^{(a,b)}(\bar{x}_k^k,\bar{y}_k^k,1\leq
k\leq a+b)$ in the variables $\bar{x}_k,\bar{y}_k$ is homogeneous of
degree $a+b$. 
\end{prop}

\begin{proof}
We first make the trivial observation that for $b=0$, \eqref{eq:2200} holds with
$P^{(a,0)}=x_a$ for $a\geq 0$, and \eqref{eq:2300} holds with
$P^{(a+1,-1)}=y_a$ for $a\geq 1$.  Moreover, by the combinatorial definition of $R$ and $S$ given in
  Section~\ref{sec:defs}, the derivatives in \eqref{eq:2200} and \eqref{eq:2300} are of the form
  $T^{(0)}_{1,\ldots,1,0,\ldots 0}$ in both cases, with $a$ terms
  equal to $0$ and $b+2$ terms equal to $1$ in the first case, and
  with $a+1$ terms equal to $0$ and $b+1$ terms equal to $1$ in the
  second case. This shows that, for $a\geq 0$ and $b \geq 1$,
  \eqref{eq:2300} is nothing but \eqref{eq:2200} applied
  to $a+1$ and $b-1$ instead of $a,b$. 
Furthermore, we deduce from this discussion and Theorem A.1 the identities
  \begin{equation}
    \label{eq:63}
    \frac{\partial R}{\partial t_1}=RS'\, ,\qquad \frac{\partial 
      S}{\partial t_1}=R'\, , 
  \end{equation}
which are consistent with the particular values $P^{(0,1)}=y_1$ and $P^{(1,0)}=x_1$. 

  In view of these preliminary remarks, it now suffices to
  establish~\eqref{eq:2200} for general $b$, which we will do by
  induction. Assume that it holds at rank $b \geq 0$, and let us
  establish it at rank $b+1$. To this end, we write
\begin{equation}
  \label{eq:64}
  \frac{\partial^{a+b+1}
R}{\partial t_0^a\partial
t_1^{b+1}}= \frac{\partial}{\partial t_1}\left(R^{\frac{b}{2}+1}P^{(a,b)}\left(\frac{R^{(k)}}{R},\frac{S^{(k)}}{\sqrt{R}},1\leq
  k\leq a+b\right)\right)\, , 
\end{equation}
and we expand the right-hand side using the Leibniz and chain
rules. The first term in the expansion, obtained by differentiating $R^{b/2+1}$, is
equal to 
\begin{equation}
  \label{eq:67}
 \left(\frac{b}{2}+1\right) R^{\frac{b}{2}} \frac{\partial R}{\partial t_1} P^{(a,b)} = \left(\frac{b}{2}+1\right) R^{\frac{b+1}{2}+1}\frac{S'}{\sqrt{R}}P^{(a,b)}\,
  .
\end{equation}
The term obtained by
differentiating with respect to the variable $x_k=R^{(k)}/R$ in $P^{(a,b)}$
is equal to 
\begin{equation}
  \label{eq:75}
  R^{\frac{b}{2}+1}\partial_{x_k} P^{(a,b)}\frac{\partial}{\partial t_1}\left(\frac{R^{(k)}}{R}\right)\, ,
\end{equation}
where we may write
\begin{equation}
  \label{eq:76}
  \begin{split}
  \frac{\partial}{\partial
    t_1}\left(\frac{R^{(k)}}{R}\right)&=\frac{\frac{\partial R^{(k)}}{\partial t_1}}{R}-\frac{\frac{\partial R}{\partial t_1}  R^{(k)}}{R^2}
 =\frac{\frac{\partial^k (RS')}{\partial t_0^k}}{R}-\frac{S'R^{(k)}}{R}\\
 &=\sqrt{R}
   \sum_{i=0}^{k-1} \binom{k}i \frac{R^{(i)}}{R}\frac{S^{(k+1-i)}}{\sqrt{R}}
   \end{split}
\end{equation}
by \eqref{eq:63} and Leibniz's formula. 
Finally, the term obtained by
differentiating with respect to the variable $y_k=S^{(k)}/\sqrt{R}$ in
$P^{(a,b)}$ is equal to
\begin{equation}
  \label{eq:77}
   R^{\frac{b}{2}+1}\partial_{y_k} P^{(a,b)}\frac{\partial}{\partial t_1}\left(\frac{S^{(k)}}{\sqrt{R}}\right)\, ,
\end{equation}
where we may write
\begin{equation}
  \label{eq:79}
  \frac{\partial}{\partial
    t_1}\left(\frac{S^{(k)}}{\sqrt{R}}\right)=\frac{\frac{\partial
      S^{(k)}}{\partial t_1}}{\sqrt{R}}-\frac{\frac{\partial
      R}{\partial t_1} 
    S^{(k)}}{2R^{3/2}}
=\sqrt{R}\left(\frac{R^{(k+1)}}{R}-\frac{1}{2}\frac{S'}{\sqrt{R}}\frac{S^{(k)}}{\sqrt{R}}\right)\, .
\end{equation}
Putting \eqref{eq:67}--\eqref{eq:79} together, we
obtain that \eqref{eq:2200} is valid for $(a,b+1)$, where
$P^{(a,b+1)}$ is given by \eqref{eq:78}, as wanted. The statement
about the degree and homogeneity of
$P^{(a,b)}(\bar{x}_k^k,\bar{y}_k^k,1\leq k\leq a+b)$ is readily
derived by induction, using \eqref{eq:78}. 
\end{proof}

At this point, we use the formula \eqref{eq:2400} for derivatives of
inverses discussed in Appendix~\ref{sec:higherdiff}, applied to
$\mathbf{f}=\bU$ and $\mathbf{g}=\bZ$ (for $n=2$ in the notation
therein)
\begin{equation}
  \label{eq:41}
  \frac{\partial^k Z_\epsilon}{\partial r^k}(\bU(\bt))=\sum_{\tau\in
    \mathcal{T}^{\mathbf{0}}_\epsilon(k)}\prod_{v\in
    \tau}(d_\bt\bU)^{-1}_{i_v,j_v}\prod_{v\in
    \tau: k_v(\tau)>0}\frac{1}{k_v(\tau)!}\frac{\partial^{k_v(\tau)}U_{j_v}}{\partial
  t_{i_{v1}}\cdots\partial t_{i_{vk_v(\tau)}}}(\bt)
\end{equation}
where $\mathcal{T}^{\mathbf{0}}_\epsilon(k)$ is the family of planted P\'olya
trees $\tau$ with $k$ leaves, and where the bottom and top half-edges of the
edge $e_v$ below the vertex $v\in \tau$ are labeled by elements
$i_v,j_v\in \{0,1\}$, in such a way that 
$i_\varnothing=\epsilon$, and $i_v=0$ for every leaf $v\in \tau$.

Let us look more closely at the contribution of a given
tree $\tau\in \mathcal{T}^{\mathbf{0}}_\epsilon(k)$ to the preceding sum.
Note that by \eqref{eq:13}, we have 
\begin{equation}
  \label{eq:42}
  d_\bt\bU=\left(
  \begin{array}{cc}
    \frac{\partial R}{\partial t_0} & \frac{\partial R}{\partial t_1}   \\
    \frac{\partial S}{\partial t_0} & \frac{\partial S}{\partial t_1} 
  \end{array}
\right)=\left(\begin{array}{cc}
    R' & RS'   \\
    S' & R' 
  \end{array}
\right)
\end{equation}
We use
the comatrix formula to compute
\begin{equation}
  \label{eq:43}
  (d_\bt\bU^{-1})_{i,j}=R^{-1-\frac{i}{2}+\frac{j}{2}}\frac{\mathcal{D}(i,j)}{\left(\frac{R'}{R}\right)^2-\left(\frac{S'}{\sqrt{R}}\right)^2}\, ,
\end{equation}
where $\mathcal{D}(i,j)$ is $R'/R$ if $i+j$ is even, and
$-S'/\sqrt{R}$ otherwise. 
Hence, the contribution of
edges of $\tau$---the first product in \eqref{eq:41}---is 
\begin{equation}
  \label{eq:44}
  \prod_{v\in \tau}R^{-1-\frac{i_v}{2}+\frac{j_v}{2}}\frac{\mathcal{D}(i_v,j_v)}{\left(\frac{R'}{R}\right)^2-\left(\frac{S'}{\sqrt{R}}\right)^2}\,
  ,
\end{equation}
while the contribution of internal vertices---the second product---is,
by Proposition \ref{sec:expr-deriv-bz},
\begin{equation}
  \label{eq:45}
  \prod_{v\in \tau:k_v(\tau)>0}\frac{R^{\frac{\sum_{l=1}^{k_v(\tau)}i_{vl}}{2}+1-\frac{j_v}{2}}}{k_v(\tau)!}P^{\left(\sum_{l=1}^{k_v(\tau)}(1-i_{vl})+j_v,\sum_{l=1}^{k_v(\tau)}i_{vl}-j_v\right)}\left(\frac{R^{(l)}}{R},\frac{S^{(l)}}{\sqrt{R}},1\leq
    l\leq k_v(\tau)\right)\, .
\end{equation}
The exponents in $R$ have many cancellations and result in a global
exponent of 
\begin{equation}
  \label{eq:15}
  -k-i_\varnothing/2+\sum_{v\in
    \tau:k_v(\tau)=0}j_v/2=-k-\epsilon/2\, ,
\end{equation}
since $\tau\in \mathcal{T}^{\mathbf{0}}_\epsilon(k)$. The rest of the
contribution of $\tau$ is a homogeneous polynomial of degree
$2\#\tau-1$ (in the sense of Proposition \ref{sec:expr-deriv-bz}, that
is, where the variables of index $l$ are considered to have degree
$l$), times $((R'/R)^2-(S'/\sqrt{R})^2)^{-\#\tau}$. Note that 
$\#\tau$ is maximal and equal to $2k-1$ when $\tau$ is a binary 
tree. Hence, we reduce the sum in the right-hand side of \eqref{eq:41} to a
common denominator by multiplying and dividing the contribution of
$\tau$ by $((R'/R)^2-(S'/\sqrt{R})^2)^{2k-1-\#\tau}$, which is thus of
the form 
\begin{equation}
  \label{eq:72}
  R^{-k-\epsilon/2}\frac{\Pi^{\tau}_{\pm,k}\left(\frac{R^{(l)}}{R},\frac{S^{(l)}}{\sqrt{R}},1\leq
  l\leq
  k\right)}{\left(\left(\frac{R'}{R}\right)^2-\left(\frac{S'}{\sqrt{R}}\right)^2\right)^{2k-1}}\,
,
\end{equation}
where $\Pi^\tau_{\pm,k}$ is a homogeneous polynomial of degree
$4k-3$.
After multiplying by $R^k$, summing over all possible $\tau\in
\mathcal{T}^{\mathbf{0}}_{\epsilon}(k)$, and applying
\eqref{eq:62} and the remark leading to \eqref{eq:2100}, we finally obtain that
\begin{equation}
  \label{eq:73}
  \bar{M}_{\pm,h}=R^{-1/2}\frac{\Pi_{\pm,h}\left(\frac{R^{(k)}}{R},\frac{S^{(k)}}{\sqrt{R}},1\leq
    k\leq
  h+1\right)}{\left(\left(\frac{R'}{R}\right)^2-\left(\frac{S'}{\sqrt{R}}\right)^2\right)^{2h+1}}\,
,
\end{equation}
where $\Pi_{\pm,h}$ are homogeneous polynomials of degree $4h+1$. 
Coming back to the ``dimensionless'' quantities
$R^{h/2}M_{\pm,h}=\bar{M}_{\pm,h}/(\sqrt{R}M_{\pm,0})$, those are,
recalling \eqref{eq:39}, 
\begin{equation}
  \label{eq:2500}
  R^{h/2}M_{\pm,h}=\left(\frac{R'}{R}\pm\frac{S'}{\sqrt{R}}\right)
  \frac{\Pi_{\pm,h}\left(\frac{R^{(k)}}{R},\frac{S^{(k)}}{\sqrt{R}},1\leq
  k\leq
  h+1\right)}{\left(\left(\frac{R'}{R}\right)^2-\left(\frac{S'}{\sqrt{R}}\right)^2\right)^{2h+1}}\,
.
\end{equation}
Combining this with Theorem \ref{sec:high-moments-high}, this
concludes the proof that Proposition~\ref{prop:quasipolnonbip} holds
for $g \geq 2$ and $n=0$.

\section{Higher order differentials of inverse functions of several variables}
\label{sec:higherdiff}

Here we provide a general discussion revisiting
Lagrange inversion that is used in Section \ref{sec:closer-look-at}.
The results are probably known in some other form, the paper by
Warren Johnson cited in \cite{budd2020irreducible} pointing as far as Sylvester.

Fix $n\geq 1$ and let $\bx=(x_1,\ldots,x_n)$, $\by=(y_1,\ldots,y_n)$
be formal variables. We let 
$\mathbf{f}(\bx)=(f_1(\bx),\ldots,f_n(\bx))$ and
$\mathbf{g}(\by)=(g_1(\by),\ldots,g_n(\by))$ be two families of
$n$ elements in $K[[\bx]]$ and $K[[\by]]$ respectively, where $K$ is
some given field. We assume that $\mathbf{f}$ and $\mathbf{g}$ are
compositional inverses, that is, $\mathbf{g}(\mathbf{f}(\bx))=\bx$ and
$\mathbf{f}(\mathbf{g}(\by))=\by$. By classical consideration, given
$\mathbf{f}$, the compositional inverse $\mathbf{g}$ exists if and
only if $\mathbf{f}(\mathbf{0})=\mathbf{0}$ and the differential of $\mathbf{f}$ at $\mathbf{0}$ is invertible, which
we assume. 

Let us write the expansion
\begin{equation}
  \label{eq:49}
  \mathbf{f}(\bx)=a_1\bx-a_2(\bx,\bx)-a_3(\bx,\bx,\bx)-\cdots\, ,
\end{equation}
where $a_1,-a_2,-a_3,\ldots$ are, up to a multiplicative coefficient, the successive differentials of $\mathbf{f}$ at
$\mathbf{0}$, which are respectively $k$-linear symmetric. Using
tensor notations, we explicitly have
\begin{equation}
    \label{eq:46}
(a_1)_i^j=\frac{\partial f_i}{\partial x_j}(\mathbf{0})\, ,\qquad 
    (a_k)_{i}^{j_1,\ldots,j_k}=
   -\frac{1}{k!}
  \frac{\partial^kf_i}{\partial x_{j_1}\ldots\partial
    x_{j_k}}(\mathbf{0})\, ,\quad k\geq 2\, .
\end{equation}
This yields, after substituting $\mathbf{g}(\by)$ to $\bx$,
\begin{equation}
  \label{eq:48}
  \mathbf{g}(\by)=a_1^{-1}\by+a_1^{-1}a_2(\mathbf{g}(\by),\mathbf{g}(\by))+a_1^{-1}a_3(\mathbf{g}(\by),
  \mathbf{g}(\by), \mathbf{g}(\by))+\cdots\, .
\end{equation}
This former expression has a natural tree interpretation, illustrated
in Figure \ref{fig:derivtree}.
\begin{figure}[!]
  \centering
  \includegraphics[scale=.9]{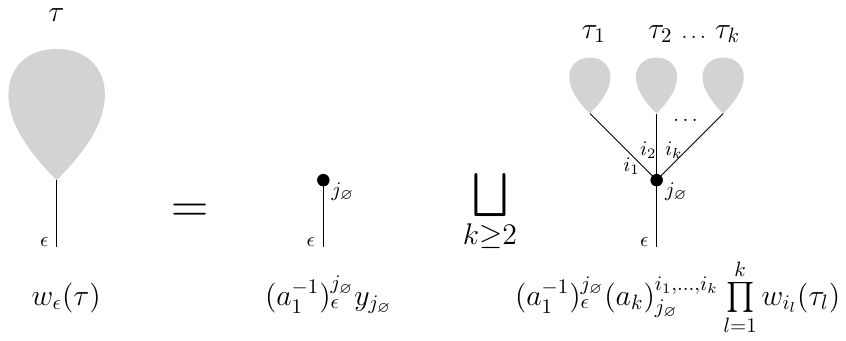}
  \caption{An illustration of the tree structure of derivatives of
    inverses.}
  \label{fig:derivtree}
\end{figure}
We use the
Ulam-Harris notation, where a tree $\tau$ is a subset of the set of
integer words, rooted at the empty word $\varnothing$. Every vertex
$v\in \tau$ has an arity (number of children) denoted by $k_v(\tau)$,
and these children are the words $v1,v2,\ldots,vk_v(\tau)$.  
The trees
considered are planted P\'olya trees (with an extra edge pointing to
the root vertex $\varnothing$), in which the arity $k_v(\tau)$ of every
vertex $v\in \tau$ is an element of $\{0,2,3,4,\ldots\}$. We let
$\mathcal{T}$ be the family of such trees in which the bottom and top half-edges of the
parent edge of $v$ is decorated with two elements $i_v,j_v\in
\{1,\ldots,n\}$. For a given $\epsilon\in \{1,\ldots,n\}$, we also let
$\mathcal{T}_\epsilon$ be the set of those trees for which
$i_\varnothing=\epsilon$. Then
\begin{equation}
  \label{eq:2000}
  g_\epsilon(\by)=\sum_{\tau\in \mathcal{T}_\epsilon}\prod_{v\in
  \tau}(a_1^{-1})_{i_v}^{j_v}(a_{k_v(\tau)})_{j_v}^{i_{v1},i_{v2},\ldots,i_{vk_v(\tau)}}\,
,
\end{equation}
with the convention that $(a_0)_j^\varnothing=y_j$, that is, the
contributions of the variable $\by$ come from the leaves of the
trees. The partial derivatives of order $k$ of $g_\epsilon$ come from
the finite family $\mathcal{T}_\epsilon(k)$ of elements of $\mathcal{T}_\epsilon$ with $k$
leaves. More precisely, for a given
composition $\mathbf{c}=(c_1,\ldots,c_n)$ of $k$ with $n$ parts, one
has
\begin{equation}
  \label{eq:47}
  \frac{\partial^kg_{\epsilon}}{\partial x_1^{c_1}\cdots\partial x_n^{c_n}}(\mathbf{0})=\sum_{\tau\in \mathcal{T}^{\mathbf{c}}_\epsilon}\prod_{v\in
  \tau}(a_1^{-1})_{i_v}^{j_v}\prod_{v\in \tau:k_v(\tau)>0}(a_{k_v(\tau)})_{j_v}^{i_{v1},i_{v2},\ldots,i_{vk_v(\tau)}}\,
,
\end{equation}
where $\mathcal{T}^{\mathbf{c}}_\epsilon$ is the set of trees $\tau\in
\mathcal{T}_\epsilon(k)$ with $c_i$ leaves $v$ labeled by $j_v=i$, for
$1\leq i\leq n$. 

By applying this reasoning to the series 
$\mathbf{f}(\bx+\bx')-\mathbf{f}(\bx)$ and to $\mathbf{g}(\by+\by')-\mathbf{g}(\by)$
instead, we obtain
\begin{equation}
  \label{eq:2400}
  \frac{\partial^kg_{\epsilon}}{\partial x_1^{c_1}\cdots\partial x_n^{c_n}}(\mathbf{f}(\bx))=\sum_{\tau\in \mathcal{T}^{\mathbf{c}}_\epsilon(k)}\prod_{v\in
  \tau}(d_\bx\mathbf{f}^{-1})_{i_v}^{j_v}\prod_{v\in
  \tau:k_v(\tau)>0}\frac{-1}{k_v(\tau)!} \frac{\partial^{k_v(\tau)}
  f_{j_v}}{\partial x_{i_{v1}}\cdots \partial x_{i_{vk_v(\tau)}}}(\bx)\,
.
\end{equation}

\printbibliography

\end{document}